\documentclass[12pt,a4paper,english]{article}
\usepackage[T1]{fontenc}
\usepackage[latin9]{inputenc}
\usepackage{xcolor}
\definecolor{shadecolor}{rgb}{1, 0, 0}
\usepackage{babel}
\usepackage{float}
\usepackage{calc}
\usepackage{framed}
\usepackage{amscd}
\usepackage{amsmath}
\usepackage{amsthm}
\usepackage{amssymb}
\usepackage{graphicx}
\usepackage[unicode=true,pdfusetitle,
 bookmarks=true,bookmarksnumbered=false,bookmarksopen=false,
 breaklinks=false,pdfborder={0 0 1},backref=false,colorlinks=false]
 {hyperref}

\makeatletter

\pdfpageheight\paperheight
\pdfpagewidth\paperwidth

\newcommand{\noun}[1]{\textsc{#1}}

  \theoremstyle{remark}
  \newtheorem*{acknowledgement*}{\protect\acknowledgementname}
\theoremstyle{plain}
\newtheorem{thm}{\protect\theoremname}[section]
  \theoremstyle{definition}
  \newtheorem{defn}[thm]{\protect\definitionname}
  \theoremstyle{remark}
  \newtheorem{rem}[thm]{\protect\remarkname}
  \theoremstyle{plain}
  \newtheorem{prop}[thm]{\protect\propositionname}
  \theoremstyle{plain}
  \newtheorem{cor}[thm]{\protect\corollaryname}
  \theoremstyle{plain}
  \newtheorem{lem}[thm]{\protect\lemmaname}
  \theoremstyle{plain}
  \newtheorem{assumption}[thm]{\protect\assumptionname}
  \theoremstyle{plain}
  \newtheorem{conjecture}[thm]{\protect\conjecturename}

\usepackage{a4wide}
\usepackage{color}
\usepackage{graphics}

\usepackage{amsmath, amsthm, amssymb}

\makeatother

  \providecommand{\acknowledgementname}{Acknowledgement}
  \providecommand{\assumptionname}{Assumption}
  \providecommand{\conjecturename}{Conjecture}
  \providecommand{\corollaryname}{Corollary}
  \providecommand{\definitionname}{Definition}
  \providecommand{\lemmaname}{Lemma}
  \providecommand{\propositionname}{Proposition}
  \providecommand{\remarkname}{Remark}
\providecommand{\theoremname}{Theorem}

\begin{document}

\title{Global normal form and asymptotic spectral gap for open partially
expanding maps}

\date{10 august 2017}

\author{\href{https://www-fourier.ujf-grenoble.fr/~faure/}{Frédéric Faure}\textit{}\thanks{Institut Fourier, UMR 5582, 100 rue des Maths, BP74 38402 St Martin
d'Hères, Université Grenoble Alpes. {\small{}\protect\href{mailto:frederic.faure@univ-grenoble-alpes.fr}{frederic.faure@univ-grenoble-alpes.fr}}}\textit{,} \href{https://www2.math.uni-paderborn.de/people/tobias-weich.html}{Tobias Weich}\thanks{Institut für Mathematik, Universität Paderborn, Warburger Str. 100,
33098 Paderborn. \protect\href{mailto:weich@math.upb.de}{weich@math.upb.de}}}
\maketitle
\begin{abstract}
We consider a \noun{$\mathbb{R}$}-extension of one dimensional uniformly
expanding open dynamical systems and prove a new explicit estimate
for the asymptotic spectral gap. To get these results, we use a new
application of a ``\textbf{global normal form}'' for the dynamical
system, a ``\textbf{semiclassical expression beyond the Ehrenfest
time}'' that expresses the transfer operator at large time as a sum
over rank one operators (each is associated to one orbit). In this
paper we establish the validity of the so-called ``\textbf{diagonal
approximation'' up to twice the local Ehrenfest time}.
\end{abstract}
\footnote{2010 Mathematics Subject Classification:37D20 Uniformly hyperbolic
systems (expanding, Anosov, Axiom A, etc.) 37D35 Thermodynamic formalism,
variational principles, equilibrium states 37C30 Zeta functions, (Ruelle-Frobenius)
transfer operators, and other functional analytic techniques in dynamical
systems 81Q20 Semiclassical techniques, including WKB and Maslov methods
81Q50 Quantum chaos

Keywords: Transfer operator; Ruelle resonances; decay of correlations;
Semi-classical analysis. }

\tableofcontents{}

\section{Introduction}

In this paper we consider a \noun{$\mathbb{R}$}-extension of one
dimensional uniformly expanding open dynamical systems, so called
iterated function systems (IFS) . The dynamical properties of these
IFS are on the one hand interesting, because of relations to the spectral
theory on Riemann surfaces and questions in number theory. On the
other hand the $\mathbb{R}$-extension adds a neutral direction to
the dynamics and our model can also be considered as a toy model for
more complicated dynamical systems such as Anosov or Axiom A flows
\cite{katok_hasselblatt}. The main object of study in this paper
is the asymptotic spectral gap for the family of transfer operators
associated to these specific open partially expanding maps. In appendix
\ref{sec:Discussion-about-} we propose a discussion for motivating
the study of the ``asymptotic spectral gap'' $\gamma_{\mathrm{asympt.}}$
from the general point of view of hyperbolic flows (in both classical
and quantum mechanics). For related models we review known results
on $\gamma_{\mathrm{asympt.}}$ and we also discuss the conjecture
for $\gamma_{\mathrm{asympt.}}$ that generically $\gamma_{\mathrm{asympt.}}=\gamma_{conj}:=\frac{1}{2}\mathrm{Pr}\left(2\left(V-J\right)\right)$.
This appendix may be consulted first by readers who are interested
by more detailed motivations. It is however not mandatory for understanding
the main results.

In Section \ref{sec:The-model} we define the model under study. The
transfer operator $\mathcal{L}_{\nu}$ acting on functions $u$ is
$\mathcal{L}_{\nu}u:=e^{i\nu\tau+V}u\circ\phi^{-1}$ depending on
a parameter $\nu\in\mathbb{R}$, smooth functions $\tau,V$ and $\phi^{-1}$
an expanding map on intervals.

In Section \ref{sec:The-main-result} we define the ``asymptotic
spectral gap'' $\gamma_{\mathrm{asympt.}}:=\limsup_{\nu\rightarrow\infty}\log\left(r_{s}\left(\mathcal{L}_{\nu}\right)\right)$
(where $r_{s}\left(\mathcal{L}_{\nu}\right)$ stands for the spectral
radius) and give the main results of this paper: in Theorem \ref{thm:main_result}
we show that $\gamma_{\mathrm{asympt.}}\leq\gamma_{\mathrm{up}}:=\frac{1}{2}\mathrm{Pr}\left(2\left(V-J\right)\right)+\frac{1}{4}\left\langle J\right\rangle $
where $\mathrm{Pr}\left(.\right)$ is the topological pressure, $J=\log\left|\left(\phi^{-1}\right)'\right|>0$
is the expansion rate, $\left\langle J\right\rangle $ is an averaged
expansion rate given in Eq.(\ref{eq:<J>}). In Theorem \ref{thm:main_result-resolvent}
we also get an upper bound for the norm of the resolvent of the transfer
operator. We discuss their consequences in terms of decay of correlations.
In Section \ref{subsec:Other-interesting-results:} we discuss other
interesting results obtained in this paper which may be extended to
more general hyperbolic dynamics: a global normal form and an asymptotic
expansion of the transfer operator. In Section \ref{subsec:Sketch-of-the}
we provide a (very short) sketch of proof of the main results.

From Section \ref{sec:The-canonical-map} to \ref{subsec:Proof-of-main_Theorem}
we provide the proof of the main Theorems and develop tools for this.

Under non local integrability (NLI) hypothesis it has been shown by
D. Dolgopyat \cite{dolgopyat_02} that $\exists\epsilon>0,$ $\gamma_{\mathrm{asympt.}}\leq\gamma_{\mathrm{Gibbs}}-\epsilon$
with $\gamma_{\mathrm{Gibbs}}=\mathrm{Pr}\left(V-J\right)$. Using
semiclassical analysis and some hypothesis it is also known \cite{faure_arnoldi_tobias_13}
that $\gamma_{\mathrm{asympt.}}\leq\gamma_{\mathrm{sc}}=\mathrm{tsup}\left(V-\frac{1}{2}J\right)$
where $\mathrm{tsup}$ means supremum after time average (see (\ref{eq:def_gamma_sc})). 

In Appendix \ref{sec:Examples} we consider examples based on linear
maps and the Gauss map and compare our bounds with numerical results
for the Ruelle spectrum. We also show that the new bound $\gamma_{\mathrm{up}}$
improves the previous bounds $\gamma_{\mathrm{Gibbs}}$ and $\gamma_{\mathrm{sc}}$
in some range of parameters. On the web site of the first author \cite{faure_animations_expanding_maps}
we propose movies and additional multimedia contents that illustrate
these models.
\begin{acknowledgement*}
We would like to thank Masato Tsujii, Michiro Hirayama for discussions
and Mark Pollicott and Richard Sharp for discussions and for the explanation
of the formula in Appendix \ref{sec:A-formula-by} which can be derived
by their previous work. We thank the referees for very precise reading
and valuable comments. This work has been supported by ANR-13-BS01-0007-01.
T.W. acknowledges financial support by DFG HI 412 12-1.
\end{acknowledgement*}

\section{The model\label{sec:The-model}}

In this Section we introduce the model which we study in this paper.
This model has already been studied\footnote{Compared to the previous paper \cite{faure_arnoldi_tobias_13}, we
have changed the notation of the transfer operator from $\hat{F}$
to $\mathcal{L}$ and of its associated symplectic map from $F$ to
$\tilde{\phi}$. We have also replaced $\hbar$ by $\nu=1/\hbar$. } in \cite{faure_arnoldi_tobias_13} and we refer to this paper for
more comments, examples or details.

\subsection{Iterated function system}

See Figure \ref{fig:op_F_hat} for an illustration.

\begin{center}{\color{red}\fbox{\color{black}\parbox{16cm}{
\begin{defn}
\label{def:IFSAn-iterated-function}``\textbf{An iterated function
system (I.F.S.)''. }Let $I_{1},\ldots I_{N}\subset\mathbb{R}$ be
a finite collection of \textbf{disjoint bounded and closed} intervals
with $N\geq1$. Let $A\in\left\{ 0,1\right\} ^{N\times N}$ called
an adjacency matrix and assume that the matrix $A$ is primitive,
i.e. there is $T\geq0$ such that $\forall i,j,\left(A^{T}\right)_{i,j}>0$.
We will note $i\rightsquigarrow j$ if $A_{i,j}=1.$ Assume that for
each pair $i,j\in\left\{ 1,\ldots,N\right\} $ such that $i\rightsquigarrow j$,
we have a smooth invertible map $\phi_{i,j}:I_{i}\rightarrow\phi_{i,j}\left(I_{i}\right)\subset\mbox{Int}\left(I_{j}\right)$.
Assume that the map $\phi_{i,j}$ is a \textbf{strict contraction},
i.e. there exists $0<\theta<1$ such that for every $x\in I_{i}$,
\begin{equation}
0<\phi_{i,j}'\left(x\right)\leq\theta.\label{eq:theta_contraction}
\end{equation}
We suppose that different images of the maps $\phi_{i,j}$ do not
intersect (this is the ``strong separation condition'' in \cite[p.35]{Falconer_97}):
\begin{equation}
\left(i,j\right)\neq\left(k,l\right)\quad\Rightarrow\quad\phi_{i,j}\left(I_{i}\right)\cap\phi_{k,l}\left(I_{k}\right)=\emptyset.\label{eq:hyp_non_intersect}
\end{equation}
\end{defn}

}}}\end{center}
\begin{rem}
We have assumed for simplicity that the map $\phi_{i,j}$ preserves
orientation, i.e. $0<\phi_{i,j}'\left(x\right)\leq\theta$. The results
of this paper also hold if we only suppose that $0<\left|\phi_{i,j}'\left(x\right)\right|\leq\theta$.
To treat this case, we can define $\sigma_{i,j}=\mathrm{sign}\left(\phi_{i,j}'\right)\in\left\{ -1,1\right\} $
and replace in every formula of this paper, the term $e^{J_{i,j}\left(x\right)}$
by $\sigma_{i,j}e^{J_{i,j}\left(x\right)}$. For example the truncated
Gauss model presented in Section \ref{subsec:Truncated-Gauss-map}
has negative derivatives $\phi_{i,j}'\left(x\right)<0$.
\end{rem}

\subsection{The trapped set $K$}

We define
\begin{equation}
I:=\bigcup_{i=1}^{N}I_{i}.\label{eq:def_I}
\end{equation}
The multivalued map: 
\[
\phi:I\rightarrow I,\qquad\phi:=\left(\phi_{i,j}\right)_{i,j}
\]
can be iterated and generates a multivalued map $\phi^{n}:I\rightarrow I$
for $n\geq1$. From Condition (\ref{eq:hyp_non_intersect}) the inverse
map
\[
\phi^{-1}:\phi\left(I\right)\rightarrow I
\]
is uni-valued. Let 
\begin{equation}
K_{n}:=\phi^{n}\left(I\right)\label{eq:def_Kn}
\end{equation}
and $K_{0}=I$. We have $K_{n+1}\subset K_{n}$ so we can define the
limit set
\begin{equation}
K:=\bigcap_{n\in\mathbb{N}}K_{n}\label{eq:trapped_set_K_def1}
\end{equation}
called the \textbf{trapped set}. The map 
\begin{equation}
\phi^{-1}:K\rightarrow K\label{eq:phi_-1_on_K}
\end{equation}
 is well defined and uni-valued.

\subsection{The transfer operator $\mathcal{L}$}

\paragraph{Notations:}

We denote $C_{0}^{\infty}\left(\mathbb{R}\right)$ the space of smooth
functions on $\mathbb{R}$ with compact support. If $B\subset\mathbb{R}$
is a finite union of closed intervals, we denote by $C_{0}^{\infty}\left(B\right)\subset C_{0}^{\infty}\left(\mathbb{R}\right)$
the space of smooth functions on $\mathbb{R}$ with support included
in $B$. We denote by $C^{\infty}\left(B;\mathbb{R}\right)$ and $C^{\infty}\left(B;\mathbb{C}\right)$
the space of real (respect. complex) valued smooth functions on $B$.

\begin{center}{\color{red}\fbox{\color{black}\parbox{16cm}{
\begin{defn}
\label{def_Transfer_op}Let $\tau\in C^{\infty}\left(\phi\left(I\right);\mathbb{R}\right)$
and $V\in C^{\infty}\left(\phi\left(I\right);\mathbb{R}\right)$ be
smooth functions called respectively \textbf{roof function} and \textbf{potential}
function. Let $\nu>0$. We define the \textbf{transfer operator}:
\begin{equation}
\mathcal{L}_{\nu}:\begin{cases}
C_{0}^{\infty}\left(I\right) & \rightarrow C_{0}^{\infty}\left(I\right)\\
\varphi=\left(\varphi_{i}\right)_{i} & \rightarrow\left(\sum_{i=1}^{N}\mathcal{L}_{i,j}\varphi_{i}\right)_{j}
\end{cases}\label{eq:def_transfert_op_F}
\end{equation}
 with 
\begin{equation}
\mathcal{L}_{i,j}:\begin{cases}
C_{0}^{\infty}\left(I_{i}\right) & \rightarrow C_{0}^{\infty}\left(I_{j}\right)\\
\varphi_{i} & \rightarrow\left(\mathcal{L}_{i,j}\varphi_{i}\right)\left(x\right)=\begin{cases}
e^{i\nu\tau\left(x\right)+V\left(x\right)}\varphi_{i}\left(\phi_{i,j}^{-1}\left(x\right)\right) & \quad\mbox{if }i\rightsquigarrow j\mbox{ and }x\in\phi_{i,j}\left(I_{i}\right)\\
0 & \quad\mbox{otherwise.}
\end{cases}
\end{cases}\label{e:def_F_op_ij}
\end{equation}
\end{defn}

}}}\end{center}

See Figure \ref{fig:op_F_hat}.
\begin{rem}
\label{rem:2.4}~
\end{rem}

\begin{enumerate}
\item Eq.(\ref{eq:def_transfert_op_F}) is a family of transfer operators
depending on the parameter $\nu\in\mathbb{R}$. We will be interested
in the spectrum of these operators in the ``semiclassical limit''
$\nu\rightarrow+\infty$.
\item From assumption (\ref{eq:hyp_non_intersect}), for any $x\in I$,
the sum $\sum_{i=1}^{N}\left(\mathcal{L}_{i,j}\varphi_{i}\right)\left(x\right)$
which appears on the right hand side of (\ref{eq:def_transfert_op_F})
contains at most one non vanishing term.
\item For any $\varphi\in C_{0}^{\infty}\left(I\right)$, $n\geq0$ we have
\begin{equation}
\mbox{supp}\left(\mathcal{L}_{\nu}^{n}\varphi\right)\subset K_{n}\label{eq:supp_F_phi}
\end{equation}
 with $K_{n}$ defined in (\ref{eq:def_Kn}).
\item The family of operators $\left(\mathcal{L}_{\nu}\right)_{\nu\in\mathbb{R}}$
can naturally be obtained from a dynamical system (\ref{eq:def_extended_map_f})
that is a $\mathbb{R}$-extension of the IFS and take the Fourier
component with frequency $\nu$ in the neutral direction (see Section
\ref{subsec:Partially-expanding-maps} or \cite[Sec.2.2]{faure_arnoldi_tobias_13}
for a detailed explanation). The limit $\nu\rightarrow+\infty$ corresponds
to the limit of high Fourier modes. In this sense studying the spectral
properties of the whole family of operators $\left(\mathcal{L}_{\nu}\right)_{\nu}$
corresponds to studying the spectral properties of this $\mathbb{R}$-extension
of the IFS, i.e. a dynamical system with a neutral direction.
\end{enumerate}
\begin{figure}[h]
\begin{centering}
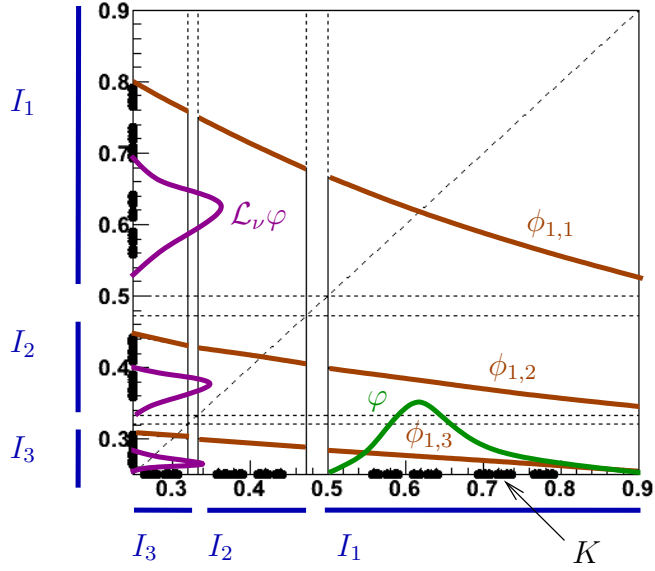
\par\end{centering}
\caption{\label{fig:op_F_hat}Action of the transfer operator $\mathcal{L}_{\nu}$
on a function $\varphi$ as defined in (\ref{eq:def_transfert_op_F})
for the dynamics of truncated Gauss map with $N=3$ intervals $\left(I_{j}\right)_{j=1\ldots N}$,
defined in Section \ref{subsec:Truncated-Gauss-map}. In this picture
$\varphi$ is supported on $I_{1}$ and $\mathcal{L}_{\nu}\varphi$
is supported on the three intervals $I_{1}\cup I_{2}\cup I_{3}$.
The maps $\phi$: $\phi_{i,j}:I_{i}\rightarrow I_{j}$, $i,j=1\ldots N$
are contracting and given by $\phi_{i,j}\left(x\right)=\frac{1}{x+j}$.
The trapped set $K$ defined in (\ref{eq:trapped_set_K_def1}) is
a $N$-adic Cantor set. It is obtained as the limit of the sets $K_{0}=\left(I_{1}\cup I_{2}\ldots\cup I_{N}\right)\supset K_{1}=\phi\left(K_{0}\right)\supset K_{2}=\phi\left(K_{1}\right)\supset\ldots\supset K$.
In this example, $\mathrm{dim}_{H}\left(K\right)=0.705\ldots$. In
this schematic figure we have $\tau=0,V=0$. In general the factor
$e^{V\left(x\right)}$ changes the amplitude of $\mathcal{L}_{\nu}\varphi$
and $e^{i\nu\tau\left(x\right)}$ creates some fast oscillations if
$\nu\gg1$.}
\end{figure}

\subsubsection{\label{subsec:Extension}Extension of the transfer operator to distributions}

In \cite[Sec.3.1]{faure_arnoldi_tobias_13} it is explained how the
transfer operator $\mathcal{L}$, initially defined on smooth functions
$C_{0}^{\infty}\left(I\right)$, can be extended to the space of distributions.
For completeness we recall this construction. We first introduce a
cut-off function $\chi\in C_{0}^{\infty}\left(K_{a}\right)$ such
that $0<\chi\left(x\right)$ for $x\in\mbox{Int}\left(K_{a}\right)$
where $a\in\mathbb{N}$ and $K_{a}$ is defined in (\ref{eq:def_Kn})
and $\chi\left(x\right)=0$ for $x\in\partial K_{a}$. Let us remark
that in the proof of Lemma \ref{lem:Separation-of-xw_zetaw}, we will
need to fix the value of $a$ according to (\ref{eq:Ka_w}). We denote
$\hat{\chi}$ the multiplication operator by the function $\chi$.
We define\footnote{The conjugation by $\chi$ is necessary to extend the operators to
distributions and thus, later in Section \ref{subsec:Escape-function},
to Sobolev spaces. This is due to the fact that the dynamics is ``open''.
The spectral properties are, however, independent of the choice of
$\chi$ (cf. Theorem \ref{thm:main_result}).}
\begin{equation}
\mathcal{L}_{i,j,\chi}:=\hat{\chi}^{-1}\mathcal{L}_{i,j}\hat{\chi}:C_{0}^{\infty}\left(\mathbb{R}\right)\rightarrow C_{0}^{\infty}\left(I_{j}\right),\qquad\mathcal{L}_{\nu,\chi}:=\hat{\chi}^{-1}\mathcal{L}_{\nu}\hat{\chi}\label{eq:def_F_Chi_Chim-1}
\end{equation}
which is well defined since $\mbox{supp}\left(\mathcal{L}_{i,j}\hat{\chi}\varphi\right)\subset\mbox{Int}\left(K_{a}\right)$
where $\chi$ does not vanish, although $\hat{\chi}^{-1}$ is not
defined by itself. The formal $L^{2}$-adjoint operator $\mathcal{L}_{i,j,\chi}^{*}:C_{0}^{\infty}\left(\mathbb{R}\right)\rightarrow C_{0}^{\infty}\left(I_{i}\right)$
is defined by 
\begin{equation}
\langle\varphi_{i},\mathcal{L}_{i,j,\chi}^{*}\psi_{j}\rangle_{L^{2}}=\langle\mathcal{L}_{i,j,\chi}\varphi_{i},\psi_{j}\rangle_{L^{2}},\qquad\forall\varphi_{i}\in C_{0}^{\infty}\left(\mathbb{R}\right),\psi_{j}\in C_{0}^{\infty}\left(\mathbb{R}\right),\label{eq:def_F*}
\end{equation}
with the $L^{2}$-scalar product\footnote{We will omit the index $L^{2}$ sometimes.}
\begin{equation}
\langle u,v\rangle_{L^{2}}:=\int\overline{u}\left(x\right)v\left(x\right)dx.\label{eq:scalar_prodcut_L2}
\end{equation}
The $L^{2}$-adjoint operator $\mathcal{L}_{\nu,\chi}^{*}:C_{0}^{\infty}\left(\mathbb{R}\right)\rightarrow C_{0}^{\infty}\left(I\right)$
is defined by
\[
\psi=\left(\psi_{j}\right)_{j}\rightarrow\left(\mathcal{L}_{\nu,\chi}^{*}\psi\right)_{i}\left(y\right)=\sum_{j\mbox{ s.t. }i\rightsquigarrow j}\left(\mathcal{L}_{i,j,\chi}^{*}\psi_{j}\right)\left(y\right)
\]
whose components are given by \cite[Lemma 3.1]{faure_arnoldi_tobias_13}
\begin{equation}
\left(\mathcal{L}_{i,j,\chi}^{*}\psi_{j}\right)\left(y\right)=\frac{\chi\left(y\right)}{\chi\left(\phi_{i,j}\left(y\right)\right)}\left|\phi'_{i,j}\left(y\right)\right|e^{V\left(\phi_{i,j}\left(y\right)\right)}e^{-i\nu\tau\left(\phi_{i,j}\left(y\right)\right)}\psi_{j}\left(\phi_{i,j}\left(y\right)\right).\label{eq:express_dual_1}
\end{equation}

\begin{center}{\color{blue}\fbox{\color{black}\parbox{16cm}{
\begin{prop}
\cite[Sec.3.2]{faure_arnoldi_tobias_13} By duality the transfer operators
$\mathcal{L}_{\nu,\chi}$ and $\mathcal{L}_{\nu,\chi}^{*}$ extend
to distributions:
\begin{equation}
\mathcal{L}_{\nu,\chi}:\mathcal{D}'\left(\mathbb{R}\right)\rightarrow\mathcal{D}'\left(\mathbb{R}\right)\label{eq:F_Chi_Distributions}
\end{equation}
\[
\mathcal{L}_{\nu,\chi}^{*}:\mathcal{D}'\left(\mathbb{R}\right)\rightarrow\mathcal{D}'\left(\mathbb{R}\right)
\]
\end{prop}

}}}\end{center}

\subsection{\label{subsec:Escape-function}Escape function}

In this section we want to introduce Hilbert spaces in which the transfer
operator has discrete spectrum. We therefore consider the following
``escape\footnote{The name ``escape function''  will be justified by Remark \ref{rem:escape}
which shows that $A_{m}$ decays along the dynamics of $\tilde{\phi}$
for $\left|\xi\right|\geq C$.} function'' on the cotangent space $T^{*}\mathbb{R}=\mathbb{R}^{2}$
with coordinates $\left(x,\xi\right)$. Let $m>0$ and let $A_{m}\in C^{\infty}\left(\mathbb{R}^{2};\mathbb{R}\right)$
be the ``symbol'' given by\footnote{in fact $A_{m}\left(x,\xi\right)$ is independent on $x$.}
\begin{equation}
A_{m}\left(x,\xi\right):=\left\langle \xi\right\rangle ^{-m}\label{eq:def_A}
\end{equation}
with $\left\langle \xi\right\rangle :=\left(1+\xi^{2}\right)^{1/2}$.
We will use the $L^{2}$-unitary $\nu-$Fourier transform $\mathcal{F}_{\nu}:L^{2}\left(\mathbb{R}_{x}\right)\rightarrow L^{2}\left(\mathbb{R}_{\xi}\right)$
and its inverse:
\begin{equation}
\left(\mathcal{F}_{\nu}\varphi\right)\left(\xi\right):=\frac{1}{\sqrt{2\pi/\nu}}\int_{\mathbb{R}}e^{-i\nu\xi.x}\varphi\left(x\right)dx,\quad\left(\mathcal{F}_{\nu}^{-1}\psi\right)\left(x\right):=\frac{1}{\sqrt{2\pi/\nu}}\int_{\mathbb{R}}e^{i\nu\xi.x}\psi\left(\xi\right)d\xi.\label{eq:FT}
\end{equation}
Notice that the parameter $\nu$ is just a scaling in $\xi$. Let
$\hat{A}_{m}:=\mathrm{Op}_{\nu}\left(A_{m}\right):\mathcal{S}\left(\mathbb{R}\right)\rightarrow\mathcal{S}\left(\mathbb{R}\right)$,
be the linear operator defined as the semiclassical quantization of
$A_{m}$ \cite{zworski_book_2012}. In this case this is simple. For
$\varphi\in\mathcal{S}\left(\mathbb{R}\right)$, 
\begin{eqnarray}
\left(\hat{A}_{m}\varphi\right)\left(x\right): & = & \frac{1}{2\pi/\nu}\int A_{m}\left(x,\xi\right)e^{i\nu\left(x-y\right)\xi}\varphi\left(y\right)dyd\xi\label{eq:quantiz_rule_n}\\
 & = & \mathcal{F}_{\nu}^{-1}\left(\left\langle \xi\right\rangle ^{-m}\left(\mathcal{F}_{\nu}\varphi\right)\right)\left(x\right)\nonumber 
\end{eqnarray}
where in the last line $\left\langle \xi\right\rangle ^{-m}$ denotes
the multiplication operator. By duality $\hat{A}_{m}$ is extended
to\footnote{$\mathcal{S}'\left(\mathbb{R}\right)$ is the space of tempered distributions,
see \cite[p.204]{taylor_tome1}.} $\hat{A}_{m}:\mathcal{S}'\left(\mathbb{R}\right)\rightarrow\mathcal{S}'\left(\mathbb{R}\right)$.
For $m\in\mathbb{R}$, the $\nu$-Sobolev space of order $m$ is defined
by 
\begin{equation}
H_{\nu}^{-m}\left(\mathbb{R}\right):=\hat{A}_{m}^{-1}\left(L^{2}\left(\mathbb{R}\right)\right)\label{eq:def_Hm}
\end{equation}
and the norm of $\varphi\in H_{\nu}^{-m}\left(\mathbb{R}\right)$
is defined by
\begin{equation}
\left\Vert \varphi\right\Vert _{H_{\nu}^{-m}\left(\mathbb{R}\right)}:=\left\Vert \hat{A}_{m}\varphi\right\Vert _{L^{2}\left(\mathbb{R}\right)}.\label{eq:def_norm_Hm}
\end{equation}

Let
\[
\hat{Q}_{i,j}:=\hat{A}_{m}\mathcal{L}_{i,j,\chi}\hat{A}_{m}^{-1}\quad:L^{2}\left(\mathbb{R}\right)\rightarrow L^{2}\left(\mathbb{R}\right).
\]
and
\begin{equation}
\hat{Q}:=\hat{A}_{m}\hat{\chi}^{-1}\mathcal{L}_{\nu}\hat{\chi}\hat{A}_{m}^{-1}=\hat{A}_{m}\mathcal{L}_{\nu,\chi}\hat{A}_{m}^{-1}.\label{eq:def_Q}
\end{equation}
Equivalently we have the following commutative diagram
\begin{equation}
\begin{CD}L^{2}\left(\mathbb{R}\right)@>{\hat{Q}}>>L^{2}\left(\mathbb{R}\right)\\
@V{\hat{A}_{m}^{-1}}VV@V{\hat{A}_{m}^{-1}}VV\\
H_{\nu}^{-m}\left(\mathbb{R}\right)@>{\mathcal{L}_{\nu,\chi}}>>H_{\nu}^{-m}\left(\mathbb{R}\right)
\end{CD}.\label{eq:comm_diag}
\end{equation}

\begin{center}{\color{blue}\fbox{\color{black}\parbox{16cm}{
\begin{thm}
\label{thm:-Discrete-spectrum.}\cite[th.2.6]{faure_arnoldi_tobias_13}
\textbf{``Discrete spectrum''}. For any $r>0$, there is $m_{0}>0$
such that for all $m>m_{0}$ and for all $\nu\in\mathbb{R}$,
\[
\mathcal{L}_{\nu,\chi}:H_{\nu}^{-m}\left(\mathbb{R}\right)\rightarrow H_{\nu}^{-m}\left(\mathbb{R}\right)
\]
has purely discrete spectrum $\left(\lambda_{j}\left(\nu\right)\right)_{j\in\mathbb{N}}$
on the spectral domain $\left\{ \lambda\in\mathbb{C},\left|\lambda\right|>r\right\} $.
The eigenvalues $\left(\lambda_{j}\left(\nu\right)\right)_{j\in\mathbb{N}}$
in this domain are independent on $m$ and $\chi$ and are called
the \textbf{Ruelle-Pollicott resonances} of the transfer operator
$\mathcal{L}_{\nu}$.
\end{thm}

}}}\end{center}
\begin{rem}
From the commutative diagram (\ref{eq:comm_diag}), the spectral properties
of $\hat{Q}:L^{2}\left(\mathbb{R}\right)\rightarrow L^{2}\left(\mathbb{R}\right)$
are equivalent to those of $\mathcal{L}_{\nu,\chi}:H_{\nu}^{-m}\left(\mathbb{R}\right)\rightarrow H_{\nu}^{-m}\left(\mathbb{R}\right)$.
In practice (in the proofs) we will work with $\hat{Q}$ on $L^{2}\left(\mathbb{R}\right)$.
\end{rem}

\section{\label{sec:The-main-result}The main results}

Let $r_{s}\left(\mathcal{L}_{\nu,\chi}\right)=\sup_{j\in\mathbb{N}}\left\{ \left|\lambda_{j}\left(\nu\right)\right|\right\} $
be the spectral radius of the operator $\mathcal{L}_{\nu,\chi}:H_{\nu}^{-m}\left(\mathbb{R}\right)\rightarrow H_{\nu}^{-m}\left(\mathbb{R}\right)$
with $m$ large enough so that $r_{s}\left(\mathcal{L}_{\nu,\chi}\right)$
does not depend on $m$ nor on $\chi$ (for this we need that the
Ruelle spectrum is non empty, otherwise we put $r_{s}\left(\mathcal{L}_{\nu,\chi}\right):=0$).
We are interested in the asymptotic value

\noindent\fcolorbox{red}{white}{\begin{minipage}[t]{1\columnwidth - 2\fboxsep - 2\fboxrule}%
\begin{equation}
\gamma_{\mathrm{asympt.}}:=\limsup_{\nu\rightarrow+\infty}\left(\log\left(r_{s}\left(\mathcal{L}_{\nu,\chi}\right)\right)\right).\label{eq:def_gamma_asympt-1}
\end{equation}
\end{minipage}}

To express the main results below we need to introduce the topological
pressure. It can be defined from the periodic points as follows. A
\textbf{periodic point} of period $n\geq1$ is $x\in K$ such that
$x=\phi^{-n}\left(x\right)$.

\begin{center}{\color{red}\fbox{\color{black}\parbox{16cm}{
\begin{defn}
\cite[p.72]{Falconer_97} The \textbf{topological pressure} of a Lipschitz
function $\varphi\in U\rightarrow\mathbb{R}$ with $U$ a neighborhood
of the trapped set $K$, is
\begin{equation}
\mathrm{Pr}\left(\varphi\right):=\lim_{n\rightarrow\infty}\frac{1}{n}\log\left(\sum_{x=\phi^{-n}\left(x\right)}e^{\varphi_{n}\left(x\right)}\right)\label{eq:def_Pr}
\end{equation}
where
\[
\varphi_{n}\left(x\right):=\sum_{k=0}^{n-1}\varphi\left(\phi^{-k}\left(x\right)\right)
\]
is the Birkhoff sum of $\varphi$ along the periodic orbit.
\end{defn}

}}}\end{center}

We define the \textbf{``Jacobian function''} 
\begin{equation}
J\left(x\right):=\log\frac{d\phi^{-1}}{dx}\left(x\right)>0.\label{eq:def_J}
\end{equation}
and
\begin{align}
J_{max}: & =\mathrm{tsup}\left(J\right):=\lim_{n\rightarrow\infty}\sup_{x\in K}\left(\frac{1}{n}\sum_{k=0}^{n-1}J\left(\phi^{-k}\left(x\right)\right)\right),\label{eq:Jmax}\\
J_{min}: & =\mathrm{tinf}\left(J\right):=\lim_{n\rightarrow\infty}\inf_{x\in K}\left(\frac{1}{n}\sum_{k=0}^{n-1}J\left(\phi^{-k}\left(x\right)\right)\right).\label{eq:Jmin}
\end{align}
\begin{rem}
The limits on the right hand sides of (\ref{eq:Jmax}), (\ref{eq:Jmin})
exist because the sequences 
\[
a_{n}:=\inf_{x\in K}\left(\sum_{k=0}^{n-1}J\left(\phi^{-k}\left(x\right)\right)\right),\quad b_{n}:=\sup_{x\in K}\left(\sum_{k=0}^{n-1}J\left(\phi^{-k}\left(x\right)\right)\right)
\]
are superadditive (i.e. $a_{n}+a_{m}\leq a_{n+m}$) and subadditive
(i.e. $b_{n}+b_{m}\geq b_{n+m}$) respectively and Fekete's Lemma
guaranties existence of the limits $J_{min}=\lim_{n\rightarrow\infty}a_{n}/n$
and $J_{max}=\lim_{n\rightarrow\infty}b_{n}/n$.
\end{rem}

\subsection{\label{subsec:Theorems}Theorems}

\begin{center}{\color{blue}\fbox{\color{black}\parbox{16cm}{
\begin{thm}
\label{thm:main_result}''\textbf{Bound of the spectral radius}''.
Let $\beta>0$ be defined by
\[
\mathrm{Pr}\left(2\left(V-J\right)+\beta J\right)=2\mathrm{Pr}\left(V-J\right)
\]
and
\begin{equation}
\left\langle J\right\rangle :=\frac{2\mathrm{Pr}\left(V-J\right)-\mathrm{Pr}\left(2\left(V-J\right)\right)}{\beta}\in\left[J_{min},J_{max}\right].\label{eq:<J>}
\end{equation}
Under the assumption \ref{hyp:minimal_capt} of minimal captivity
defined below, if $\beta\geq\frac{1}{2}$ and $\left\langle J\right\rangle <2J_{min}$
then
\begin{equation}
\gamma_{\mathrm{asympt.}}\leq\gamma_{\mathrm{up}}:=\frac{1}{2}\mathrm{Pr}\left(2\left(V-J\right)\right)+\frac{1}{4}\left\langle J\right\rangle \label{eq:def_gamma_up}
\end{equation}
otherwise
\begin{equation}
\gamma_{\mathrm{asympt.}}\leq\gamma_{\mathrm{Gibbs}}:=\mathrm{Pr}\left(V-J\right).\label{eq:def_gamma_Gibbs}
\end{equation}
\end{thm}

}}}\end{center}
\begin{rem}
The assumption \ref{hyp:minimal_capt} of minimal captivity will be
explained later but can be summarized as follows. If $\tilde{\phi}$
is the symplectic map on $T^{*}\mathbb{R}$ associated to the transfer
operator $\mathcal{L}_{\nu}$ and $\mathcal{K}\subset T^{*}\mathbb{R}$
is its trapped set (it is a Cantor set) then the ``minimal captivity
assumption'' is that $\tilde{\phi}$ is univalued on a small neighborhood
of $\mathcal{K}$.
\end{rem}

~
\begin{rem}
It is remarkable that the bound $\gamma_{\mathrm{up}}$ in (\ref{eq:def_gamma_up})
does not depend on the roof function $\tau$, however beware that
$\tau$ will appear in the expression of $\tilde{\phi}$ (see (\ref{eq:def_symplectic_map_Fij-1}))
and therefore the assumption \ref{hyp:minimal_capt} needed to get
(\ref{eq:def_gamma_up}) depends on $\tau$.
\end{rem}

\paragraph{Previous known results about $\gamma_{\mathrm{asympt.}}$:}
\begin{itemize}
\item the bound (\ref{eq:def_gamma_Gibbs}) is already well known and holds
without any assumption \cite{ruelle_89}.
\item D. Dolgopyat \cite{dolgopyat_02} has shown under a generic condition
that $\exists\epsilon>0,$ 
\[
\gamma_{\mathrm{asympt.}}\leq\gamma_{\mathrm{Gibbs}}-\epsilon,\qquad\gamma_{\mathrm{Gibbs}}:=\mathrm{Pr}\left(V-J\right).
\]
\item in \cite{faure_arnoldi_tobias_13} it has been shown under the assumption
of ``minimal captivity'' that we have
\begin{equation}
\gamma_{\mathrm{asympt.}}\leq\gamma_{\mathrm{sc}}:=\mathrm{tsup}\left(D\right):=\lim_{n\rightarrow\infty}\sup_{x\in K}\left(\frac{1}{n}\sum_{k=0}^{n-1}D\left(\phi^{-k}\left(x\right)\right)\right),\qquad D:=V-\frac{1}{2}J,\label{eq:def_gamma_sc}
\end{equation}
(here $\gamma_{\mathrm{sc}}$ stands for ``$\gamma_{semi-classical}$''
since we used semiclassical analysis to obtain it and $D=V-\frac{1}{2}J$
is called the ``effective damping function'' from \cite{faure-tsujii_prequantum_maps_12}).
\end{itemize}
The next Theorem gives an upper bound for the norm of the resolvent
of the transfer operator in $H_{\nu}^{-m}\left(\mathbb{R}\right)$
outside the radius $e^{\gamma_{\mathrm{up}}}$. This is useful to
control the asymptotic decay of correlation functions for the corresponding
dynamical system (see Corollary \ref{thm:correl_funct}). For an operator
$O:\mathcal{H}\rightarrow\mathcal{H}$ we will use the notation $\left\Vert O\right\Vert _{\mathcal{H}}:=\left\Vert O\right\Vert _{\mathcal{H}\rightarrow\mathcal{H}}$.

\begin{center}{\color{blue}\fbox{\color{black}\parbox{16cm}{
\begin{thm}
\label{thm:main_result-resolvent}''\textbf{bound of the resolvent''
}With $\gamma_{\mathrm{up}}$ and $\left\langle J\right\rangle $
given in Theorem \ref{thm:main_result}, and $\gamma_{\mathrm{sc}}:=\mathrm{tsup}\left(D\right)$
defined in (\ref{eq:def_gamma_sc}), let us suppose that $\gamma_{\mathrm{up}}<\gamma_{\mathrm{sc}}$.
Then for any $\epsilon>0$, there exists $\nu_{\epsilon}>0$, $C_{\epsilon}>0$,
such that for any $\nu>\nu_{\epsilon}$ we have for any $\left|z\right|>e^{\left(\gamma_{\mathrm{up}}+\epsilon\right)}$,
\end{thm}

\begin{equation}
\left\Vert \left(z-\mathcal{L}_{\nu,\chi}\right)^{-1}\right\Vert _{H_{\nu}^{-m}\left(\mathbb{R}\right)}\leq C_{\epsilon}\nu^{\frac{2}{\left\langle J\right\rangle +\epsilon}\left(\gamma_{\mathrm{sc}}-\gamma_{\mathrm{up}}\right)},\label{eq:bound_resvol}
\end{equation}

}}}\end{center}
\begin{rem}
\label{rem:3.3}~

\begin{enumerate}
\item If $\left|z\right|>e^{\left(\gamma_{\mathrm{sc}}+\epsilon\right)}$
then a bound of the resolvent norm independent on $\nu$ has been
obtained in \cite[Thm 2.9]{faure_arnoldi_tobias_13}.
\item The positive power of $\nu$ in (\ref{eq:bound_resvol}) (that diverges
for $\nu\rightarrow\infty$), is related to our choice of the escape
function that defines the norm in Sobolev space: the norm $\left\Vert \mathcal{L}_{\nu,\chi}^{n}\right\Vert $
is controlled by $\gamma_{\mathrm{up}}$ only for long time $n$ of
order $\frac{2}{\left\langle J\right\rangle }\log\nu$. This time
is the required time for a wave packet starting on the trapped set
to reach the region where there is an effective damping by the escape
function. Using cutoff functions in some exotic symbol classes that
allows a sharper cutoff in $\xi$, one should be able to improve this
term. 
\end{enumerate}
\end{rem}

\subsection{\label{subsec:Partially-expanding-maps}Expansion of correlation
functions for partially expanding maps}

Theorem \ref{thm:main_result-resolvent} has a direct application
for decay of correlation functions for a related partially expanding
dynamical system (see Remark \ref{rem:2.4}(3)). One obtains exactly
the same result as in \cite[Theorem 2.9]{faure_arnoldi_tobias_13}
with the only change that we take any $\rho>e^{\gamma_{\mathrm{up}}}$
and initial functions $u\in H^{-m}\left(I\right)\otimes H^{\sigma}\left(S^{1}\right)$,
$v\in H^{m}\left(I\right)\otimes H^{\sigma}\left(S^{1}\right)$ should
have regularity of positive order $\sigma=\frac{2}{\left\langle J\right\rangle +\epsilon}\left(\gamma_{\mathrm{sc}}-\gamma_{\mathrm{up}}\right)$
in the neutral direction. Here is the precise statement.

Let $\phi$ be an iterated function system as defined in Definition
\ref{def:IFSAn-iterated-function}. Recall that the map $\phi^{-1}:\phi\left(I\right)\rightarrow I$
is univalued and expanding. Let $\tau\in C^{\infty}\left(\phi\left(I\right);\mathbb{R}\right)$
as in Definition \ref{def_Transfer_op}. We define the map
\begin{equation}
f:\begin{cases}
\phi\left(I\right)\times S^{1} & \rightarrow I\times S^{1}\\
\left(x,y\right) & \rightarrow\left(\phi^{-1}\left(x\right),y+\tau\left(x\right)\right)
\end{cases}\label{eq:def_extended_map_f}
\end{equation}
with $S^{1}:=\mathbb{R}/\left(2\pi\mathbb{Z}\right)$. Notice that
the map $f$ is expanding in the $x$ variable whereas it is neutral
in the $y$ variable in the sense that $\frac{\partial f}{\partial y}=\left(0,1\right)$.
This is called a partially expanding map and may serve as a very simple
model for the general study of partially open hyperbolic dynamics
\cite{pesin_04} such as Axiom A flows. Let $V\in C^{\infty}\left(\phi(I);\mathbb{R}\right)$.

\begin{center}{\color{red}\fbox{\color{black}\parbox{16cm}{
\begin{defn}
The transfer operator of the map $f$ with potential $V$ is
\begin{equation}
\tilde{\mathcal{L}}:\begin{cases}
C_{0}^{\infty}\left(I\times S^{1}\right) & \rightarrow C_{0}^{\infty}\left(\phi\left(I\right)\times S^{1}\right)\\
\psi\left(x,y\right) & \mapsto e^{V\left(x\right)}\psi\left(f\left(x,y\right)\right)
\end{cases}.\label{eq:def_extended_transfer_op}
\end{equation}
\end{defn}

}}}\end{center}

A function $\psi\in C_{0}^{\infty}\left(I\times S^{1}\right)$ can
be decomposed in Fourier modes in the $y$ variable:

\begin{equation}
\psi\left(x,y\right)=\frac{1}{\sqrt{2\pi}}\sum_{\nu\in\mathbb{Z}}e^{i\nu y}\varphi_{\nu}\left(x\right)\label{eq:Fourier components}
\end{equation}
then
\[
\left(\tilde{\mathcal{L}}\psi\right)\left(x,y\right)=\frac{1}{\sqrt{2\pi}}\sum_{\nu\in\mathbb{Z}}e^{i\nu y}\left(\mathcal{L}_{\nu}\varphi_{\nu}\right)\left(x\right)
\]
with $\mathcal{L}_{\nu}$ given by (\ref{eq:def_transfert_op_F}).

We introduce some notation: for a given $\nu\in\mathbb{Z}$, we have
 seen in Theorem \ref{thm:-Discrete-spectrum.} that the transfer
operator $\mathcal{L}_{\nu}$ has a discrete spectrum of resonances.
Let $\rho>0$ such that there is no eigenvalue on the circle $\left|z\right|=\rho$
for any $\nu\in\mathbb{Z}$ and denote by $\Pi_{\rho,\nu}$ the spectral
projector of the operator $\mathcal{L}_{\nu}$ on the domain $\left\{ z\in\mathbb{C},\left|z\right|>\rho\right\} $.
These projection operators have obviously finite rank and each commutes
with $\mathcal{L}_{\nu}$. Theorem \ref{thm:main_result-resolvent}
has the following corollary.

\noindent\fcolorbox{blue}{white}{\begin{minipage}[t]{1\columnwidth - 2\fboxsep - 2\fboxrule}%
\begin{cor}
\label{thm:correl_funct}''\textbf{Expansion of correlations}''.
For $m$ large enough (such that $r<e^{\gamma_{\mathrm{up}}}$ in
Th. \ref{thm:-Discrete-spectrum.}), for any $\epsilon>0$, there
exists $\nu_{\epsilon}\in\mathbb{N}$ and $C_{\epsilon}>0$ such that
if we put $\sigma_{\epsilon}=\frac{2}{\left\langle J\right\rangle +\epsilon}\left(\gamma_{\mathrm{sc}}-\gamma_{\mathrm{up}}\right)$,
$\rho_{\epsilon}=e^{\gamma_{\mathrm{up}}+\epsilon}$, we have for
any $u\in H^{-m}\left(I\right)\otimes H^{\sigma_{\epsilon}}\left(S^{1}\right)$,
$v\in H^{m}\left(I\right)\otimes H^{\sigma_{\epsilon}}\left(S^{1}\right)$,
any $n\in\mathbb{N}$,
\begin{equation}
\left|\langle v|\tilde{\mathcal{L}}^{n}u\rangle-\sum_{\left|\nu\right|\leq\nu_{\epsilon}}\langle v_{\nu}|\left(\mathcal{L}_{\nu,\chi}\Pi_{\rho_{\epsilon},\nu}\right)^{n}u_{\text{\ensuremath{\nu}}}\rangle\right|\leq C_{\epsilon}\rho_{\epsilon}^{n}\left\Vert u\right\Vert _{H^{-m}\left(I\right)\otimes H^{\sigma_{\epsilon}}\left(S^{1}\right)}^{2}\left\Vert v\right\Vert _{H^{m}\left(I\right)\otimes H^{\sigma_{\epsilon}}\left(S^{1}\right)}^{2}.\label{eq:correlations_external_band}
\end{equation}
 Here $u_{\nu}\in H^{-m}\left(I\right)$, $v_{\nu}\in H^{m}\left(I\right)$
stand for the Fourier components in $S^{1}$ direction of $u,v$ defined
as in (\ref{eq:Fourier components}) and $\langle v|u\rangle:=\intop\overline{v}(x)u(x)dx$
(extended to distributions).
\end{cor}

\end{minipage}}
\begin{rem}
The second term in Eq.(\ref{eq:correlations_external_band}) is a
finite sum and each operator $\mathcal{L}_{\nu,\chi}\Pi_{\rho,\nu}$
has finite rank hence $\left(\mathcal{L}_{\nu,\chi}\Pi_{\rho_{\epsilon},\nu}\right)^{n}$
can be expended over individual eigenvalues. Using the spectral decomposition
of $\mathcal{L}_{\nu,\chi}$ we get an expansion of the correlation
function $\langle v|\tilde{\mathcal{L}}^{n}u\rangle$ with a finite
number of terms which involve the leading Ruelle resonances (i.e.
those with modulus greater than $\rho$) and an error term that is
$O\left(\rho^{n}\right)$.
\end{rem}

~
\begin{rem}
In (\ref{eq:def_extended_map_f}) we could consider $\left(x,y\right)\in\phi\left(I\right)\times\mathbb{R}$
instead which would give a Fourier decomposition $\psi\left(x,y\right)=\frac{1}{\sqrt{2\pi}}\int_{\mathbb{R}}e^{i\nu y}\varphi_{\nu}\left(x\right)d\nu$
with $\nu\in\mathbb{R}$. Then expansion of correlations would manifest
some ``diffusive behavior'' governed by the range $\left|\nu\right|\leq\nu_{0}$. 
\end{rem}

\begin{proof}
The proof of Corollary \ref{thm:correl_funct} is similar to the proof
of \cite[Theorem 2.9]{faure_arnoldi_tobias_13}.
\end{proof}

\subsection{\label{subsec:Other-interesting-results:}Other interesting results:
global normal form and asymptotic expansion.}

To get the result (\ref{eq:def_gamma_up}) we establish a ``\textbf{global
normal form}'' for the transfer operator. The term ``global'' means
here that the normal form is not specific to an individual fixed point
or a periodic orbit as it is usually done \cite{arnold-ed2} but concerns
the global dynamics in its whole. ``Global normal forms'' have already
been considered for hyperbolic dynamics \cite{delatte_92,delatte_95,fred-trace-06}
under the name ``non stationary normal form''. In this paper the
use of global normal form shows that the transfer operator is conjugated
to a simple dilation operator in a vicinity of any point and that
the conjugation is Hölder continuous with respect to the point considered.
This is particularly useful because dilation operators can be easily
composed and this helps to study the dynamics for ``long times''
$n$: this is Theorem \ref{thm:Global-normal-form.} that can be considered
as an interesting result of this paper by itself. Then we use an expansion
for large time\footnote{The symbol $n\gg\log\nu$ means here the right hand side of (\ref{eq:expansion_of_Fn})
get relatively small.} $n\gg\log\nu$, and obtain in Theorem \ref{thm:6.6} an asymptotic
expression for the transfer operator $\mathcal{L}_{\nu}^{n}$ as a
sum of rank one operators $\Pi_{w}$ that can be written\footnote{We don't give here a precise statement of the result. We just write
the main terms and ignore the remainders. We refer to Theorem \ref{thm:6.6}
for a precise statement.}:
\begin{equation}
\mathcal{L}_{\nu}^{n}\sim\sum_{w}e^{i\nu\tau_{w}+V_{w}-J_{w}}\Pi_{w},\label{eq:expansion}
\end{equation}
where the sum is over symbolic words $w$ of length $n$ that represent
orbits (explained in Section \ref{subsec:Symbolic-dynamics}), $V_{w}$
is the Birkhoff sum of the function $V$ along the orbit $w$ (similarly
for $\tau_{w}$ and $J_{w}$) and $\Pi_{w}$ is a rank one operator
of the form
\begin{equation}
\Pi_{w}=|\mathcal{U}_{w}\rangle\langle\mathcal{S}_{w}|\quad:\begin{cases}
H_{\nu}^{-m}\left(\mathbb{R}\right) & \rightarrow H_{\nu}^{-m}\left(\mathbb{R}\right)\\
u & \rightarrow\mathcal{U}_{w}\,\langle\mathcal{S}_{w}|u\rangle_{H_{\nu}^{-m}}
\end{cases}\label{eq:def_Pi_w-1}
\end{equation}
where $\mathcal{U}_{w},\mathcal{S}_{w}\in H_{\nu}^{-m}\left(\mathbb{R}\right)$
are distributions associated respectively to the unstable/stable manifolds
of the orbit $w$ as shown on figure \ref{fig:The-canonical-map}
(more precisely $\mathcal{U}_{w},\mathcal{S}_{w}$ are Lagrangian
WBK states, they will be precisely defined in Theorem \ref{thm:6.6}).

\subsection{\label{subsec:Sketch-of-the}Sketch of the proof}

Let us shortly outline the principal mechanism in the proof of the
new asymptotic gap bound (\ref{eq:def_gamma_up}) without discussing
the technical difficulties. If the operator $\mathcal{L}_{\nu}$ would
be trace class in $H_{\nu}^{-m}\left(\mathbb{R}\right)$ then $r_{s}\left(\mathcal{L}_{\nu}\right)\leq\left|\mathrm{Tr}\left(\left(\mathcal{L}_{\nu}^{n}\right)^{\dagger}\mathcal{L}_{\nu}^{n}\right)\right|^{1/\left(2n\right)}$
for any $n\geq1$, where $\dagger$ stands for adjointness in the
specific Hilbert space $H_{\nu}^{-m}\left(\mathbb{R}\right)$. In
order to obtain good bounds on the trace norm we develop in a first
step a global normal form (Section \ref{sec:Global-normal-form})
as well as an asymptotic expansion (Section \ref{sec:Asymptotic-expansion}).
This leads to the expansion (\ref{eq:expansion}) for $\mathcal{L}_{\nu}^{n}$
and similarly for its adjoint $\left(\mathcal{L}_{\nu}^{n}\right)^{\dagger}\sim\sum_{w'}e^{-i\nu\tau_{w'}+V_{w'}-J_{w'}}\Pi_{w'}^{\dagger}$.
Using the definition of the asymptotic spectral gap (\ref{eq:def_gamma_asympt-1})
we get for $\nu\rightarrow\infty$,
\begin{equation}
\gamma_{\mathrm{asympt.}}\leq\log\left(r_{s}\left(\mathcal{L}_{\nu}\right)\right)\leq\frac{1}{2n}\log\sum_{w,w'}e^{\left(V-J\right)_{w}+\left(V-J\right)_{w'}+i\nu\left(\tau_{w}-\tau_{w'}\right)}\mathrm{Tr}\left(\Pi_{w'}^{\dagger}\Pi_{w}\right).\label{eq:bound}
\end{equation}
We have $\mathrm{Tr}\left(\Pi_{w'}^{\dagger}\Pi_{w}\right)\underset{(\ref{eq:def_Pi_w-1})}{=}\langle\mathcal{S}_{w}|\mathcal{S}_{w'}\rangle\langle\mathcal{U}_{w'}|\mathcal{U}_{w}\rangle$.
In Proposition \ref{Prop:Separation-of-orbits} we will show that
for time $n$ less than twice the Ehrenfest time i.e. such that $n<2\frac{\log\nu}{\left\langle J\right\rangle }$
(or $e^{n\left\langle J\right\rangle }<\nu^{2}$) then the unstable/stable
manifolds $\mathcal{S}_{w},\mathcal{U}_{w}$ are well separated\footnote{This is up to some few pairs $\left(w,w'\right)$ that give negligible
contributions. To control these terms we use large deviation techniques
explained in Appendix \ref{sec:A-formula-by}.} in the sense that $w\neq w'\Rightarrow\mathrm{Tr}\left(\Pi_{w'}^{\dagger}\Pi_{w}\right)\simeq0$.
Applying this separation to the double sum (\ref{eq:bound}) means
that the non diagonal terms can be neglected and we call this the
``diagonal approximation''. For the diagonal terms that remain we
compute that $\mathrm{Tr}\left(\Pi_{w}^{\dagger}\Pi_{w}\right)\leq\nu C$.
This gives
\begin{eqnarray}
\gamma_{\mathrm{asympt.}} & \underset{(\ref{eq:bound})}{\leq} & \frac{1}{2n}\log\left(\sum_{w}e^{2\left(V-J\right)_{w}}C\nu\right)\label{eq:final_step}\\
 & \underset{(\ref{eq:def_pressure-1})}{=} & \frac{1}{2}\mathrm{Pr}\left(2\left(V-J\right)\right)+\frac{1}{2n}\log C+\frac{1}{2n}\log\nu+o\left(1\right)\nonumber 
\end{eqnarray}
Then we take the time $n=2\frac{\log\nu}{\left\langle J\right\rangle }\left(1-\epsilon\right)$
with any $\epsilon>0$ and $\nu\rightarrow+\infty$ and get the result
(\ref{eq:def_gamma_up}) that $\gamma_{\mathrm{asympt.}}\leq\gamma_{\mathrm{up}}:=\frac{1}{2}\mathrm{Pr}\left(2\left(V-J\right)\right)+\frac{1}{4}\left\langle J\right\rangle $.

\section{\label{sec:The-canonical-map}The canonical map $\tilde{\phi}$ and
its trapped set $\mathcal{K}$ in phase space}

According to \cite[Lemma 4.2]{faure_arnoldi_tobias_13}, the transfer
operator $\mathcal{L}_{i,j}$ in (\ref{e:def_F_op_ij}) is a $\nu$-semiclassical
Fourier integral operator (FIO)\footnote{The reader does not need to be familiar with the theory of global
Fourier integral operators for the rest of the article. For a discussion
of FIOs in the context of IFS-transfer operators we refer to \cite[Section 4]{faure_arnoldi_tobias_13}.
For a more general introduction we refer to \cite[Chap.10]{zworski_book_2012}. }. It is a general fact in semiclassical analysis that various properties
of Fourier integral operators are obtained from the properties of
their associated symplectic map (or canonical map) which are maps
on the cotangent space. Here we use coordinates $x\in I_{i}$ and
$\xi\in T_{x}^{*}I_{i}$ . The canonical map associated to $\mathcal{L}_{i,j}$
is defined by \cite[Lemma 4.2]{faure_arnoldi_tobias_13}

\begin{equation}
\tilde{\phi}_{i,j}:\begin{cases}
T^{*}I_{i} & \rightarrow T^{*}I_{j}\\
\left(x,\xi\right) & \rightarrow\begin{cases}
x' & =\phi_{i,j}\left(x\right)\\
\xi' & =\frac{1}{\phi_{i,j}'\left(x\right)}\xi+\tau'\left(x'\right).
\end{cases}
\end{cases}\label{eq:def_symplectic_map_Fij-1}
\end{equation}
This gives a multi-valued canonical map $\tilde{\phi}:T^{*}I\rightarrow T^{*}I$
on the phase space $T^{*}I\cong I\times\mathbb{R}$, given by:
\begin{equation}
\tilde{\phi}:\begin{cases}
T^{*}I & \rightarrow T^{*}I\\
\left(x,\xi\right) & \rightarrow\left\{ \tilde{\phi}_{i,j}\left(x,\xi\right)\quad\quad\mbox{with }i,j\mbox{ s.t. }x\in I_{i},\,i\rightsquigarrow j\right\} .
\end{cases}\label{eq:canonical_map_Fij}
\end{equation}

We have the following property of ``escape at infinity outside a
compact'' for the dynamics defined by $\tilde{\phi}:T^{*}I\rightarrow T^{*}I$:

\vspace{0.cm}\begin{center}{\color{blue}\fbox{\color{black}\parbox{16cm}{
\begin{lem}
\label{lem:escape_F}\cite[Lemma 4.4]{faure_arnoldi_tobias_13}. For
any $1<\kappa<e^{J_{\mathrm{min}}}$, there exists $C\geq0$ such
that $\forall x\in I_{i}$, $\forall\xi\in\mathbb{R}$, $\forall j\text{ s.t. }i\rightsquigarrow j$,
\begin{equation}
\left(x',\xi'\right)=\tilde{\phi}_{i,j}\left(x,\xi\right)\text{ and }\left|\xi\right|>C\quad\Rightarrow\quad\left|\xi'\right|>\kappa\left|\xi\right|.\label{eq:expand_nu}
\end{equation}
\end{lem}

}}}\end{center}\vspace{0.cm}
\begin{rem}
\label{rem:escape}At this stage we observe from (\ref{eq:expand_nu})
that the function $A_{m}\left(x,\xi\right)=\left\langle \xi\right\rangle ^{-m}$
defined in (\ref{eq:def_A}) satisfies $\left|\xi\right|>C\Rightarrow\frac{A_{m}\left(\tilde{\phi}_{i,j}\left(x,\xi\right)\right)}{A_{m}\left(x,\xi\right)}\leq c^{m}$
with $c=\sqrt{\frac{C^{2}+1}{\left(C\kappa\right)^{2}+1}}<1$, i.e.
$A_{m}$ decreases strictly with the dynamics. This explains why we
call $A_{m}$ an ``escape function''.
\end{rem}

\subsection{The trapped set $\mathcal{K}$}

We define the trapped set $\mathcal{K}$ for the dynamics of the canonical
map $\tilde{\phi}$ in (\ref{eq:canonical_map_Fij}), as the points
which do not escape ``totally'' neither in the past nor the future:

\begin{center}{\color{red}\fbox{\color{black}\parbox{16cm}{
\begin{defn}
The \textbf{trapped set} in phase space $T^{*}I\equiv I\times\mathbb{R}$
is defined as
\begin{eqnarray*}
\mathcal{K} & = & \left\{ \left(x,\xi\right)\in I\times\mathbb{R},\quad\exists\mathcal{D}\Subset I\times\mathbb{R}\mbox{ compact},\right.\\
 &  & \quad s.t.\left.\forall n\in\mathbb{Z}\quad\mbox{ }\tilde{\phi}^{n}\left(x,\xi\right)\cap\mathcal{D}\neq\emptyset\right\} 
\end{eqnarray*}
\end{defn}

}}}\end{center}
\begin{rem}
Since the map $\tilde{\phi}$ is a lift of the map $\phi:I\rightarrow I$
we conclude that $\mathcal{K}\subset\left(K\times\mathbb{R}\right)$
with $K$ the trapped set of $\phi$ defined in (\ref{eq:trapped_set_K_def1}).
Using any value of $C$ given from Lemma \ref{lem:escape_F} one can
make this precise and obtain that
\[
\mathcal{K}\subset\left(K\times\left\{ \xi\in\mathbb{R},\left|\xi\right|\leq C\right\} \right).
\]
\end{rem}

For $\varepsilon>0$, let $\mathcal{K}_{\varepsilon}$ denote a closed
$\varepsilon-$neighborhood of the trapped set $\mathcal{K}$, namely
\[
\mathcal{K}_{\varepsilon}:=\left\{ \left(x,\xi\right)\in T^{*}I,\quad\exists\left(x_{0},\xi_{0}\right)\in\mathcal{K},\quad\max\left(\left|x-x_{0}\right|,\left|\xi-\xi_{0}\right|\right)\leq\varepsilon\right\} .
\]
From now on we will make the following hypothesis on the multi-valued
map $\tilde{\phi}$.

\vspace{0.cm}\begin{center}{\color{red}\fbox{\color{black}\parbox{16cm}{
\begin{assumption}
\label{hyp:minimal_capt}We assume the following property called ``\textbf{minimal
captivity'':}
\begin{equation}
\exists\varepsilon>0,\quad\forall\left(x,\xi\right)\in\mathcal{K}_{\varepsilon},\quad\sharp\left\{ \tilde{\phi}\left(x,\xi\right)\bigcap\mathcal{K}_{\varepsilon}\right\} \leq1.\label{eq:hyp_minimal_captivity_epsilon}
\end{equation}
This means that the dynamics of $\tilde{\phi}$ is univalued on a
neighborhood of the trapped set $\mathcal{K}$.
\end{assumption}

}}}\end{center}\vspace{0.cm}
\begin{rem}
\label{rem:min_captive}~

\begin{enumerate}
\item The minimal captivity assumption has been introduced in \cite{faure_arnoldi_tobias_13}
and it allows an easy description of the trapped set $\mathcal{K}$.
We refer to \cite[Prop. 4.1]{faure_arnoldi_tobias_13} for further
discussions and alternative equivalent formulations.
\item The minimal captivity assumption is a stronger assumption than the
Dolgopyat non-local integrability condition, \cite{dolgopyat_98},\cite[Definition 2.1]{naud_expanding_2005}.
We refer to \cite[Section 4.3]{faure_arnoldi_tobias_13} for a comparison.
\item In \cite[Prop. 7.3]{faure_arnoldi_tobias_13} an explicit procedure
to verify the minimal captivity assumption is discussed and minimal
captivity assumption is proven for the examples of Bowen Series map
and the truncated Gauss maps.
\item Given the minimal captivity assumption \ref{hyp:minimal_capt}, let
$U\subset\mathbb{R}^{2}$ open such that $\mathcal{K\subset}U\subset(\mathcal{K_{\varepsilon}\cap\tilde{\phi}}^{-1}(\mathcal{K_{\varepsilon}}))$.
We can extend the univalued map $\tilde{\phi}$ on $\mathcal{K}$
to an embedding $\tilde{\phi}:U\to\mathbb{R}^{2}$. With this point
of view $\mathcal{K}$ is simply the maximal invariant hyperbolic
set of the diffeomorphism $\tilde{\phi}$ (cf. \cite[Section 1.b, Section 2.h]{hasselblatt_02b}).
This observation will be useful in the proof of Lemma \ref{lem:9}
to use regularity estimates on the stable foliation of $\tilde{\phi}$
.
\end{enumerate}
\end{rem}

\subsection{\label{subsec:Symbolic-dynamics}Symbolic dynamics on the trapped
set $K\subset I$}

In the following two sub-sections we introduce the symbolic dynamics
on the trapped set $K\subset I$ and the trapped set $\mathcal{K}\subset T^{*}I$
in phase space. Note that the symbolic dynamics is not necessary for
the definition of Ruelle-Pollicott resonances neither for the statement
of the results. It is a useful tool to keep track of orbits and is
natural in the context of I.F.S. dynamics that is defined from a set
of intervals $\left(I_{j}\right)_{j=1\ldots N}$.

We first consider the dynamics of $\phi:I\rightarrow I$ on the ``base
space'' $I$. Let

\begin{equation}
\mathcal{W}_{-}:=\left\{ \left(\ldots,w_{-2},w_{-1},w_{0}\right)\in\left\{ 1,\ldots,N\right\} ^{-\mathbb{N}},w_{l-1}\rightsquigarrow w_{l},\forall l\leq0\right\} \label{eq:def_W_-}
\end{equation}
be the set of admissible left semi-infinite sequences. In other words,
$\mathcal{W}_{-}$ is a subshift of finite type \cite[p.56]{brin-02}.
For $w_{-}\in\mathcal{W}_{-}$ and $i<j\leq0$ we write 
\begin{equation}
w_{i,j}:=\left(w_{i},w_{i+1},\ldots w_{j}\right)\label{eq:def_i,j}
\end{equation}
for an extracted sequence. For simplicity we will use the following
notation for the composition of maps: 
\begin{equation}
\phi_{w_{i,j}}:=\phi_{w_{j-1},w_{j}}\circ\phi_{w_{j-2},w_{j-1}}\circ\ldots\circ\phi_{w_{i},w_{i+1}}:I_{w_{i}}\rightarrow I_{w_{j}}.\label{eq:def_phi_w_ij}
\end{equation}
For $n\geq0$, let
\begin{equation}
I_{w_{-n,0}}:=\phi_{w_{-n,0}}\left(I_{w_{-n}}\right)\subset I_{w_{0}}.\label{eq:def_Iwn}
\end{equation}
See figure \ref{fig:symbolic}.

\begin{figure}[h]
\begin{centering}
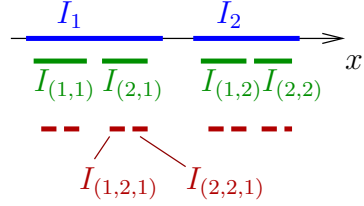
\par\end{centering}
\caption{\label{fig:symbolic} This picture illustrates the definition of intervals
$I_{w_{-n,0}}$ defined in (\ref{eq:def_Iwn}) from a word $w_{-n,0}=\left(w_{-n},\ldots w_{-1},w_{0}\right)$
and the contracting maps $\phi_{i,j}:I_{i}\rightarrow I_{j}$. For
example here $I_{\left(2,2,1\right)}:=\phi_{\left(2,2,1\right)}\left(I_{2}\right):=\phi_{2,1}\circ\phi_{2,2}\left(I_{2}\right)=\phi_{2,1}\left(I_{2,2}\right)\subset I_{1}.$}
\end{figure}
For any $0<m<n$ we have the strict inclusions $I_{w_{-n,0}}\subset I_{w_{-m,0}}\subset I_{w_{0}}$
and from (\ref{eq:theta_contraction}), the size of $I_{w_{-n,0}}$
is bounded by $\left|I_{w_{-n,0}}\right|\leq\theta^{n}\left|I_{w_{0}}\right|,$
hence the sequence of sets $\left(I_{w_{-n,0}}\right)_{n\geq1}$ is
a sequence of non empty and decreasing closed intervals and $\bigcap_{n=1}^{\infty}I_{w_{-n,0}}$
is a point in the trapped set $K$, Eq.(\ref{eq:trapped_set_K_def1}).
So we can define

\begin{center}{\color{red}\fbox{\color{black}\parbox{16cm}{
\begin{defn}
\label{def:The-symbolic-coding_S}The ``\textbf{symbolic coding map}
of $K$'' is 
\begin{equation}
S:\begin{cases}
\mathcal{W}_{-} & \rightarrow K\\
w_{-} & \mapsto S\left(w_{-}\right):=\bigcap_{n=1}^{\infty}I_{w_{-n,0}}
\end{cases}\label{eq:def_coding_S}
\end{equation}
\end{defn}

}}}\end{center}

Let us introduce the \textbf{left shift $L$,} a multivalued map,
defined by
\begin{equation}
L:\begin{cases}
\mathcal{W}_{-} & \rightarrow\mathcal{W}_{-}\\
\left(\ldots,w_{-2},w_{-1},w_{0}\right) & \rightarrow\left(\ldots,w_{-2},w_{-1},w_{0},w_{1}\right)
\end{cases}\label{eq:def_L}
\end{equation}
with $w_{1}\in\left\{ 1,\ldots,N\right\} $ such that $w_{0}\rightsquigarrow w_{1}$
and let the \textbf{right shift} $R$ be the univalued map defined
by
\begin{equation}
R:\begin{cases}
\mathcal{W}_{-} & \rightarrow\mathcal{W}_{-}\\
\left(\ldots,w_{-2},w_{-1},w_{0}\right) & \rightarrow\left(\ldots,w_{-2},w_{-1}\right)
\end{cases}.\label{eq:def_R}
\end{equation}
\begin{framed}%
\begin{prop}
\label{prop:symbolic dynamics}\cite[Prop. 4.12]{faure_arnoldi_tobias_13}
The following diagram is commutative
\begin{eqnarray}
\mathcal{W}_{-} & \overset{S}{\longrightarrow} & \qquad K\label{eq:diagram_S}\\
R\uparrow\downarrow L &  & \phi^{-1}\uparrow\downarrow\phi\nonumber \\
\mathcal{W}_{-} & \overset{S}{\longrightarrow} & \qquad K\nonumber 
\end{eqnarray}

and the map $S:\mathcal{W}_{-}\rightarrow K$ is one to one.
\end{prop}

\end{framed}

\subsection{Symbolic dynamics on the trapped set $\mathcal{K}\subset T^{*}I$}

We consider now the dynamics of the canonical map $\tilde{\phi}:I\times\mathbb{R}\rightarrow I\times\mathbb{R}$
on the phase space. Let

\[
\mathcal{W}_{+}:=\left\{ \left(w_{0},w_{1},w_{2}\ldots\right)\in\left\{ 1,\ldots,N\right\} ^{\mathbb{N}},\quad w_{l}\rightsquigarrow w_{l+1},\forall l\geq0\right\} 
\]
be the set of admissible right semi-infinite sequences. For any $n\geq0$
let
\begin{equation}
\tilde{I}_{w_{0,n}}:=\tilde{\phi}^{-n}\left(I_{w_{0,n}}\times\left[-C,C\right]\right)\label{eq:def_set_I_tilde}
\end{equation}
be the image of the rectangle under the univalued map $\tilde{\phi}^{-n}$,
where $I_{w_{0,n}}$ is given in (\ref{eq:def_Iwn}) and $C$ is given
by Lemma \ref{lem:escape_F}. See figure \ref{fig:trapped_set} (a).
Notice that $\pi\left(\tilde{I}_{w_{0,n}}\right)=I_{w_{0}}$ where
$\pi\left(x,\xi\right)=x$ is the canonical projection map. Since
the map $\tilde{\phi}^{-1}$ contracts strictly in the variable $\xi$
by the factor $\theta<1$ the sequence $\left(\tilde{I}_{w_{0,n}}\right)_{n\in\mathbb{N}}$
is strictly decreasing: $\tilde{I}_{w_{0,n+1}}\subset\tilde{I}_{w_{0,n}}$
and we can define the limit

\begin{equation}
\tilde{S}:\begin{cases}
\mathcal{W}_{+} & \rightarrow I\times\mathbb{R}\\
w_{+} & \mapsto\tilde{S}\left(w_{+}\right):=\bigcap_{n\geq0}\tilde{I}_{w_{0,n}}
\end{cases}.\label{eq:def_S_tilde}
\end{equation}

\begin{center}{\color{blue}\fbox{\color{black}\parbox{16cm}{
\begin{prop}
\cite[Prop. 4.13]{faure_arnoldi_tobias_13}For every $w_{+}\in\mathcal{W}_{+}$,
the set $\tilde{S}\left(w_{+}\right)\subset T^{*}I_{w_{0}}$ is a
smooth curve given by
\begin{equation}
\tilde{S}\left(w_{+}\right)=\left\{ \left(x,\zeta_{w_{+}}\left(x\right)\right),\quad x\in I_{w_{0}}\right\} \label{eq:def_S_tilde_w+}
\end{equation}
with
\begin{equation}
\zeta_{w_{+}}\left(x\right):=-\sum_{k\geq1}e^{-J_{w_{0,k}}\left(x\right)}\tau'\left(\phi_{w_{0,k}}(x)\right),\label{eq:expression_zeta_w}
\end{equation}
 and 
\begin{equation}
J_{w_{0,n}}\left(x\right):=\sum_{k=1}^{n}J_{w_{k},w_{k+1}}\left(\phi_{w_{0,k}}\left(x\right)\right)\label{eq:def_Birkhoff_sum}
\end{equation}
 is the Birkhoff sum of the ``Jacobian function'' defined in (\ref{eq:def_J})
\begin{equation}
J_{i,j}\left(x\right)=-\log\left(\phi'_{i,j}\left(x\right)\right)>0.\label{eq:def_J_i,j}
\end{equation}

We have an estimate of regularity, uniform in $w$: $\forall\alpha\in\mathbb{N}$,
$\exists C_{\alpha}>0$, $\forall w_{+}\in\mathcal{W}_{+}$, $\forall x\in I_{w_{0}}$,
\begin{equation}
\left|\left(\partial_{x}^{\alpha}\zeta_{w_{+}}\right)\left(x\right)\right|\leq C_{\alpha}.\label{eq:estimate}
\end{equation}

Moreover, with the hypothesis \ref{hyp:minimal_capt} of minimal captivity
with a neighborhood $\mathcal{K}_{\varepsilon}$ of $\mathcal{K}$,
there exists $a\geq1$ and $K_{a}$ defined in (\ref{eq:def_Kn})
such that
\begin{equation}
\forall x\in K_{a},\forall w_{+}\in\mathcal{W}_{+},\quad\left(x,\zeta_{w_{+}}\left(x\right)\right)\in\mathcal{K}_{\varepsilon}.\label{eq:Ka_w}
\end{equation}
\end{prop}

}}}\end{center}
\begin{proof}
For (\ref{eq:expression_zeta_w}) and (\ref{eq:estimate}) see \cite[Prop. 4.13]{faure_arnoldi_tobias_13}.
For (\ref{eq:Ka_w}) see \cite[Prop. 4.10 (1)]{faure_arnoldi_tobias_13}.
\end{proof}
Let

\[
\mathcal{W}:=\left\{ \left(\ldots w_{-2},w_{-1},w_{0},w_{1},\ldots\right)\in\left\{ 1,\ldots,N\right\} ^{\mathbb{Z}},\quad w_{l}\rightsquigarrow w_{l+1},\forall l\in\mathbb{Z}\right\} 
\]
be the set of bi-infinite admissible sequences. For a given $w\in\mathcal{W}$
and $a,b\in\mathbb{N}$, let
\begin{equation}
\mathcal{I}_{w_{-a,b}}:=\left(\pi^{-1}\left(I_{w_{-a,0}}\right)\cap\tilde{I}_{w_{0,b}}\right).\label{eq:def_I_tilde_Iab}
\end{equation}

See Figure \ref{fig:trapped_set} (a).

\begin{figure}[h]
\begin{centering}
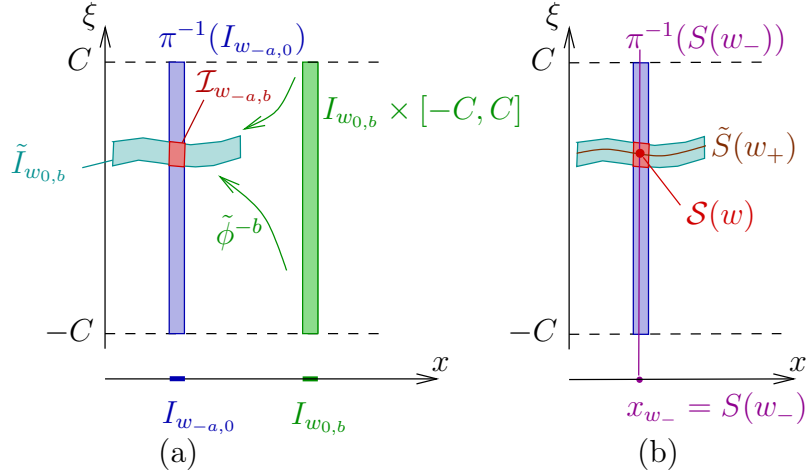
\par\end{centering}
\caption{\label{fig:trapped_set}Figure (a) illustrates the construction of
$\tilde{I}_{w_{0,b}}$ given in (\ref{eq:def_set_I_tilde}), the construction
of $\mathcal{I}_{w_{-a,b}}$ given in (\ref{eq:def_I_tilde_Iab}).
Figure (b) illustrates the limit of these sets for semi-infinite words,
i.e. $a,b\rightarrow\infty$. This gives the smooth curve $\tilde{S}\left(w_{+}\right)$
given in (\ref{eq:def_S_tilde_w+}), the point $x_{w_{-}}=S\left(w_{-}\right)\in K$
given in (\ref{eq:def_coding_S}), the vertical line $\pi^{-1}\left(S\left(w_{-}\right)\right)$
and finally the intersection $\mathcal{S}\left(w\right):=\left(\pi^{-1}\left(S\left(w_{-}\right)\right)\cap\tilde{S}\left(w_{+}\right)\right)\in\mathcal{K}$
given in (\ref{eq:def_S_trapped}), that depends on a bi-infinite
word $w\equiv\left(w_{-},w_{+}\right)\in\mathcal{W}$.}
\end{figure}

\begin{center}{\color{red}\fbox{\color{black}\parbox{16cm}{
\begin{defn}
The \textbf{symbolic coding map} of $\mathcal{K}$ is
\begin{equation}
\mathcal{S}:\begin{cases}
\mathcal{W} & \rightarrow\mathcal{K}\\
w & \mapsto\mathcal{S}\left(w\right):=\bigcap_{n=1}^{\infty}\mathcal{I}_{w_{-n,n}}=\left(\pi^{-1}\left(S\left(w_{-}\right)\right)\cap\tilde{S}\left(w_{+}\right)\right)
\end{cases}\label{eq:def_S_trapped}
\end{equation}
with $w_{-}=\left(\ldots w_{-1},w_{0}\right)\in\mathcal{W}_{-}$,
$w_{+}=\left(w_{0},w_{1},\ldots\right)\in\mathcal{W}_{+}$ (with the
same extreme letter $w_{0}$). 
\end{defn}

}}}\end{center}

See figure \ref{fig:trapped_set} (b). More precisely we can express
the point $\mathcal{S}\left(w\right)\in\mathcal{K}$ as
\begin{equation}
\mathcal{S}\left(w\right)=\left(x_{w},\xi_{w}\right),\quad x_{w}=S\left(w_{-}\right)\in K,\quad\xi_{w}=\zeta_{w_{+}}\left(x_{w}\right)\in T_{x_{w}}^{*}I,\label{eq:point_in_trapped_K}
\end{equation}
with $S\left(w_{-}\right)$ given in (\ref{eq:def_coding_S}) and
$\zeta_{w_{+}}$ given in (\ref{eq:expression_zeta_w}).

Let $L,R$ denote the full left/right shift on $\mathcal{W}$ defined
similarly to (\ref{eq:def_L}) and (\ref{eq:def_R}) by $\left(Lw\right)_{j}=w_{j+1}$
and $\left(Rw\right)_{j}=w_{j-1}$.

\begin{center}{\color{blue}\fbox{\color{black}\parbox{16cm}{
\begin{prop}
\label{prop:S_tilde_bijective}\cite[Prop. 4.15]{faure_arnoldi_tobias_13}
The following diagram is commutative
\begin{eqnarray}
\mathcal{W} & \overset{\mathcal{S}}{\longrightarrow} & \qquad\mathcal{K}\label{eq:diagram_S-1}\\
R\uparrow\downarrow L &  & \tilde{\phi}^{-1}\uparrow\downarrow\tilde{\phi}\nonumber \\
\mathcal{W} & \overset{\mathcal{S}}{\longrightarrow} & \qquad\mathcal{K}\nonumber 
\end{eqnarray}
If assumption \ref{hyp:minimal_capt} of minimal captivity holds true
then the map $\mathcal{S}:\mathcal{W}\rightarrow\mathcal{K}$ is one
to one. This means that the univalued dynamics of points on the trapped
set $\mathcal{K}$ under the maps $\tilde{\phi}^{-1},\tilde{\phi}$
is equivalent to the symbolic dynamics of the full shift maps $R,L$
on the set of words $\mathcal{W}$.
\end{prop}

}}}\end{center}
\begin{rem}
\label{rem:invariant_manifolds}Considering the trapped set $\mathcal{K}$
as the hyperbolic set with a local product structure of $\tilde{\phi}$
(cf. Remark \ref{rem:min_captive}(4)), the curve $\tilde{S}_{w_{+}}=(x,\zeta_{w}(x))$
is precisely the stable manifold \cite[Section 2.e]{hasselblatt_02b}
through the point $(x_{w},\zeta_{w^{+}}(x_{w}))\in\mathcal{K}$. The
unstable manifold is the vertical line $\pi^{-1}\left(S\left(w_{-}\right)\right)=\left\{ \left(x_{w},\xi\right),\xi\in\mathbb{R}\right\} $
(cf. Figure \ref{fig:The-canonical-map} or Figure \ref{fig:trapped_set}
(b)). 
\end{rem}

\section{Global normal form\label{sec:Global-normal-form}}

Normal forms are usually constructed for individual fixed points or
individual periodic orbits \cite{arnold-ed2}. In few papers, normal
forms have already been considered globally for a hyperbolic dynamics
\cite{delatte_92,delatte_95,fred-trace-06}. We present here the global
normal form for the transfer operator $\mathcal{L}_{\nu}$ considered
in this paper. This is Theorem \ref{thm:Global-normal-form.} below.
We will need the following elementary (Fourier integral) operators
on $C_{0}^{\infty}\left(\mathbb{R}\right)$ and their associated symplectic
(or canonical) maps on $T^{*}\mathbb{R}\equiv\mathbb{R}_{x,\xi}^{2}$
\cite[chap.10]{zworski_book_2012}. Let $\varphi\in C_{0}^{\infty}\left(\mathbb{R}\right)$.
\begin{itemize}
\item For $\lambda\in\mathbb{R}$, the \textbf{dilation operator} is
\begin{equation}
\left(\hat{D}_{\lambda}\varphi\right)\left(y\right):=\varphi\left(e^{\lambda}y\right)\label{eq:def_D_lambda}
\end{equation}
whose canonical map is $D_{\lambda}:\left(x,\xi\right)\rightarrow\left(x'=e^{-\lambda}x,\xi'=e^{\lambda}\xi\right)$.
\item For $x\in\mathbb{R}$, the \textbf{translation operator} is
\[
\left(\hat{T}_{x}\varphi\right)\left(y\right):=\varphi\left(y-x\right)
\]
whose canonical map is $T_{x}:\left(y,\xi\right)\rightarrow\left(y'=y+x,\xi'=\xi\right)$.
\item For a smooth diffeomorphism $f:\mathbb{R}\rightarrow\mathbb{R}$ ,
the \textbf{composition operator} is
\[
\left(\hat{\mathcal{L}}_{f}\varphi\right)\left(y\right):=\varphi\left(f^{-1}\left(y\right)\right)
\]
whose canonical map is $\mathcal{L}_{f}:\left(x,\xi\right)\rightarrow\left(x'=f\left(x\right),\xi'=\left(f'\left(x\right)\right)^{-1}\cdot\xi\right)$.
\item For two smooth functions $\varUpsilon^{\left(0\right)},\varUpsilon^{\left(1\right)}\in C^{\infty}(\mathbb{R};\mathbb{R})$
and $\nu>0$, let $\varUpsilon=\varUpsilon^{\left(0\right)}+\frac{i}{\nu}\varUpsilon^{\left(1\right)}$
and
\[
\left(\hat{\Theta}_{\varUpsilon}\varphi\right)\left(y\right):=e^{i\nu\varUpsilon\left(y\right)}\varphi\left(y\right)
\]
whose canonical map is $\Theta_{\varUpsilon}:\left(x,\xi\right)\rightarrow\left(x'=x,\xi'=\xi+\frac{d\varUpsilon^{\left(0\right)}}{dx}\left(x\right)\right)$.
\end{itemize}
\begin{rem}
The dilation operator is a special case of a composition operator:
$\hat{D}_{\lambda}=\hat{\mathcal{L}}_{f}$ with $f\left(y\right)=e^{-\lambda}y$.
Similarly for the translation operator: $\hat{T}_{x}=\hat{\mathcal{L}}_{f}$
with $f\left(y\right)=y-x$. For any two diffeomorphism $f,g$ one
has $\hat{\mathcal{L}}_{f}\circ\hat{\mathcal{L}}_{g}=\hat{\mathcal{L}}_{f\circ g}$.
\end{rem}

The next theorem shows that using a combination of these previous
simple operators, the transfer operator $\mathcal{L}_{\nu}:C_{0}^{\infty}\left(I\right)\rightarrow C_{0}^{\infty}\left(I\right)$
defined in (\ref{eq:def_transfert_op_F}) is ``globally conjugated''
to a simple dilation operator. This is illustrated in Figure \ref{fig:The-canonical-map}.

\begin{center}{\color{blue}\fbox{\color{black}\parbox{16cm}{
\begin{thm}
\textbf{\label{thm:Global-normal-form.}``Global normal form''.
}For any word $w=\left(\ldots w_{-1},w_{0},w_{1},\ldots\right)\in\mathcal{W}$
there exist functions $\varUpsilon_{w}^{\left(0\right)},\varUpsilon_{w}^{\left(1\right)}\in C^{\infty}(I_{w_{0}};\mathbb{R})$,
$\varUpsilon_{w}=\varUpsilon_{w}^{\left(0\right)}+\frac{i}{\nu}\varUpsilon_{w}^{\left(1\right)}\in C^{\infty}\left(I_{w_{0}};\mathbb{C}\right)$
as well as a map $H_{w}:\mathcal{J}\to\mathbb{R}$ defined on an neighborhood
$\mathcal{J}\subset\mathbb{R}$ of $0$, which is independent of the
word $w\in\mathcal{W}$. $H_{w}$ is a $C^{\infty}$ diffeomorphism
onto its image and the following points hold.

\begin{enumerate}
\item There exists a neighborhood $U$ of $x_{w}$ such that the transfer
operator in (\ref{e:def_F_op_ij}) acting on $C_{0}^{\infty}\left(U\right)$
can be expressed as
\begin{equation}
\mathcal{L}_{w_{0},w_{1}}=e^{i\nu\tau\left(x_{L\left(w\right)}\right)+V\left(x_{L\left(w\right)}\right)}\cdot\widehat{\mathcal{T}}_{L\left(w\right)}\circ\hat{D}_{J_{w_{0},w_{1}}\left(x_{w}\right)}\circ\widehat{\mathcal{T}}_{w}^{-1}\label{eq:conj_F01}
\end{equation}
with 
\begin{equation}
\widehat{\mathcal{T}}_{w}:=\hat{\Theta}_{\varUpsilon_{w}}\circ\hat{T}_{x_{w}}\circ\hat{\mathcal{L}}_{H_{w}},\label{eq:def_Tau_w}
\end{equation}
$x_{w}\in I_{w_{0}}$ defined in (\ref{eq:point_in_trapped_K}) and
$J_{w_{0},w_{1}}$ defined in (\ref{eq:def_J_i,j}). Eq.(\ref{eq:conj_F01})
means that the components of the transfer operator (\ref{e:def_F_op_ij})
are conjugated to some dilation operator multiplied by a constant.
The operator $\widehat{\mathcal{T}}_{w}$ is a FIO whose canonical
map $\mathcal{T}_{w}$ sends $\left(0,0\right)$ to the point $\mathcal{S}\left(w\right)=\left(x_{w},\xi_{w}\right)\in\mathcal{K}$.
\item $H_{w}\left(0\right)=0$, $H'_{w}\left(0\right)=1$, $\varUpsilon_{w}\left(x_{w}\right)=0$
and 
\[
\frac{d\varUpsilon_{w}^{\left(0\right)}}{dy}\left(y\right)=\zeta_{w}\left(y\right)
\]
 with $\zeta_{w}\left(y\right):=\zeta_{w_{+}}\left(y\right)$ given
in (\ref{eq:expression_zeta_w}).
\item For any $\alpha\in\mathbb{N}$ there exists $C_{\alpha}$ such that
for any $w\in\mathcal{W}$,
\begin{equation}
\left|\partial^{\alpha}H_{w}\right|_{L^{\infty}}<C_{\alpha},\left|\partial^{\alpha}H_{w}^{-1}\right|_{L^{\infty}}<C_{\alpha},\quad\left|\partial^{\alpha}\varUpsilon_{w}\right|_{L^{\infty}}<C_{\alpha}.\label{eq:estimate_H_Upsilon}
\end{equation}
that express some regularity of the functions $H_{w},\varUpsilon_{w}$.
\end{enumerate}
\end{thm}

}}}\end{center}
\begin{rem}
Note that for a given operator $\mathcal{L}_{w_{0},w_{1}}$, the word
$w\in\mathcal{W}$ appearing in the conjugation $\widehat{\mathcal{T}}_{w}$
can be an arbitrary extension of $\left(w_{0},w_{1}\right)$. This
freedom for the choice of extension will be used from time to time
in the sequel. When necessary, we will check that the choice does
give bounded or negligible corrections, see e.g. Lemma \ref{lem:bounded_variation}.
The right hand side in (\ref{eq:conj_F01}) is acting on functions
with support in a neighborhood of $x_{w}$ that contains $K\cap I_{w_{0}}$.
This is enough for us.
\end{rem}

\begin{center}
\begin{figure}[h]
\begin{centering}
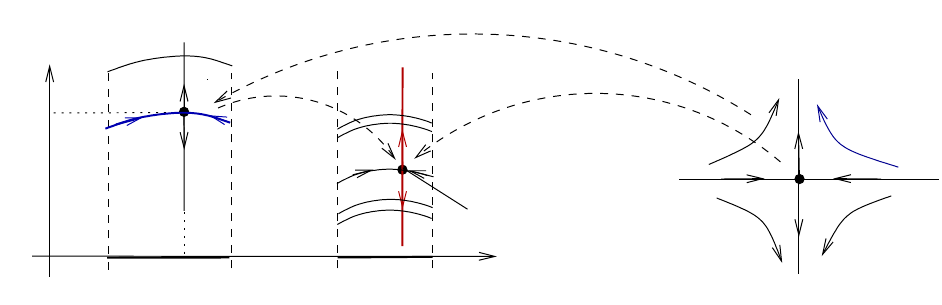
\par\end{centering}
\caption{\label{fig:The-canonical-map}According to (\ref{eq:diagram_S-1})
and (\ref{eq:point_in_trapped_K}), a word $w\in\mathcal{W}$ is associated
to a point $\mathcal{S}\left(w\right)=\left(x_{w},\xi_{w}\right)\in\mathcal{K}$.
By the canonical map $\tilde{\phi}$, this point is sent to $\tilde{\phi}\left(\mathcal{S}\left(w\right)\right)=\mathcal{S}\left(L\left(w\right)\right)$.
In fact, in a vicinity of $\mathcal{S}\left(w\right)$ the map $\tilde{\phi}$
is conjugated to the dilation map: $\tilde{\phi}_{w_{0},w_{1}}=\mathcal{T}_{L\left(w\right)}\circ D_{J_{w_{0},w_{1}}\left(x_{w}\right)}\circ\mathcal{T}_{w}^{-1}$.
Eq.(\ref{eq:conj_F01}) shows that this conjugation is also true for
the component $\mathcal{L}_{w_{0},w_{1}}$ of the transfer operator.
The stable (respectively unstable) manifold of the point $\mathcal{S}\left(w\right)$
(in blue, respect. in red) supports a Lagrangian state $\mathcal{S}_{w}^{\left(k\right)}$
(respect. $\mathcal{U}_{w}^{\left(k\right)}$) that are defined in
Theorem \ref{thm:6.6} and used to express the transfer operator $\mathcal{L}_{\nu}^{n}$
for large time, as an asymptotic expansion.}
\end{figure}
\par\end{center}
\begin{rem}
\label{rem:5.3}A consequence of (\ref{eq:estimate_H_Upsilon}) is
that for any $\chi\in C_{0}^{\infty}\left(I\right)$, $m\in\mathbb{R}$,
\begin{eqnarray}
\widehat{\mathcal{T}}_{w}\hat{\chi} & :H_{\nu}^{-m}\left(\mathbb{R}\right) & \rightarrow H_{\nu}^{-m}\left(\mathbb{R}\right),\label{eq:T_bounded}\\
\widehat{\mathcal{T}}_{w}^{-1}\hat{\chi} & H_{\nu}^{-m}\left(\mathbb{R}\right) & \rightarrow H_{\nu}^{-m}\left(\mathbb{R}\right),\nonumber 
\end{eqnarray}
 are bounded uniformly with respect to $w\in\mathcal{W}$ and $\nu\in\mathbb{R}$.
\end{rem}

The next Corollary uses Theorem \ref{thm:Global-normal-form.} and
iterations of it to express long time evolution.

\begin{center}{\color{blue}\fbox{\color{black}\parbox{16cm}{
\begin{cor}
\label{cor:5.4}For any $n\geq1$,
\begin{equation}
\mathcal{L}_{\nu}^{n}=\sum_{w_{0,n}}\mathcal{L}_{w_{0,n}}.\label{eq:cor_F^n}
\end{equation}
For each term $\mathcal{L}_{w_{0,n}}$, let $w\in\mathcal{W}$ be
an arbitrary extension of $w_{0,n}$. We can write
\begin{equation}
\mathcal{L}_{w_{0,n}}=e^{i\nu\tau_{w_{0,n}}\left(x_{w}\right)+V_{w_{0,n}}\left(x_{w}\right)}\widehat{\mathcal{T}}_{L^{n}\left(w\right)}\hat{D}_{J_{w_{0,n}}\left(x_{w}\right)}\widehat{\mathcal{T}}_{w}^{-1},\qquad w\in\mathcal{W}\label{eq:conjug_2}
\end{equation}
with $V_{w_{0,n}},\tau_{w_{0,n}}$ and $J_{w_{0,n}}$ being Birkhoff
sums defined as in (\ref{eq:def_Birkhoff_sum}).
\end{cor}

}}}\end{center}
\begin{proof}
of Theorem \ref{thm:Global-normal-form.}. For simplicity we define
\begin{equation}
\mathcal{V}\left(x\right):=\tau\left(x\right)-\frac{i}{\nu}V\left(x\right).\label{eq:def_cal_V}
\end{equation}
Let us denote $\hat{y}:\varphi\left(y\right)\rightarrow y\varphi\left(y\right)$
the multiplication operator by $y$. The operator $\mathcal{L}_{w_{0},w_{1}}$
in (\ref{e:def_F_op_ij}) can be written:
\begin{equation}
\mathcal{L}_{w_{0},w_{1}}=e^{i\nu\mathcal{V}\left(\hat{y}\right)}\mathcal{L}_{\phi_{w_{0},w_{1}}}\label{eq:expression_F}
\end{equation}
with $\mathcal{L}_{\phi_{w_{0},w_{1}}}\varphi:=\varphi\circ\phi_{w_{0},w_{1}}^{-1}$.
The aim is to transform progressively (\ref{eq:expression_F}) into
the expression (\ref{eq:conj_F01}). For the first step we write
\begin{eqnarray*}
\mathcal{L}_{w_{0},w_{1}} & = & e^{i\nu\mathcal{V}\left(x_{L\left(w\right)}\right)}e^{i\nu\left(\mathcal{V}\left(\hat{y}\right)-\mathcal{V}\left(x_{L\left(w\right)}\right)\right)}\mathcal{L}_{\phi_{w_{0},w_{1}}}\\
 & = & e^{i\nu\mathcal{V}\left(x_{L\left(w\right)}\right)}\mathcal{L}_{\phi_{w_{0},w_{1}}}e^{i\nu\left(\mathcal{V}\left(\phi_{w_{0},w_{1}}\left(\hat{y}\right)\right)-\mathcal{V}\left(x_{L\left(w\right)}\right)\right)}
\end{eqnarray*}
For any $y\in I_{w_{0}}$ and $k\geq1$ we have $\left|\phi{}_{w_{0,k}}\left(y\right)-x_{L^{k}\left(w\right)}\right|\leq C.e^{-kJ_{\mathrm{min}}}$
with some $C>0$ independent on $w$ and $y$. Hence 
\begin{equation}
\varUpsilon_{w}\left(y\right):=-\sum_{k\geq0}\left(\mathcal{V}\left(\phi{}_{w_{0,k+1}}\left(y\right)\right)-\mathcal{V}\left(x_{L^{k+1}\left(w\right)}\right)\right)\label{eq:def_Sx,w}
\end{equation}
defines a smooth complex valued function with regularity estimate
given in (\ref{eq:estimate_H_Upsilon}). We have also $\varUpsilon_{w}\left(x_{w}\right)\underset{(\ref{eq:def_Sx,w})}{=}0$
and 
\[
\frac{d}{dy}\varUpsilon_{w}\left(y\right)\underset{(\ref{eq:def_Sx,w})}{=}-\sum_{k\geq0}\mathcal{V}'\left(\phi{}_{w_{0,k+1}}\left(y\right)\right)\cdot\phi'{}_{w_{0,k+1}}\left(y\right)
\]
so $\frac{d}{dy}\varUpsilon_{w}^{\left(0\right)}\left(y\right)\underset{(\ref{eq:expression_zeta_w})}{=}\zeta_{w_{+}}\left(y\right)$.
The family of functions $\varUpsilon_{w}$ solves the homological
equation
\begin{eqnarray*}
\varUpsilon_{L\left(w\right)}\left(\phi_{w_{0},w_{1}}\left(y\right)\right) & = & -\sum_{k\geq0}\left(\mathcal{V}\left(\phi{}_{w_{0,k+2}}\left(y\right)\right)-\mathcal{V}\left(x_{L^{k+2}\left(w\right)}\right)\right)\\
 & = & -\sum_{k\geq1}\left(\mathcal{V}\left(\phi{}_{w_{0,k+1}}\left(y\right)\right)-\mathcal{V}\left(x_{L^{k+1}\left(w\right)}\right)\right)\\
 & = & \varUpsilon_{w}\left(y\right)+\left(\mathcal{V}\left(\phi_{w_{0},w_{1}}\left(y\right)\right)-\mathcal{V}\left(x_{L\left(w\right)}\right)\right).
\end{eqnarray*}
Therefore we get
\begin{eqnarray*}
\mathcal{L}_{w_{0},w_{1}} & = & e^{i\nu\mathcal{V}\left(x_{L\left(w\right)}\right)}\mathcal{L}_{\phi_{w_{0},w_{1}}}e^{i\nu\left(\mathcal{V}\left(\phi_{w_{0},w_{1}}\left(\hat{y}\right)\right)-\mathcal{V}\left(x_{L\left(w\right)}\right)\right)}\\
 & = & e^{i\nu\mathcal{V}\left(x_{L\left(w\right)}\right)}\mathcal{L}_{\phi_{w_{0},w_{1}}}e^{i\nu\left(\varUpsilon_{L\left(w\right)}\left(\phi_{w_{0},w_{1}}\left(\hat{y}\right)\right)-\varUpsilon_{w}\left(\hat{y}\right)\right)}\\
 & = & e^{i\nu\mathcal{V}\left(x_{L\left(w\right)}\right)}e^{+i\nu\varUpsilon_{L\left(w\right)}\left(\hat{y}\right)}\mathcal{L}_{\phi_{w_{0},w_{1}}}e^{-i\nu\varUpsilon_{w}\left(\hat{y}\right)}.
\end{eqnarray*}
For the second step we write
\begin{eqnarray}
\mathcal{L}_{w_{0},w_{1}} & = & e^{i\nu\mathcal{V}\left(x_{L\left(w\right)}\right)}e^{+i\nu\varUpsilon_{L\left(w\right)}\left(\hat{y}\right)}\hat{T}_{x_{L\left(w\right)}}\mathcal{L}_{f_{0,1}}\hat{T}_{-x_{w}}e^{-i\nu\varUpsilon_{w}\left(\hat{y}\right)}\label{eq:express_F_w0w1}
\end{eqnarray}
with\footnote{Beware that $f_{0,1}$ depends on the full word $w$.}
\[
f_{0,1}\left(z\right):=\phi_{w_{0},w_{1}}\left(z+x_{w}\right)-x_{L\left(w\right)},
\]
satisfying $f_{0,1}\left(0\right)=0$ and $f'_{0,1}\left(0\right)=\phi'_{w_{0},w_{1}}\left(x_{w}\right)=e^{-J_{w_{0},w_{1}}\left(x_{w}\right)}$
with $J_{w_{0},w_{1}}\left(x_{w}\right):=-\log\phi'_{w_{0},w_{1}}\left(x_{w}\right)$.
For the third step, as shown in \cite[th7,p.45]{nelson_69}, there
exists a family of smooth functions $H_{w}:\mathcal{J}\rightarrow\mathbb{R}$
defined on $\mathcal{J}\subset\mathbb{R}$ a sufficiently small neighborhood
of the origin\footnote{This is possible because the points $x_{w}$ are bounded away from
the boundary of $I$, uniformly with respect to the words $w\in\mathcal{W}$. } and satisfying $H_{w}\left(0\right)=0$, $H'_{w}\left(0\right)=1$
and 
\begin{equation}
\forall z\in\mathcal{J},\quad H_{L\left(w\right)}\left(e^{-J_{w_{0},w_{1}}\left(x_{w}\right)}z\right)=f_{0,1}\left(H_{w}\left(z\right)\right)\label{eq:conj_with_H}
\end{equation}
which gives
\begin{equation}
\mathcal{L}_{f_{0,1}}=\mathcal{L}_{H_{L\left(w\right)}}\circ\hat{D}_{J_{w_{0},w_{1}}\left(x_{w}\right)}\circ\mathcal{L}_{H_{w}}^{-1}.\label{eq:conj_L_Hw}
\end{equation}
In other words, the contracting map $f_{0,1}$ is ``globally conjugated''
to the linear contracting map $z\rightarrow e^{-J_{w_{0},w_{1}}\left(x_{w}\right)}z$.
The functions $H_{w}$ can be constructed by the ''scattering process''\footnote{The term ``scattering process'' comes from \cite{nelson_69}. Here
in the ``non interacting region'' is $z\rightarrow0$ whereas in
the usual theory of scattering of waves, the non interacting region
is the infinity, far from the action of the potential.} as follows. For $n\geq1$ let $f_{0,n}:=f_{n-1,n}\circ f_{n-2,n-1}\ldots\circ f_{0,1}$.
For $z\in\mathbb{R}$ (closed enough to $0$) let
\[
H_{w}^{\left(n\right)}\left(z\right):=f_{0,n}^{-1}\left(e^{-J_{w_{0,n}}\left(x_{w}\right)}z\right).
\]
As $n\rightarrow\infty$, the uniform convergence of $H_{w}^{\left(n\right)}$
and its derivatives can be obtained using bounded distortion estimates
\cite[prop 4.2]{Falconer_97}. This gives the existence of the limit
\[
H_{w}\left(z\right):=\lim_{n\rightarrow\infty}H_{w}^{\left(n\right)}\left(z\right).
\]
We also get (\ref{eq:estimate_H_Upsilon}) from bounded distortion
estimates. Then

\begin{eqnarray*}
H_{L\left(w\right)}^{\left(n-1\right)}\left(e^{-J_{w_{0},w_{1}}\left(x_{w}\right)}.z\right) & = & f_{1,n}^{-1}\left(e^{-J_{w_{0,n}}\left(x_{w}\right)}z\right)\\
 & = & f_{0,1}\left(f_{0,n}^{-1}\left(e^{-J_{w_{0,n}}\left(x_{w}\right)}z\right)\right)\\
 & = & f_{0,1}\left(H_{w}^{\left(n\right)}\left(z\right)\right).
\end{eqnarray*}
Taking the limit $n\rightarrow\infty$ and noting $L\left(w\right)=\left(w_{1},w_{2},\ldots\right)$,
we get (\ref{eq:conj_with_H}). With (\ref{eq:express_F_w0w1}) and
(\ref{eq:conj_L_Hw}) we have obtained that:
\begin{eqnarray*}
\mathcal{L}_{w_{0},w_{1}} & = & e^{i\nu\mathcal{V}\left(x_{L\left(w\right)}\right)}e^{+i\nu\varUpsilon_{L\left(w\right)}\left(\hat{y}\right)}\hat{T}_{x_{L\left(w\right)}}\mathcal{L}_{H_{L\left(w\right)}}\circ\hat{D}_{J_{w_{0},w_{1}}\left(x_{w}\right)}\circ\mathcal{L}_{H_{w}}^{-1}\hat{T}_{-x_{w}}e^{-i\nu\varUpsilon_{w}\left(\hat{y}\right)}\\
 & = & e^{i\nu\mathcal{V}\left(x_{L\left(w\right)}\right)}\widehat{\mathcal{T}}_{L\left(w\right)}\hat{D}_{J_{w_{0},w_{1}}\left(x_{w}\right)}\widehat{\mathcal{T}}_{w}^{-1}
\end{eqnarray*}
This is (\ref{eq:conj_F01}).
\end{proof}

\section{Asymptotic expansion\label{sec:Asymptotic-expansion}}

In this section we first give a simple but useful expansion for the
dilation operator defined in (\ref{eq:def_D_lambda}) in terms of
rank one operators in Theorem \ref{thm:Dilation_expansion}. Then
we use this expansion and the global normal form (\ref{eq:conjug_2})
to deduce an expansion for the transfer operator $\mathcal{L}_{\nu}^{n}$
for large time $n$ in Theorem \ref{thm:6.6}.

\subsection{\label{subsec:Asymptotic-expansion-for}Asymptotic expansion for
the dilation operator}

Fix $\lambda_{0}>0$ and let $\lambda\geq\lambda_{0}$, let $y_{0}>0$,
$\varphi\in C_{0}^{\infty}\left(]-y_{0},y_{0}[\right)$. Recall from
(\ref{eq:def_D_lambda}) that we defined $\left(\hat{D}_{\lambda}\varphi\right)\left(y\right):=\varphi\left(e^{\lambda}y\right)$.
Let $\chi_{0}\in C_{0}^{\infty}\left(\mathbb{R}\right)$ such that
$\mathrm{supp}\chi_{0}\subset\left[-y_{0},y_{0}\right]$ and $\chi_{0}\left(x\right)=1$
for $x\in\left[-e^{-\lambda_{0}}y_{0},e^{-\lambda_{0}}y_{0}\right]$
so that $\chi_{0}\equiv1$ on $\mathrm{supp}\left(\hat{D}_{\lambda}\chi_{0}\right)$.
Hence we have $\hat{\chi}_{0}^{-1}\circ\hat{D}_{\lambda}\circ\hat{\chi}_{0}=\hat{D}_{\lambda}\circ\hat{\chi}_{0}$
where $\hat{\chi}_{0}$ is the multiplication operator by $\chi_{0}$.

For $k\geq0$ let us denote $\delta^{\left(k\right)}$ for the $k$-th
derivative of the Dirac distribution. Let $\langle x^{k},\hat{\chi}_{0}\psi\rangle:=\int_{\mathbb{R}}x^{k}\chi_{0}\left(x\right)\psi\left(x\right)dx$.
We introduce the rank one operator\footnote{$\mathcal{S}\left(\mathbb{R}\right)$ is the Schwartz space \cite[p.197]{Taylor_book_PDO}.}
\begin{equation}
\Pi_{k}:\psi\in\mathcal{S}\left(\mathbb{R}\right)\rightarrow\frac{1}{k!}\langle x^{k},\hat{\chi}_{0}\psi\rangle\delta^{\left(k\right)}\in\mathcal{S}'\left(\mathbb{R}\right)\label{eq:def_Pi_k}
\end{equation}
We will use the Dirac notations of physics and write
\[
\langle x^{k}|:\psi\in C_{0}^{\infty}\left(\mathbb{R};\mathbb{C}\right)\rightarrow\langle x^{k},\psi\rangle=\int_{\mathbb{R}}x^{k}\psi\left(x\right)dx\in\mathbb{C}
\]

\begin{equation}
\Pi_{k}=|\frac{1}{k!}\delta^{\left(k\right)}\rangle\langle x^{k}|\hat{\chi}_{0}\label{eq:Pi_k_dirac_notation}
\end{equation}

\noindent\fcolorbox{blue}{white}{\begin{minipage}[t]{1\columnwidth - 2\fboxsep - 2\fboxrule}%
\begin{lem}
If $k<m-\frac{1}{2}$ then $\Pi_{k}:H_{\nu}^{-m}\left(\mathbb{R}\right)\rightarrow H_{\nu}^{-m}\left(\mathbb{R}\right)$
is a bounded operator and
\begin{equation}
\exists C>0,\forall\nu>0,\quad\left\Vert \Pi_{k}\right\Vert _{H_{\nu}^{-m}}\leq C\nu^{k+1/2}.\label{eq:norm_Pi_k}
\end{equation}
Furthermore 
\begin{equation}
\exists C>0,\forall m,\nu,\qquad\left\Vert \hat{\chi}_{0}\right\Vert {}_{H_{\nu}^{-m}}\leq C.\label{eq:norm_Chi_0}
\end{equation}
\end{lem}

\end{minipage}}
\begin{proof}
In order to prove (\ref{eq:norm_Chi_0}) we use $\nu$-semiclassical
calculus: we have $\hat{\chi}_{0}=\mathrm{Op}_{\nu}\left(\chi_{0}\right)$
and using composition of PDO as well as $L^{2}$-continuity \cite{zworski_book_2012},
\[
\left\Vert \hat{\chi}_{0}\right\Vert {}_{H_{\nu}^{-m}}=\left\Vert \mathrm{Op}_{\nu}\left(A_{m}\right)\mathrm{Op}_{\nu}\left(\chi_{0}\right)\mathrm{Op}_{\nu}\left(A_{m}^{-1}\right)\right\Vert _{L^{2}\left(\mathbb{R}\right)}=\left\Vert \mathrm{Op}_{\nu}\left(\chi_{0}\right)\right\Vert _{L^{2}\left(\mathbb{R}\right)}+O\left(\nu^{-1}\right)\leq C.
\]
We will use (\ref{eq:norm_Chi_0}) later in (\ref{eq:norm_sum}).

One has for $2\left(k-m\right)<-1$,
\begin{eqnarray}
\left\Vert \delta^{(k)}\right\Vert _{H_{\nu}^{-m}} & = & \left\Vert \langle\xi\rangle^{-m}\mathcal{F}_{\nu}[\delta^{(k)}]\right\Vert _{L^{2}}=\left\Vert \langle\xi\rangle^{-m}(2\pi/\nu)^{-1/2}\left(i\nu\xi\right)^{k}\right\Vert _{L^{2}}\nonumber \\
 & \leq & C\nu^{k+1/2}\label{eq:bound_Dirac_k}
\end{eqnarray}
and 
\[
\left\Vert x^{k}\chi_{0}(x)\right\Vert _{H_{\nu}^{m}}=\nu^{1/2}\left\Vert \langle\xi\rangle^{m}\mathcal{F}_{1}[x^{k}\chi_{0}(x)]\left(\nu\xi\right)\right\Vert _{L^{2}}=\|\underbrace{\langle\xi/\nu\rangle^{m}}_{\leq\langle\xi\rangle^{m}}\mathcal{F}_{1}[x^{k}\chi_{0}(x)]\left(\xi\right)\|_{L^{2}}\leq C
\]

This gives 
\begin{equation}
\left\Vert \Pi_{k}\right\Vert _{H_{\nu}^{-m}}=\left\Vert |\frac{1}{k!}\delta^{(k)}\rangle\langle\chi_{0}(x)x^{k}|\right\Vert _{H_{\nu}^{-m}}\leq C\nu^{k+1/2}\label{eq:bound_Pi_k}
\end{equation}
\end{proof}
\begin{rem}
We have
\[
\Pi_{k}\circ\Pi_{k'}\underset{(\ref{eq:Pi_k_dirac_notation})}{=}\delta_{k,k'}\Pi_{k},\qquad\left[\hat{D}_{\lambda}\circ\hat{\chi}_{0},\Pi_{k}\right]=0,\quad\hat{D}_{\lambda}\circ\hat{\chi}_{0}\circ\Pi_{k}=e^{-\left(k+1\right)\lambda}\Pi_{k},
\]
i.e. the operators $\left(\Pi_{k}\right)_{k}$ form a set of commuting
spectral projection operators for the operator $\hat{D}_{\lambda}\circ\hat{\chi}_{0}$.
The next Theorem shows that they are complete in the sense that they
give a spectral decomposition of $\hat{D}_{\lambda}\circ\hat{\chi}_{0}$
up to some remainder with small norm.
\end{rem}

\begin{center}{\color{blue}\fbox{\color{black}\parbox{16cm}{
\begin{thm}
\label{thm:Dilation_expansion}Fix $\lambda_{0}>0$ and $\nu_{0}>0.$
Let $m>\frac{1}{2}$, $\nu>\nu_{0}$, $\lambda\geq\lambda_{0}$ and
$\chi_{0}$ as above. The operator $\hat{D}_{\lambda}\circ\hat{\chi}_{0}:H_{\nu}^{-m}\rightarrow H_{\nu}^{-m}$
is bounded and for any $d<m-\frac{3}{2}$, and we have
\begin{equation}
\left\Vert \hat{D}_{\lambda}\circ\hat{\chi}_{0}-\sum_{k=0}^{d}e^{-\left(k+1\right)\lambda}\Pi_{k}\right\Vert _{H_{\nu}^{-m}\left(\mathbb{R}\right)}\leq Ce^{-\lambda/2}\left(e^{-\lambda}\nu\right)^{\left(d+1\right)+\frac{1}{2}}\label{eq:expansion_D}
\end{equation}
where the constant $C$ in the remainder does not depend on $\lambda,\nu$
but depends on $d$,$\chi_{0}$ and $m$.
\end{thm}

}}}\end{center}
\begin{rem}
In higher dimensions, there is a similar result in \cite[Prop. 4.19]{faure-tsujii_prequantum_maps_12}.
Formula (\ref{eq:expansion_D}) can be considered as a Taylor expansion
of the dilation operator. From (\ref{eq:norm_Pi_k}) we have $e^{-\left(k+1\right)\lambda}\left\Vert \Pi_{k}\right\Vert _{H_{\nu}^{-m}\left(\mathbb{R}\right)}=O\left(e^{-\lambda/2}\left(e^{-\lambda}\nu\right)^{k+1/2}\right)$,
therefore if $e^{\lambda}>\nu$, (\ref{eq:expansion_D}) can be interpreted
as an asymptotic expansion in powers of $e^{-\lambda}\nu$.
\end{rem}

\begin{proof}
We directly see that 
\begin{equation}
\hat{D}_{\lambda}\left(\hat{\chi}_{0}-\sum\limits _{k=0}^{d}\Pi_{k}\right)=\left(\hat{D}_{\lambda}\hat{\chi}_{0}-\sum\limits _{k=0}^{d}e^{-\left(k+1\right)\lambda}\Pi_{k}\right)\label{eq:RedToLambdaEq1}
\end{equation}
and first prove the following Lemma.

\noindent\fcolorbox{blue}{white}{\begin{minipage}[t]{1\columnwidth - 2\fboxsep - 2\fboxrule}%
\begin{lem}
\label{lem:FTOfExp} Let $\phi\in C_{0}^{\infty}$ then 
\begin{equation}
\mathcal{F}_{\nu}\left[\left(\hat{\chi}_{0}-\sum\limits _{k=0}^{d}\Pi_{k}\right)\phi\right](\xi)=\sqrt{\nu}\left(\tilde{\Psi}\left(\nu\xi\right)-\sum\limits _{k=0}^{d}\frac{1}{k!}\left(\nu\xi\right)^{k}\tilde{\Psi}^{(k)}(0)\right)\label{eq:FTOfExp}
\end{equation}
where $\tilde{\Psi}:=\mathcal{F}_{1}\left[\chi_{0}\phi\right]$.
\end{lem}

\end{minipage}}
\begin{proof}
First we consider 
\[
\tilde{\Psi}^{(k)}(\xi)=\left(\frac{d}{d\xi}\right)^{k}\frac{1}{\sqrt{2\pi}}\int e^{-ix\xi}\chi_{0}(x)\phi(x)dx=\frac{1}{\sqrt{2\pi}}\int e^{ix\xi}(-ix)^{k}\chi_{0}(x)\phi(x)dx
\]
and conclude that 
\begin{equation}
\tilde{\Psi}^{(k)}(0)=\frac{(-i)^{k}}{\sqrt{2\pi}}\langle x^{k},\hat{\chi}_{0}\phi\rangle.\label{eq:lem1}
\end{equation}
Furthermore we calculate 
\begin{equation}
\mathcal{F}_{\nu}\left[\delta^{(k)}\right](\xi)=(2\pi/\nu)^{-1/2}\left(-i\nu\xi\right)^{k}\label{eq:lem2}
\end{equation}
and 
\begin{equation}
\mathcal{F}_{\nu}\left[\hat{\chi}_{0}\phi\right](\xi)=\nu^{1/2}\tilde{\Psi}\left(\nu\xi\right).\label{eq:lem3}
\end{equation}
Finally (\ref{eq:Pi_k_dirac_notation}) with (\ref{eq:lem1}), (\ref{eq:lem2})
and (\ref{eq:lem3}) give (\ref{eq:FTOfExp}) which finishes the proof
of Lemma \ref{lem:FTOfExp}.
\end{proof}
We can now return to the proof of Theorem \ref{thm:Dilation_expansion}.
We have $\mathcal{F}_{\nu}[\hat{D}_{\lambda}\Psi](\xi)=e^{-\lambda}\mathcal{F}_{\nu}[\Psi]\left(e^{-\lambda}\xi\right)$
which follows directly from variable substitution. Using Taylor's
Theorem we have
\begin{eqnarray}
\mathcal{F}_{\nu}\left[\left(\hat{D}_{\lambda}\hat{\chi}_{0}-\sum\limits _{k=0}^{d}e^{-\left(k+1\right)\lambda}\Pi_{k}\right)\phi\right](\xi) & \underset{(\ref{eq:RedToLambdaEq1})}{=} & e^{-\lambda}\mathcal{F}_{\nu}\left[\left(\hat{\chi}_{0}-\sum\limits _{k=0}^{d}\Pi_{k}\right)\phi\right]\left(e^{-\lambda}\xi\right)\nonumber \\
\underset{(\ref{eq:FTOfExp})}{=}\sqrt{\nu}e^{-\lambda} &  & \left(\tilde{\Psi}\left(\nu e^{-\lambda}\xi\right)-\sum\limits _{k=0}^{d}\frac{1}{k!}\left(\nu e^{-\lambda}\xi\right)^{k}\tilde{\Psi}^{(k)}(0)\right)\nonumber \\
=\sqrt{\nu}e^{-\lambda} & \frac{1}{(d+1)!} & \left(\nu e^{-\lambda}\xi\right)^{d+1}\tilde{\Psi}^{(d+1)}\left(\theta_{\nu e^{-\lambda}\xi}\nu e^{-\lambda}\xi\right)\label{eq:FTOfExpDl}
\end{eqnarray}
with $|\theta_{\nu e^{-\lambda}\xi}|\leq1$. By definition of $H_{\nu}^{-m}$
in (\ref{eq:def_norm_Hm}) we have
\begin{eqnarray*}
\left(*\right): & = & \left\Vert \left(D_{\lambda}\hat{\chi}_{0}-\sum\limits _{k=0}^{d}e^{-\left(k+1\right)\lambda}\Pi_{k}\right)\phi\right\Vert _{H_{\nu}^{-m}}\\
 & = & \left\Vert \langle\xi\rangle^{-m}\mathcal{F}_{\nu}\left[\left(D_{\lambda}\hat{\chi}_{0}-\sum\limits _{k=0}^{d}e^{-\left(k+1\right)\lambda}\Pi_{k}\right)\phi\right](\xi)\right\Vert _{L^{2}}\\
 & \underset{(\ref{eq:FTOfExpDl})}{=} & \left\Vert \langle\xi\rangle^{-m}\sqrt{\nu}e^{-\lambda}\frac{1}{(d+1)!}\left(\nu e^{-\lambda}\xi\right)^{d+1}\tilde{\Psi}^{(d+1)}\left(\theta_{\nu e^{-\lambda}\xi}\nu e^{-\lambda}\xi\right)\right\Vert _{L^{2}}\\
 & = & e^{-\lambda/2}\left(\nu e^{-\lambda}\right)^{d+3/2}\frac{1}{(d+1)!}\sqrt{\int_{\mathbb{R}}\langle\xi\rangle^{-2m}\xi^{2(d+1)}\left|\tilde{\Psi}^{(d+1)}\left(\theta_{\nu e^{-\lambda}\xi}\nu e^{-\lambda}\xi\right)\right|^{2}d\xi}
\end{eqnarray*}

To finish the proof of Theorem \ref{thm:Dilation_expansion} we have
to show that 
\[
(*)\leq C\|\phi\|_{H_{\nu}^{-m}}e^{-\lambda/2}\left(\nu e^{-\lambda}\right)^{d+3/2}
\]
with $C$ independent of $\phi$, $\lambda$, $\nu$. We decompose
the integral over $\mathbb{R}$ under the square root above into 
\[
\int\limits _{\mathbb{R}}\dots d\xi=\underbrace{\int\limits _{[-e^{\lambda},e^{\lambda}]}\dots d\xi}_{(A)}+\underbrace{\int\limits _{\mathbb{R}\setminus[-e^{\lambda},e^{\lambda}]}\dots d\xi}_{(B)}
\]
and treat them separately. We have: 
\begin{eqnarray*}
(A) & \leq & \left(\int\limits _{\mathbb{R}}\langle\xi\rangle^{-2m}\xi^{2(d+1)}d\xi\right)\max\limits _{\xi\in[-e^{\lambda},e^{\lambda}]}\left|\tilde{\Psi}^{(d+1)}\left(\theta_{\nu e^{-\lambda}\xi}\nu e^{-\lambda}\xi\right)\right|^{2}\\
 & \leq & \left(\int\limits _{\mathbb{R}}\langle\xi\rangle^{-2m}\xi^{2(d+1)}d\xi\right)\max\limits _{\xi\in[-1,1]}\left|\tilde{\Psi}^{(d+1)}\left(\theta_{\nu\xi}\nu\xi\right)\right|^{2}\\
 & \underset{(\ref{eq:ineq-1})}{\leq} & C\|\phi\|_{H_{\nu}^{-m}}^{2}.
\end{eqnarray*}
Recall that we have assumed that 
\begin{equation}
d<m-\frac{3}{2}\label{eq:d_m_3/2}
\end{equation}
 for the convergence of the integral. For the second term $\left(B\right)$
we first observe that 
\[
\exists C>0,\forall\xi\in\mathbb{R}\setminus\left[-1,1\right],\frac{\left(e^{-\lambda}\left\langle \tfrac{\xi}{e^{-\lambda}}\right\rangle \right)^{-2m}}{\langle\xi\rangle^{-2m}}\leq C,
\]
with $C$ that depends on $m$ but is independent of $\lambda>0$.
Also taking $\lambda\rightarrow0$ in (\ref{eq:FTOfExpDl}) we have
\begin{equation}
\left|\mathcal{F}_{\nu}\left[\left(\hat{\chi}_{0}-\sum\limits _{k=0}^{d}\Pi_{k}\right)\phi\right]\left(\xi\right)\right|=\frac{\sqrt{\nu}}{\left(d+1\right)!}\left(\nu\xi\right)^{d+1}\left|\tilde{\Psi}^{(d+1)}\left(\theta_{\nu\xi}\nu\xi\right)\right|.\label{eq:F_hh}
\end{equation}
If $d<m-1/2$ then
\begin{equation}
\exists C>0,\forall\nu>0,\quad\left\Vert \left(\hat{\chi}_{0}-\sum\limits _{k=0}^{d}\Pi_{k}\right)\right\Vert _{H_{\nu}^{-m}}\underset{(\ref{eq:norm_Pi_k}),(\ref{eq:norm_Chi_0})}{\leq}C\nu^{d+1/2}.\label{eq:norm_sum}
\end{equation}

Hence
\begin{eqnarray*}
(B) & = & e^{\lambda}\int\limits _{\mathbb{R}\setminus[-1,1]}\langle e^{\lambda}\xi\rangle^{-2m}(e^{\lambda}\xi)^{2(d+1)}\left|\tilde{\Psi}^{(d+1)}\left(\theta_{\nu\xi}\nu\xi\right)\right|^{2}d\xi\\
 & = & e^{-2(m-d-3/2)\lambda}\nu^{-2(d+1)}\\
 &  & \int\limits _{\mathbb{R}\setminus[-1,1]}\frac{\left(\langle e^{\lambda}\xi\rangle e^{-\lambda}\right)^{-2m}}{\langle\xi\rangle^{-2m}}\langle\xi\rangle^{-2m}\left(\nu\xi\right)^{2\left(d+1\right)}\left|\tilde{\Psi}^{(d+1)}\left(\theta_{\nu\xi}\nu\xi\right)\right|^{2}d\xi\\
 & \underset{(\ref{eq:F_hh})}{\leq} & Ce^{-2(m-d-3/2)\lambda}\nu^{-2(d+1)}\left(\nu^{-1/2}\left(d+1\right)!\right)^{2}\left\Vert \left(\hat{\chi}_{0}-\sum\limits _{k=0}^{d}\Pi_{k}\right)\phi\right\Vert _{H_{\nu}^{-m}}^{2}\\
 & \underset{(\ref{eq:d_m_3/2}),(\ref{eq:norm_sum})}{\leq} & C\nu^{-2}\|\phi\|_{H_{\nu}^{-m}}^{2}
\end{eqnarray*}

Thus taking (A) and (B) together we obtain 
\[
(*)\leq C\|\phi\|_{H_{\nu}^{-m}}e^{-\lambda/2}\left(\nu e^{-\lambda}\right)^{d+3/2}
\]
which finishes the proof of Theorem \ref{thm:Dilation_expansion}.
\end{proof}
The next Lemma has be used in the previous proof.

\noindent\fcolorbox{blue}{white}{\begin{minipage}[t]{1\columnwidth - 2\fboxsep - 2\fboxrule}%
\begin{lem}
\label{lem:maxPsi}For every $\chi_{0}\in C_{0}^{\infty}\left(\mathbb{R}\right)$,
$d,m\geq0$ there exists $C>0$ such that for any $\phi\in C_{0}^{\infty}\left(\mathbb{R}\right)$,
\begin{equation}
\max\limits _{\xi\in[-1,1]}\left|\tilde{\Psi}^{(d+1)}\left(\nu\xi\right)\right|^{2}\leq C\|\phi\|_{H_{\nu}^{-m}}^{2}\label{eq:ineq-1}
\end{equation}
with $\tilde{\Psi}:=\mathcal{F}_{1}\left[\chi_{0}\phi\right]$ and
where $C$ depends only on $\chi_{0}$,$d$ and $m$.
\end{lem}

\end{minipage}}
\begin{proof}
By definition of $\tilde{\Psi}$ we have 
\[
\tilde{\Psi}^{(d+1)}\left(\xi\right)=\frac{1}{\sqrt{2\pi}}\int e^{-ix\xi}(-ix)^{d+1}\chi_{0}(x)\phi(x)dx
\]
thus 
\begin{eqnarray*}
\tilde{\Psi}^{(d+1)}\left(\nu\xi\right) & = & \frac{1}{\sqrt{2\pi}}\left\langle e^{ix\nu\xi}(ix)^{d+1}\chi_{0}(x),\phi\right\rangle _{L^{2}}\\
\left|\tilde{\Psi}^{(d+1)}\left(\nu\xi\right)\right| & \leq & \frac{1}{\sqrt{2\pi}}\left\Vert e^{i\nu x\xi}(ix)^{d+1}\chi_{0}(x)\right\Vert _{H_{\nu}^{m}}\left\Vert \phi\right\Vert _{H_{\nu}^{-m}}
\end{eqnarray*}
Thus we have to find for any $\xi_{0}\in[-1,1]$ an estimate of 
\begin{eqnarray*}
\left\Vert e^{i\nu x\xi_{0}}(ix)^{d+1}\chi_{0}(x)\right\Vert _{H_{\nu}^{m}} & = & \left\Vert \langle\xi\rangle^{m}\mathcal{F}_{\nu}[(ix)^{d+1}\chi_{0}(x)](\xi-\xi_{0})\right\Vert _{L^{2}}\\
 & \leq & \nu^{1/2}\left\Vert \langle\xi\rangle^{m}\mathcal{F}_{1}[(ix)^{d+1}\chi_{0}(x)]\left(\nu\left(\xi-\xi_{0}\right)\right)\right\Vert _{L^{2}}\\
 & \leq & \left\Vert \langle\xi/\nu+\xi_{0}\rangle^{m}\mathcal{F}_{1}[(ix)^{d+1}\chi_{0}(x)]\left(\xi\right)\right\Vert _{L^{2}}\\
 & \leq & C_{m}\left\Vert \langle\xi\rangle^{m}\mathcal{F}_{1}[(ix)^{d+1}\chi_{0}(x)]\left(\xi\right)\right\Vert _{L^{2}}\leq C
\end{eqnarray*}
where $C$ depends on $m$, $d$ and $\chi_{0}$.
\end{proof}

\subsection{Asymptotic expansion for the transfer operator}

In Corollary \ref{cor:5.4} we have shown that $\mathcal{L}_{\nu}^{n}$
is a sum of operators $\mathcal{L}_{w_{0,n}}$ and each of these operators
is conjugated to a dilation operator. For the next Theorem we additionally
use the asymptotic expansion in Theorem \ref{thm:Dilation_expansion}
for the dilation operator to deduce an asymptotic expansion for $\mathcal{L}_{\nu}^{n}$.
 In order to simplify the notation we will write $J_{w_{0,n}}=J_{w_{0,n}}\left(x_{w}\right)$
and $\tau_{w_{0,n}}=\tau_{w_{0,n}}\left(x_{w}\right)$ , $V_{w_{0,n}}=V_{w_{0,n}}\left(x_{w}\right)$
for the Birkhoff sums defined in (\ref{eq:def_Birkhoff_sum}) and
where $w$ is an arbitrary extension of $w_{0,n}$ as explained in
Corollary \ref{cor:5.4}. In the limit of large $n$ the bounded distortion
principle implies that the impact of the arbitrary extension becomes
small, anyway, see Lemma \ref{lem:bounded_variation}.

\noindent\fcolorbox{blue}{white}{\begin{minipage}[t]{1\columnwidth - 2\fboxsep - 2\fboxrule}%
\begin{thm}
\label{thm:6.6}For any $0\leq d<m-\frac{3}{2}$, there exists $C>0$
such that for any $n\geq1$, any $\nu>0$,
\begin{equation}
\left\Vert \mathcal{L}_{\nu,\chi}^{n}-\sum_{w_{0,n}}e^{i\nu\tau_{w_{0,n}}+V_{w_{0,n}}}\sum_{k=0}^{d}e^{-\left(k+1\right)J_{w_{0,n}}}\Pi_{k,w,n}\right\Vert _{H_{\nu}^{-m}}\leq C\nu^{\left(d+\frac{3}{2}\right)}e^{n\left(\mathrm{Pr}\left(V-\left(d+2\right)J\right)+R\left(n\right)\right)},\label{eq:expansion_of_Fn}
\end{equation}
with some function $R\left(n\right)\underset{n\rightarrow\infty}{\longrightarrow0}$
and with the rank one operators
\begin{eqnarray}
\Pi_{k,w,n}: & =|\mathcal{U}_{L^{n}\left(w\right)}^{\left(k\right)}\rangle\langle\mathcal{S}{}_{w}^{\left(k\right)}| & \quad:H_{\nu}^{-m}\left(\mathbb{R}\right)\rightarrow H_{\nu}^{-m}\left(\mathbb{R}\right),\label{eq:def_Pi_w}
\end{eqnarray}
where $w$ is an arbitrary extension of $w_{0,n}$ and where we used
the Dirac notation of Section \ref{subsec:Asymptotic-expansion-for}
for the following distributions (cf. Figure \ref{fig:The-canonical-map})
\[
|\mathcal{U}_{L^{n}\left(w\right)}^{\left(k\right)}\rangle:=\widehat{\mathcal{T}}_{L^{n}\left(w\right)}|\frac{1}{k!}\delta^{\left(k\right)}\rangle\in H_{\nu}^{-m}\left(\mathbb{R}\right),
\]
\[
\langle\mathcal{S}{}_{w}^{\left(k\right)}|:=\langle x^{k}|\widehat{\mathcal{T}}_{w}^{-1}\hat{\chi}\in H_{\nu}^{+m}\left(\mathbb{R}\right).
\]
\end{thm}

\end{minipage}}
\begin{rem}
$\mathcal{U}_{w}^{\left(k\right)}$ and $\mathcal{S}{}_{w}^{\left(k\right)}$
are WKB Lagrangian states \cite{weinstein_97}. This is a geometric
but non necessary remark.
\end{rem}

~
\begin{rem}
For $k<m-1/2$ we conclude from Remark \ref{rem:5.3} and Eq.(\ref{eq:bound_Pi_k})
that $\Pi_{k,w,n}$ is bounded by
\begin{equation}
\left\Vert \Pi_{k,w,n}\right\Vert _{H_{\nu}^{-m}}\leq C\nu^{k+\frac{1}{2}}.\label{eq:bound_Pi_tilde_k}
\end{equation}
with $C$ independent of $\nu$, $w$ and $n$.
\end{rem}

\begin{proof}
Recall from Section \ref{subsec:Escape-function} that 
\begin{equation}
\hat{Q}:=\hat{A}_{m}\hat{\chi}^{-1}\mathcal{L}_{\nu}\hat{\chi}\hat{A}_{m}^{-1}\quad:L^{2}\left(\mathbb{R}\right)\rightarrow L^{2}\left(\mathbb{R}\right).\label{eq:Q_F}
\end{equation}
is bounded. For further use, let
\begin{equation}
\langle\tilde{\mathcal{S}}_{w}^{\left(k\right)}|:=\langle\mathcal{S}{}_{w}^{\left(k\right)}|\hat{A}_{m}^{-1}=\langle x^{k}|\widehat{\mathcal{T}}_{w}^{-1}\hat{\chi}\hat{A}_{m}^{-1}\in L^{2}\left(\mathbb{R}\right),\label{eq:def_S_k_w}
\end{equation}
\begin{equation}
|\tilde{\mathcal{U}}_{w}^{\left(k\right)}\rangle:=\hat{A}_{m}|\mathcal{U}_{L^{n}\left(w\right)}^{\left(k\right)}\rangle=\hat{A}_{m}\hat{\chi}^{-1}\widehat{\mathcal{T}}_{w}|\frac{1}{k!}\delta^{\left(k\right)}\rangle\in L^{2}\left(\mathbb{R}\right),\label{eq:def_U_k_w}
\end{equation}
\begin{equation}
\tilde{\Pi}_{k,w,n}=\hat{A}_{m}\Pi_{k,w,n}\hat{A}_{m}^{-1}=\hat{A}_{m}\hat{\chi}^{-1}\widehat{\mathcal{T}}_{L^{n}\left(w\right)}\Pi_{k}\widehat{\mathcal{T}}_{w}^{-1}\hat{\chi}\hat{A}_{m}^{-1}=|\tilde{\mathcal{U}}_{L^{n}\left(w\right)}^{\left(k\right)}\rangle\langle\tilde{\mathcal{S}}_{w}^{\left(k\right)}|.\label{eq:Pi_tilde_U_S}
\end{equation}
Proving (\ref{eq:expansion_of_Fn} ) is equivalent to proving an expansion
for $\hat{Q}^{n}$ in $L^{2}\left(\mathbb{R}\right)$:
\begin{equation}
\left\Vert \hat{Q}^{n}-\sum_{w_{0,n}}e^{i\nu\tau_{w_{0,n}}+V_{w_{0,n}}}\sum_{k=0}^{d}e^{-\left(k+1\right)J_{w_{0,n}}}\tilde{\Pi}_{k,w,n}\right\Vert _{L^{2}}\leq C\nu^{\left(d+\frac{3}{2}\right)}e^{n\left(\mathrm{Pr}\left(V-\left(d+2\right)J\right)+R\left(n\right)\right)}.\label{eq:expansion_of_Qn}
\end{equation}
For any $n\geq1$, we have from (\ref{eq:cor_F^n})
\begin{eqnarray}
\hat{Q}^{n} & = & \hat{A}_{m}\hat{\chi}^{-1}\mathcal{L}^{n}\hat{\chi}\hat{A}_{m}^{-1}\nonumber \\
 & = & \sum_{w_{0,n}}\hat{Q}_{w_{0,n}}\label{eq:sum_Q}
\end{eqnarray}
with individual terms
\[
\hat{Q}_{w_{0,n}}=\hat{A}_{m}\hat{\chi}^{-1}\mathcal{L}_{w_{0,n}}^{n}\hat{\chi}\hat{A}_{m}^{-1}.
\]
Using (\ref{eq:conjug_2}) we get
\[
\hat{Q}_{w_{0,n}}=e^{i\nu\tau_{w_{0,n}}+V_{w_{0,n}}}\hat{A}_{m}\hat{\chi}^{-1}\widehat{\mathcal{T}}_{L^{n}\left(w\right)}\hat{D}_{J_{w_{0,n}}}\widehat{\mathcal{T}}_{w}^{-1}\hat{\chi}\hat{A}_{m}^{-1}.
\]
In order to use the expansion (\ref{eq:expansion_D}) for $\hat{D}_{J_{w_{0,n}}}\hat{\chi}_{0}$,
let us choose $\chi_{0}\in C_{0}^{\infty}\left(\mathbb{R}\right)$
as in Theorem \ref{thm:Dilation_expansion} such that $\chi_{0}\left(y\right)=1$
for every $w\in\mathcal{W}$ and $y\in\mathrm{supp}\left(\widehat{\mathcal{T}}_{w}^{-1}\hat{\chi}\right)$
. This choice of $\chi_{0}$ is possible uniformly with respect to
$w$. Then $\widehat{\mathcal{T}}_{w}^{-1}\hat{\chi}=\hat{\chi}_{0}\widehat{\mathcal{T}}_{w}^{-1}\hat{\chi}$.
So (\ref{eq:expansion_D}) gives that for a given $d\geq1$, and using
notation (\ref{eq:def_Pi_w}), 
\begin{eqnarray}
\hat{Q}_{w_{0,n}} & = & e^{i\nu\tau_{w_{0,n}}+V_{w_{0,n}}}\left(\sum_{k=0}^{d}e^{-\left(k+1\right)J_{w_{0,n}}}\hat{A}_{m}\hat{\chi}^{-1}\widehat{\mathcal{T}}_{L^{n}\left(w\right)}\Pi_{k}\widehat{\mathcal{T}}_{w}^{-1}\hat{\chi}\hat{A}_{m}^{-1}\right)+R_{w_{0,n}}\label{eq:expansion_Qn}\\
 & = & e^{i\nu\tau_{w_{0,n}}+V_{w_{0,n}}}\left(\sum_{k=0}^{d}e^{-\left(k+1\right)J_{w_{0,n}}}\tilde{\Pi}_{k,w,n}\right)+R_{w_{0,n}},\nonumber 
\end{eqnarray}
with a remainder $R_{w_{0,n}}$ given by
\[
R_{w_{0,n}}=e^{i\nu\tau_{w_{0,n}}+V_{w_{0,n}}}\hat{A}_{m}\hat{\chi}^{-1}\widehat{\mathcal{T}}_{L^{n}\left(w\right)}\left(\hat{D}_{J_{w_{0,n}}}\hat{\chi}_{0}-\sum_{k=0}^{d}e^{-\left(k+1\right)J_{w_{0,n}}}\Pi_{k}\right)\widehat{\mathcal{T}}_{w}^{-1}\hat{\chi}\hat{A}_{m}^{-1}.
\]
From (\ref{eq:expansion_D}) and (\ref{eq:T_bounded}) its norm is
bounded by:

\begin{eqnarray}
\left\Vert R_{w_{0,n}}\right\Vert _{L^{2}\left(\mathbb{R}\right)} & \leq & C\left|e^{i\nu\tau_{w_{0,n}}+V_{w_{0,n}}}\right|e^{-\frac{1}{2}J_{w_{0,n}}}\left(\nu e^{-J_{w_{0,n}}}\right)^{d+3/2}\nonumber \\
 & \leq & C\nu^{\left(d+\frac{3}{2}\right)}e^{\left(V-\left(d+2\right)J\right)_{w_{0,n}}}\label{eq:remain_R}
\end{eqnarray}
with some constant $C>0$ independent of $\nu$, $n,$ and $w$. Using
(\ref{eq:sum_with_pressure-1}) from the Appendix, the sum of these
remainders is bounded by
\begin{eqnarray*}
\sum_{w_{0,n}}\left\Vert R_{w_{0,n}}\right\Vert _{L^{2}\left(\mathbb{R}\right)} & \leq & C\nu^{\left(d+\frac{3}{2}\right)}\sum_{w_{0,n}}e^{\left(V-\left(d+2\right)J\right)_{w_{0,n}}}\\
 & \leq & C\nu^{\left(d+\frac{3}{2}\right)}e^{n\left(\mathrm{Pr}\left(V-\left(d+2\right)J\right)+R\left(n\right)\right)},\quad\mbox{with }R\left(n\right)\underset{n\rightarrow\infty}{\longrightarrow0}
\end{eqnarray*}
From (\ref{eq:sum_Q}) and (\ref{eq:expansion_Qn}) we deduce (\ref{eq:expansion_of_Fn}).
\end{proof}

\section{\label{sec:Diagonal-approximation}Diagonal approximation}

We have defined the bounded operator $\hat{Q}:L^{2}\left(\mathbb{R}\right)\rightarrow L^{2}\left(\mathbb{R}\right)$
in (\ref{eq:def_Q}). For $n\geq1$, let
\begin{equation}
P_{n}:=\left(\hat{Q}^{n}\right)^{*}\hat{Q}^{n}\label{eq:def_Pn}
\end{equation}
which is a positive bounded self-adjoint operator on $L^{2}\left(\mathbb{R}\right)$.
In the following Theorem we bound the norm of $P_{n}$ and this will
be used in Section \ref{subsec:Proof-of-main_Theorem} to deduce a
bound on the spectral radius of the transfer operator (the main result
of this paper). $\left[x\right]$ will denote the approximation of
$x\in\mathbb{R}$ by the closest smaller integer.

\noindent\fcolorbox{blue}{white}{\begin{minipage}[t]{1\columnwidth - 2\fboxsep - 2\fboxrule}%
\begin{prop}
\label{thm:7.1}We make the assumption \ref{hyp:minimal_capt} of
minimal captivity. Let $\epsilon>0$, $0\leq J_{c}<2J_{min}-\epsilon$
and $0<\beta<1$. There exists $C>0$, such that for any $\nu>0$
and $n$ given by 
\begin{equation}
n:=\left[\frac{2}{J_{c}+\epsilon}\log\nu\right]>\frac{1}{J_{min}}\log\nu.\label{eq:n_erhenfest}
\end{equation}
We have 
\begin{equation}
\left\Vert P_{n}\right\Vert _{L^{2}}\leq C\left(\nu\sum_{w_{0,n}}e^{2\left(V-J\right)_{w_{0,N^{*}}}}e^{\left(V-J\right)_{w_{N^{*},n}}}\sum_{w'_{N^{*},n}}e^{\left(V-J\right)_{w'_{N^{*},n}}}\right)+C\beta^{n}\label{eq:expansion_P_n}
\end{equation}
with 
\begin{equation}
N^{*}:=N^{*}\left(w,J_{c}\right):=\mathrm{max}\left\{ k\leq n,\quad\mbox{s.t. }J_{w_{0,k}}<nJ_{c}\right\} .\label{eq:def_N*}
\end{equation}
\end{prop}

\end{minipage}}
\begin{rem}
Note that in the inequality (\ref{eq:expansion_P_n}), the Birkhoff
sums of $(V-J)$ can be calculated with any possible extension of
the word fragments. 
\end{rem}

~
\begin{rem}
In semiclassical analysis, the time $\frac{1}{J_{min}}\log\nu$ in
(\ref{eq:n_erhenfest}) is called the maximal local Ehrenfest time.
The definition of $N^{*}$ in (\ref{eq:def_N*}) can be written as
\[
e^{J_{w_{0,N^{*}}}}\simeq\nu^{2}
\]
and $N^{*}$ can be called \textbf{``twice the standard local Ehrenfest
time''}. It is known in specific situations that this particular
time is important; see discussions and results in \cite{fred-steph-02,fred-trace-06}.
For example there are curious phenomena of ``quantum revival'' or
``quantum period'' at that time for the quantum cat map as explained
in \cite{debievre-00,fred-steph-02}. 
\end{rem}

The rest of Section \ref{sec:Diagonal-approximation} is devoted to
the proof of Proposition \ref{thm:7.1}. At some point of the proof,
i.e. Lemma \ref{Prop:Separation-of-orbits}, we will use the hypothesis
of minimal captivity and obtain that the orbits of length $n$ are
``well distributed and separated'' on phase space (inside the trapped
set $\mathcal{K}$) so that they do not ``interfere'' with each
over, provided the time $n$ is not too long. Using this, the double
sum over the orbits that appears in Lemma \ref{lem:7.3}, can be reduced
to a much smaller sum, which is basically a sum over the diagonal.
This step can thus be considered as some kind of ``diagonal approximation''.

In a first step we use the asymptotic expansion for the transfer operator,
Theorem \ref{thm:6.6}, in order to write $\left\Vert P_{n}\right\Vert _{L^{2}}$
as a double sum over orbits. The next Lemma shows that this is possible
provided we consider time $n$ long enough w.r.t. $\nu$.

\begin{center}{\color{blue}\fbox{\color{black}\parbox{16cm}{
\begin{lem}
\label{lem:7.3}Let $0<\alpha<J_{min}$ and $0<\beta<1$. There exist
$d\in\mathbb{N}$ and $C>0$ such that for any $\nu>0$ and
\begin{equation}
n:=\left[\frac{1}{\alpha}\log\nu\right]\label{eq:choice_n_hbar}
\end{equation}
we have 
\begin{equation}
\left\Vert P_{n}\right\Vert _{L^{2}}\leq\mathbb{S}+C\beta^{n}.\label{eq:first_bound_Pn}
\end{equation}
with
\begin{equation}
\mathbb{S}:=\sum_{w'_{0,n},w_{0,n}}e^{V_{w'_{0,n}}+V_{w_{0,n}}}\sum_{k',k=0}^{d}e^{-\left(k'+1\right)J_{w'_{0,n}}-\left(k+1\right)J_{w_{0,n}}}\left|\mathrm{Tr}\left(\tilde{\Pi}_{k',w',n}^{*}\tilde{\Pi}_{k,w,n}\right)\right|\label{eq:def_double_sum}
\end{equation}
\end{lem}

}}}\end{center}
\begin{rem}
Later we will provide an upper bound for $\mathbb{S}$ keeping only
the terms $k=k'=0$.
\end{rem}

\begin{proof}
We use Theorem \ref{thm:6.6} with its formulation (\ref{eq:expansion_of_Qn})
and write
\[
\hat{Q}^{n}=S_{n}+R_{n}
\]
with
\begin{equation}
S_{n}:=\sum_{w_{0,n}}e^{i\nu\tau_{w_{0,n}}+V_{w_{0,n}}}\sum_{k=0}^{d}e^{-\left(k+1\right)J_{w_{0,n}}}\tilde{\Pi}_{k,w,n}\label{eq:Sn}
\end{equation}
which is a finite rank operator and
\begin{equation}
\left\Vert R_{n}\right\Vert _{L^{2}}\underset{(\ref{eq:expansion_of_Fn})}{\leq}C\nu^{\left(d+\frac{3}{2}\right)}e^{n\left(\mathrm{Pr}\left(V-\left(d+2\right)J\right)+R\left(n\right)\right)}.\label{eq:bound_Rn}
\end{equation}
Then
\begin{equation}
P_{n}\underset{(\ref{eq:def_Pn})}{=}\left(\hat{Q}^{n}\right)^{*}\hat{Q}^{n}=S_{n}^{*}S_{n}+\underbrace{\left(R_{n}^{*}S_{n}+S_{n}^{*}R_{n}+R_{n}^{*}R_{n}\right)}_{R'_{n}},\label{eq:expr_Pn}
\end{equation}
gives
\begin{equation}
\left\Vert P_{n}\right\Vert _{L^{2}}\leq\left\Vert S_{n}^{*}S_{n}\right\Vert _{L^{2}}+\left\Vert R'_{n}\right\Vert _{L^{2}}.\label{eq:expre_norm_Pn}
\end{equation}
In order to bound the remainder $\left\Vert R'_{n}\right\Vert _{L^{2}}$
we have to bound $\left\Vert R_{n}\right\Vert _{L^{2}}$ and $\left\Vert S_{n}\right\Vert _{L^{2}}$.
Notice that
\begin{equation}
\nu^{\left(k+\frac{1}{2}\right)}\underset{(\ref{eq:choice_n_hbar})}{\leq}Ce^{n\left(k+\frac{1}{2}\right)\alpha}=Ce^{n\left(k+\frac{1}{2}\right)\left(\alpha+\epsilon\right)}e^{-\epsilon n\left(k+\frac{1}{2}\right)}\label{eq:bound_hbar_}
\end{equation}
Let
\begin{equation}
\beta_{k}:=e^{(k+\frac{1}{2})\left(\alpha+\epsilon\right)+\mathrm{Pr}(V-(k+1)J)},\label{eq:def_beta_k}
\end{equation}
with $\epsilon>0$ chosen later. So

\[
\left\Vert R_{n}\right\Vert _{L^{2}}\underset{(\ref{eq:bound_Rn})}{\leq}C\nu^{\left(d+\frac{3}{2}\right)}e^{n\left(\mathrm{Pr}\left(V-\left(d+2\right)J\right)+R\left(n\right)\right)}\underset{(\ref{eq:def_beta_k}),(\ref{eq:bound_hbar_})}{\leq}C'\beta_{d+1}^{n}e^{n\left(R\left(n\right)-\epsilon\left(d+\frac{3}{2}\right)\right)}\leq C\beta_{d+1}^{n},
\]
and 
\begin{eqnarray}
\left\Vert S_{n}\right\Vert _{L^{2}} & \underset{(\ref{eq:Sn}),(\ref{eq:bound_Pi_tilde_k})}{\leq} & C\sum_{k=0}^{d}\nu^{\left(k+\frac{1}{2}\right)}\sum_{w_{0,n}}e^{\left(V-\left(k+1\right)J\right)_{w_{0,n}}}\label{eq:norme_Sn}\\
 & \underset{(\ref{eq:sum_with_pressure-1})}{\leq} & C\sum_{k=0}^{d}\nu^{\left(k+\frac{1}{2}\right)}e^{n\left(\mathrm{Pr}\left(V-\left(k+1\right)J\right)+R\left(n\right)\right)},\quad\mbox{with }R\left(n\right)\underset{n\rightarrow\infty}{\longrightarrow0}\nonumber \\
 & \underset{(\ref{eq:def_beta_k}),(\ref{eq:bound_hbar_})}{\leq} & C\sum_{k=0}^{d}\beta_{k}^{n}\nonumber 
\end{eqnarray}
Notice that
\[
\frac{\beta{}_{k+1}}{\beta_{k}}\underset{(\ref{eq:def_beta_k})}{=}e^{\alpha+\epsilon-\delta_{k}}
\]
with
\[
\delta_{k}:=\mathrm{Pr}\left(V-\left(k+1\right)J\right)-\mathrm{Pr}\left(V-\left(k+2\right)J\right).
\]

From Proposition \ref{prop:pressure_derivatives-1} we have $\forall r,\left(\frac{\partial}{\partial r}\mathrm{Pr}\left(V-rJ\right)\right)\left(r\right)\leq-J_{min}$.
Consequently for $\alpha<J_{min}$ we have $k\alpha+\mathrm{Pr}\left(V-kJ\right)\underset{k\rightarrow\infty}{\rightarrow}-\infty$.
Hence, if $\epsilon>0$ is such that $\alpha+\epsilon<J_{min}$ then
\begin{equation}
\beta_{k}\underset{k\rightarrow\infty}{\rightarrow}0.\label{eq:lim_beta_k}
\end{equation}
Also $\delta_{k}\geq\delta_{k+1}\geq J_{min}>\alpha+\epsilon$ for
any $k$ hence 
\begin{equation}
\frac{\beta{}_{k+1}}{\beta_{k}}<1.\label{eq:beta_k_decreases}
\end{equation}
In particular $\left\Vert S_{n}\right\Vert _{L^{2}}\underset{(\ref{eq:norme_Sn})}{\leq}C\beta_{0}^{n}$
and we record for later use that
\begin{equation}
\sum_{k=0}^{d}\nu^{\left(k+\frac{1}{2}\right)}\sum_{w_{0,n}}e^{\left(V-\left(k+1\right)J\right)_{w_{0,n}}}\leq C\beta_{0}^{n}.\label{eq:for_later}
\end{equation}
We conclude that
\begin{equation}
\left\Vert R'_{n}\right\Vert _{L^{2}}\leq2\left\Vert R_{n}\right\Vert \left\Vert S_{n}\right\Vert +\left\Vert R_{n}\right\Vert ^{2}=\left\Vert R_{n}\right\Vert \left(2\left\Vert S_{n}\right\Vert +\left\Vert R_{n}\right\Vert \right)\leq C\beta_{d+1}^{n}\beta_{0}^{n}.\label{eq:bound_Rp_n}
\end{equation}

We now consider the term $S_{n}^{*}S_{n}$ in (\ref{eq:expr_Pn})
given by

\begin{equation}
S_{n}^{*}S_{n}\underset{(\ref{eq:Sn})}{=}\sum_{w'_{0,n},w_{0,n}}e^{\overline{i\nu\tau_{w'_{0,n}}+V_{w'_{0,n}}}}e^{i\nu\tau_{w_{0,n}}+V_{w_{0,n}}}\sum_{k',k=0}^{d}e^{-\left(k'+1\right)J_{w'_{0,n}}-\left(k+1\right)J_{w_{0,n}}}\tilde{\Pi}_{k',w',n}^{*}\tilde{\Pi}_{k,w,n}.\label{eq:SnSn_nondiagonal_expansion}
\end{equation}
$S_{n}^{*}S_{n}$ is a finite rank positive self-adjoint operator,
hence we have the bound
\begin{eqnarray}
\left\Vert S_{n}^{*}S_{n}\right\Vert _{L^{2}} & \leq & \left\Vert S_{n}^{*}S_{n}\right\Vert _{\mathrm{Tr}}=\mathrm{Tr}\left(S_{n}^{*}S_{n}\right)\label{eq:loss_inequality}\\
 & \leq & \sum_{w'_{0,n},w_{0,n}}e^{V_{w'_{0,n}}+V_{w_{0,n}}}\label{eq:normSS}\\
 &  & \qquad\sum_{k',k=0}^{d}e^{-\left(k'+1\right)J_{w'_{0,n}}-\left(k+1\right)J_{w_{0,n}}}\left|\mathrm{Tr}\left(\tilde{\Pi}_{k',w',n}^{*}\tilde{\Pi}_{k,w,n}\right)\right|\\
 & \underset{(\ref{eq:def_double_sum})}{=} & \mathbb{S}\nonumber 
\end{eqnarray}
and

\[
\left\Vert P_{n}\right\Vert _{L^{2}}\underset{(\ref{eq:expre_norm_Pn}),(\ref{eq:bound_Rp_n}),(\ref{eq:normSS})}{\leq}\mathbb{S}+C\left(\beta_{d+1}\beta_{0}\right)^{n}.
\]
Let $0<\beta<1$. Using (\ref{eq:lim_beta_k}), we can choose $d$
large enough so that $\beta_{d+1}\beta_{0}\leq\beta$. We have obtained
(\ref{eq:first_bound_Pn}).
\end{proof}
\begin{rem}
In the inequality of (\ref{eq:loss_inequality}) we have bounded the
$L^{2}$ norm by a trace norm. This is a crucial step in the paper.
This is obviously not an optimal bound. However it makes appear the
terms $\left|\mathrm{Tr}\left(\tilde{\Pi}_{k',w',n}^{*}\tilde{\Pi}_{k,w,n}\right)\right|$
and in the next Proposition we will see that these terms can be neglected
for many pairs of trajectories $w_{0,n},w'_{0,n}$.
\end{rem}

We first introduce the following notations. For $w,w'\in\mathcal{W}$
and $n\geq1$, suppose that $w_{0,n}\neq w_{0,n}'$ and let

\begin{equation}
n_{1}\left(w_{0,n},w'_{0,n}\right):=\min\left\{ 0\leq k\leq n,\quad w_{k}\neq w'_{k}\right\} \label{eq:def_n1}
\end{equation}
\[
n_{2}\left(w_{0,n},w'_{0,n}\right):=\min\left\{ 0\leq k\leq n,\quad w_{n-k}\neq w'_{n-k}\right\} .
\]
In other words this means that the words $w_{0,n}$ and $w'_{0,n}$
have equal letters at extremities $w_{i}=w_{i}'$ for $i\in[0,n_{1}[\cup]n-n_{2},n]$
and differ for letters: $w_{n_{1}}\neq w'_{n_{1}}$, $w_{n-n_{2}}\neq w'_{n-n_{2}}$. 

Notice that $J_{i,j}\left(x\right)>0$ hence the Birkhoff sum $J_{w_{0,k}}$
(defined in (\ref{eq:def_Birkhoff_sum})) is an increasing function
of $k$. For some given $w\in\mathcal{W}$, $n\geq1$, $J_{c}>0$,
let

\begin{eqnarray*}
N_{1}\left(w,n,J_{c}\right): & = & \mathrm{max}\left\{ 0\leq k\leq n,\quad\mbox{s.t. }J_{w_{0,k-1}}<\frac{nJ_{c}}{2}\right\} 
\end{eqnarray*}
\begin{eqnarray*}
N_{2}\left(w,n,J_{c}\right): & = & \mathrm{max}\left\{ 0\leq k\leq n,\quad\mbox{s.t. }J_{w_{n-k+1,n}}<\frac{nJ_{c}}{2}\right\} .
\end{eqnarray*}
Let us introduce the following Definition.

\begin{center}{\color{red}\fbox{\color{black}\parbox{16cm}{
\begin{defn}
For a given $J_{c}>0$, we call a pair of orbits $\left(w'_{0,n},w_{0,n}\right)$
\textbf{separable }if $w'_{0,n}\neq w_{0,n}$ and $n_{1}\left(w_{0,n},w'_{0,n}\right)\leq N_{1}\left(w,n,J_{c}\right)$
or $n_{2}\left(w_{0,n},w'_{0,n}\right)\leq N_{2}\left(w,n,J_{c}\right)$.
Otherwise we call the pair $\left(w'_{0,n},w_{0,n}\right)$ \textbf{non-separable}.
\end{defn}

}}}\end{center}

\begin{center}{\color{blue}\fbox{\color{black}\parbox{16cm}{
\begin{prop}
\label{Prop:Separation-of-orbits}\textbf{``Orbit separation''.}
We make the assumption \ref{hyp:minimal_capt} of minimal captivity.
Let $\varepsilon>0$ and $J_{c}>0$. Then for any $M\geq0$ there
exists $C_{M}>0$, such that for any $\nu>0$, $n:=\left[\frac{2}{J_{c}+\varepsilon}\log\nu\right]$,
any $w,w'\in\mathcal{W}$, if the pair of orbits $\left(w'_{0,n},w_{0,n}\right)$
is  separable\textbf{ }then
\begin{equation}
\nu^{-(k+k'+1)}\left|\mathrm{Tr}\left(\tilde{\Pi}_{k',w',n}^{*}\tilde{\Pi}_{k,w,n}\right)\right|\leq C_{M}e^{-\frac{\varepsilon}{2}nM}\label{eq:CMM}
\end{equation}
\end{prop}

}}}\end{center}
\begin{rem}
In other words, Proposition \ref{Prop:Separation-of-orbits} says
that for a separable pair of orbits $\left(w'_{0,n},w_{0,n}\right)$,
the term $\left|\mathrm{Tr}\left(\tilde{\Pi}_{k',w',n}^{*}\tilde{\Pi}_{k,w,n}\right)\right|$
will be ``negligible''. We will see later in Lemma \ref{lem:Separation-of-xw_zetaw}
that this is because a ``separable pair of orbits'' is indeed ``separated''
in phase space so that their Lagrangian states do not overlap.
\end{rem}

\begin{proof}[Proof of Proposition \ref{thm:7.1}.]
The proof of Proposition \ref{Prop:Separation-of-orbits} will be
given in Section \ref{subsec:Proof-of-Proposition}. For now, we continue
the proof of Proposition \ref{thm:7.1}. Let $J_{c}<2J_{min}$. Let
$\epsilon>0$, $\nu>0$ and $n:=\left[\frac{2}{J_{c}+\varepsilon}\log\nu\right]=\left[\frac{1}{\alpha}\log\nu\right]$
with $\alpha=\frac{1}{2}J_{c}+\frac{1}{2}\epsilon$. We suppose that
$\epsilon>0$ is small enough so that $J_{c}+\epsilon<2J_{min}$.
Hence $\alpha<J_{min}$ and we can apply Lemma \ref{lem:7.3}. We
decompose the double sum (\ref{eq:normSS}) over $w'_{0,n},w_{0,n}$
into separable and non-separable pairs: 
\[
\mathbb{S}=\mathbb{S}_{\textup{separable}}+\mathbb{S}_{\textup{non-separable}}.
\]
We first show that $\mathbb{S}_{\textup{separable}}$ is ``negligible''.
We have
\begin{eqnarray*}
\mathbb{S}_{\textup{separable}} & = & \sum_{w_{0,n},w'_{0,n}\textup{ sep.}}e^{V_{w'_{0,n}}+V_{w_{0,n}}}\sum_{k',k=0}^{d}e^{-\left(k'+1\right)J_{w'_{0,n}}-\left(k+1\right)J_{w_{0,n}}}\left|\mathrm{Tr}\left(\tilde{\Pi}_{k',w',n}^{*}\tilde{\Pi}_{k,w,n}\right)\right|\\
 & \underset{(\ref{eq:CMM})}{\leq} & C_{M}e^{-\frac{\varepsilon}{2}nM}\left(\sum_{k=0}^{d}\nu^{\left(k+\frac{1}{2}\right)}\sum_{w_{0,n}}e^{V_{w_{0,n}}-\left(k+1\right)J_{w_{0,n}}}\right)^{2}\\
 & \underset{(\ref{eq:for_later})}{\leq} & Ce^{-\frac{\varepsilon}{2}nM}\beta_{0}^{2n}=C\left(e^{-\frac{\varepsilon}{2}M}\beta_{0}^{2}\right)^{n}
\end{eqnarray*}
We deduce that for any given $0<\beta<1$ we can choose $M$ large
enough so that $e^{-\frac{\varepsilon}{2}M}\beta_{0}^{2}\leq\beta$
hence 
\[
\mathbb{S}_{\textup{separable}}\leq C_{M}\beta^{n},
\]
which means that $\mathbb{S}_{\textup{separable}}$ is ``negligible''.
We have now to bound from above the ``non separable trajectories''
for which $n_{1}>N_{1}$ and $n_{2}>N_{2}$. By the definition of
$J_{c}$ and $n$ there exists some $\tilde{\epsilon}>0$ such $J_{c}+\epsilon+\tilde{\epsilon}<2J_{min}$
which implies 
\[
n\geq\frac{1}{J_{min}-\tilde{\epsilon}}\log\left(\nu\right).
\]
For every word $w_{0,n}$, we have $J_{w_{0,n}}\geq nJ_{min}$ hence
\begin{equation}
\nu e^{-J_{w_{0,n}}}\leq\nu e^{-nJ_{min}}\leq\left(e^{-\tilde{\epsilon}}\right)^{n}.\label{eq:hyp_on_n_Jmin}
\end{equation}
We write (with a constant $C$ that is independent of $n$ but whose
actual value might change from line to line) 
\begin{eqnarray*}
\mathbb{S}_{\textup{non-separable}}: & = & \sum_{w_{0,n},w'_{0,n},\textup{non-sep}}e^{V_{w'_{0,n}}+V_{w_{0,n}}}\sum_{k',k=0}^{d}e^{-\left(k'+1\right)J_{w'_{0,n}}-\left(k+1\right)J_{w_{0,n}}}\left|\mathrm{Tr}\left(\tilde{\Pi}_{k',w',n}^{*}\tilde{\Pi}_{k,w,n}\right)\right|\\
 & \underset{(\ref{eq:general_bound})}{\leq} & C\nu\sum_{w_{0,n},w'_{0,n},\textup{non-sep}}e^{\left(V-J\right)_{w'_{0,n}}+\left(V-J\right)_{w_{0,n}}}\sum_{k',k=0}^{d}\left(\nu e^{-J_{w'_{0,n}}}\right)^{k'}\left(\nu e^{-J_{w_{0,n}}}\right)^{k}\\
 & \underset{(\ref{eq:hyp_on_n_Jmin})}{\leq} & C\nu\sum_{w_{0,n,}w'_{0,n},\textup{non-sep}}e^{\left(V-J\right)_{w'_{0,n}}+\left(V-J\right)_{w_{0,n}}}\sum_{k',k=0}^{d}e^{-n\tilde{\epsilon}\left(k+k'\right)}\\
 & \leq & C\nu\sum_{w_{0,n,}w'_{0,n},\textup{non-sep}}e^{\left(V-J\right)_{w'_{0,n}}+\left(V-J\right)_{w_{0,n}}}
\end{eqnarray*}
We will now use the fact that we only sum over non-separable pairs
of words. Recall that this requires, that the word $w'_{0,n}$ is
equal to $w_{0,n}$ for their first $N_{1}\left(w,n,J_{c}\right)$
and their last $N_{2}\left(w,n,J_{c}\right)$ symbols. Accordingly
we can write the last expression as
\begin{align*}
\mathbb{S}_{\textup{non-separable}} & \leq C\nu\sum_{w_{0,n}}\left(e^{\left(V-J\right)_{w_{0,n}}}e^{\left(V-J\right)_{w_{0,N_{1}}}+\left(V-J\right)_{w_{n-N_{2},n}}}\right.\\
 & \qquad\qquad\left.\sum_{w'_{N_{1},n-N_{2},}\textup{s.t. }w'_{N_{1}}=w_{N_{1}},w'_{n-N_{2}}=w_{n-N_{2}}}e^{\left(V-J\right)_{w'_{N_{1},n-N_{2}}}}\right).
\end{align*}
Note that the last transformation can be done in an exact way (with
the same constant $C$): One can choose appropriate extensions of
the words $w'_{0,n}$, $w'_{0,N_{1}}=w_{0,N_{1}},w'_{N_{1},n-N_{2}}$
and $w'_{n-N_{2},n}=w{}_{n-N_{2},n}$ such that one has $(V-J)_{w'_{0,n}}=(V-J)_{w_{0,N_{1}}}+(V-J)_{w'_{N_{1},n-N_{2}}}+(V-J)_{w_{n-N_{2},n}}$
. Note furthermore, that, since we are only interested in an upper
bound, we can remove the restrictions on the initial and last symbol
in the second sum and obtain
\[
\mathbb{S}_{\textup{non-separable}}\leq C\nu\sum_{w_{0,n}}e^{\left(V-J\right)_{w_{0,n}}}e^{\left(V-J\right)_{w_{0,N_{1}}}+\left(V-J\right)_{w_{n-N_{2},n}}}\sum_{w'_{N_{1},n-N_{2}}}e^{\left(V-J\right)_{w'_{N_{1},n-N_{2}}}}.
\]
Let us finally explain, how to pass from this expression to (\ref{eq:expansion_P_n})
which involves $N^{*}$ instead of $N_{1}$ and $N_{2}$: Let us first
hypothetically assume that the symbolic dynamic is complete and that
$(V-J)_{w_{0,n}}$ would only depend on the fragment $w_{1,n}$ and
not on the extension and the we have $(V-J)_{w_{0,n}}=(V-J)_{w_{0,a}}+(V-J)_{w_{a,n}}$
for any $0<a<n$. Then we could can rearrange each word $w_{0,n}$
by putting the fragments $w_{0,N_{1}}$ and $w_{n-N_{2},n}$ at the
beginning of the new word $\tilde{w}_{0,n}$ such that $N^{*}(\tilde{w}_{0,n},J_{c})=N_{1}\left(w,n,J_{c}\right)+N_{2}\left(w,n,J_{c}\right)$
and rewrite the last expression by rearranging the combinatorial sum
over the words, as
\[
\mathbb{S}_{\textup{non-separable}}\leq C\nu\sum_{w_{0,n}}e^{\left(V-J\right)_{w_{0,n}}}e^{\left(V-J\right)_{w_{0,N^{*}}}}\sum_{w'_{N^{*},n}}e^{\left(V-J\right)_{w'_{N^{*},n}}},
\]
without having to modify the constant $C.$ Now both above assumptions
are in general not true in the framework in which we are working.
Nevertheless we can obtain the above bound by modifying the constant
$C$. This is justified for the following reasons: Firstly the expressions
$(V-J)_{w_{0,n}}$ depend on the extension of the word $w_{0,n}$
only in a controlled way (see Lemma \ref{lem:bounded_variation})
thus we always have $(V-J)_{w_{0,n}}\leq(V-J)_{w_{0,a}}+(V-J)_{w_{a,n}}+c_{0}$
and the $n$ independent constant $c_{0}$ can always be absorbed
in multiplicative constant $C$. The second problem concerns non complete
symbolic dynamic: In the sum over $w_{0,n}$ there might occur word
fragments $w_{0,N_{1}}$ and $w_{n-N_{2},n}$ that do not occur as
the leading and the last letters in some $w_{0,N^{*}}$. However,
as we demanded a transitive symbolic dynamic we can assure, that the
word fragments $w_{0,N_{1}}$ and $w_{n-N_{2},n}$ appear as disjoint
fragments of some $w_{0,N^{*}+T}$ where $T$ is the maximal transition
time between two letters. Thus we can bound 
\[
\mathbb{S}_{\textup{non-separable}}\leq C\nu\sum_{w_{0,n}}e^{\left(V-J\right)_{w_{0,n}}}e^{\left(V-J\right)_{w_{0,N^{*}+T}}}\sum_{w'_{N^{*},n}}e^{\left(V-J\right)_{w'_{N^{*},n}}},
\]
where we absorb the impact of the additional letters that are needed
to concatenate $w_{0,N_{1}}$ and $w_{n-N_{2},n}$ in a modified constant
$C$. Finally we can also absorb the last $T$ terms in the Birkhoff
sum $\left(V-J\right)_{w_{0,N^{*}+T}}$ in the constant $C$ and obtain

\[
\mathbb{S}_{\textup{non-separable}}\leq C\nu\sum_{w_{0,n}}e^{\left(V-J\right)_{w_{0,n}}}e^{\left(V-J\right)_{w_{0,N^{*}}}}\sum_{w'_{N^{*},n}}e^{\left(V-J\right)_{w'_{N^{*},n}}}.
\]
 Finally we get
\begin{eqnarray*}
\mathbb{S} & = & \mathbb{S}_{\textup{separable}}+\mathbb{S}_{\textup{non-separable}}\\
 & \leq & C_{M}\beta^{n}+C\nu\sum_{w_{0,n}}e^{2\left(V-J\right)_{w_{0,N^{*}}}}e^{\left(V-J\right)_{w_{N^{*},n}}}\sum_{w'_{N^{*},n}}e^{\left(V-J\right)_{w'_{N^{*},n}}}.
\end{eqnarray*}
Together with (\ref{eq:first_bound_Pn}) we have finished the proof
of Proposition \ref{thm:7.1}.
\end{proof}

\subsection{\label{subsec:Proof-of-Proposition}Proof of Proposition \ref{Prop:Separation-of-orbits}
about separation of orbits}

The following Lemma gives bounds for the quantities $\left|x_{L^{n}(w)}-x_{L^{n}(w')}\right|$
and $\min_{x\in K_{a}}\left(\left|\zeta_{w}\left(x\right)-\zeta_{w'}\left(x\right)\right|\right)$
that will appear later in Lemma \ref{lem:9}, Eq.(\ref{eq:Tr_proj_stat_phase_bound}).

\begin{center}{\color{blue}\fbox{\color{black}\parbox{16cm}{
\begin{lem}
\label{lem:Separation-of-xw_zetaw}We make the assumption \ref{hyp:minimal_capt}
of minimal captivity. Let $w,w'\in\mathcal{W}$, $n\geq1$ and suppose
that $w_{0,n}\neq w_{0,n}'$. Furthermore as in (\ref{eq:def_n1}),
let $n_{1},n_{2}\in\mathbb{N}$ be such that $w_{i}=w_{i}'$ if $i<n_{1}$
or $i>n-n_{2}$ and $w_{n_{1}}\neq w'_{n_{1}}$, $w_{n-n_{2}}\neq w'_{n-n_{2}}$.
Then we have

\begin{equation}
Ce^{-J_{w_{n-n_{2}+1,n}}}\leq\left|x_{L^{n}(w)}-x_{L^{n}(w')}\right|,\label{eq:separation_in_x}
\end{equation}
and for any $x\in K_{a}\cap I_{w_{0}}$,
\begin{equation}
Ce^{-J_{w_{0,n_{1}-1}}}\leq\left|\zeta_{w}(x)-\zeta_{w'}(x)\right|\leq C'e^{-J_{w_{0,n_{1}-1}}},\label{eq:estimate-1}
\end{equation}
with $C,C'>0$ independent of $w,w',n,x$.
\end{lem}

}}}\end{center}
\begin{proof}
We have
\[
x_{L^{n-\left(n_{2}-1\right)}(w)}\in\phi_{w_{n-n_{2}},w_{n-\left(n_{2}-1\right)}}\left(I_{w_{n-n_{2}}}\right)\textup{ and }x_{L^{n-\left(n_{2}-1\right)}(w')}\in\phi_{w'_{n-n_{2}},w'_{n-\left(n_{2}-1\right)}}\left(I_{w'_{n-n_{2}}}\right).
\]
As we have $w_{n-n_{2}}\neq w'_{n-n_{2}}$ we conclude from the strong
separation condition (\ref{eq:hyp_non_intersect}) we have that 
\begin{equation}
|x_{L^{n-\left(n_{2}-1\right)}(w)}-x_{L^{n-\left(n_{2}-1\right)}(w')}|\geq C\label{eq:x_separation_n2-1}
\end{equation}
with $C>0$ which is the minimal distance between the intervals $\phi_{i,j}\left(I_{i}\right)$.
As $w'_{n-\left(n_{2}-1\right),n}=w{}_{n-\left(n_{2}-1\right),n}$
we obtain 
\[
x_{L^{n}(w)}=\phi_{w_{n-\left(n_{2}-1\right),n}}\left(x_{L^{n-\left(n_{2}-1\right)}(w)}\right)\textup{ and }x_{L^{n}(w')}=\phi_{w_{n-\left(n_{2}-1\right),n}}\left(x_{L^{n-\left(n_{2}-1\right)}(w')}\right).
\]
From (\ref{eq:x_separation_n2-1}) and the fact, that $\phi_{w_{n_{2}+1},n}$
is a diffeomorphism we get 
\[
|x_{L^{n}(w)}-x_{L^{n}(w')}|\geq|x_{L^{n-\left(n_{2}-1\right)}(w)}-x_{L^{n-\left(n_{2}-1\right)}(w')}|\cdot\min_{x\in I_{w_{n-\left(n_{2}-1\right)}}}|\phi'_{w_{n-\left(n_{2}-1\right),n}}(x)|\geq Ce^{-J_{w_{n-n_{2}+1,n}}}.
\]
We have obtained the first inequality in (\ref{eq:separation_in_x}).

Now we prove the second inequality of (\ref{eq:separation_in_x})
which uses Assumption \ref{hyp:minimal_capt} of minimal captivity.
The minimal captivity assumption is equivalent to the following property:
Let $\mathcal{K}_{\varepsilon}$ be a closed neighborhood of the trapped
set as in (\ref{eq:hyp_minimal_captivity_epsilon}). For any $i\rightsquigarrow j$
and $i\rightsquigarrow k$ with $j\neq k$ we have that
\[
\tilde{\phi}_{i,j}^{-1}\left(\mathcal{K}_{\varepsilon}\cap\pi^{-1}\left(I_{j}\right)\right)\bigcap\tilde{\phi}_{i,k}^{-1}\left(\mathcal{K}_{\varepsilon}\cap\pi^{-1}\left(I_{k}\right)\right)=\emptyset,
\]
because otherwise the dynamics of $\tilde{\phi}$ restricted to $\mathcal{K}_{\varepsilon}$
is not univalued. 

From this we deduce that there exists $C_{\mathrm{mini-capt}}>0$,
such that for any any $i\rightsquigarrow j$ and $i\rightsquigarrow k$
with $j\neq k$, if $\tilde{x}\in I_{i}$, $\left(\tilde{x},\xi\right),\left(\tilde{x},\xi'\right)\in\mathcal{K}_{\varepsilon}$,
$\tilde{\phi}_{i,j}\left(\tilde{x},\xi\right),\tilde{\phi}_{i,k}\left(\tilde{x},\xi'\right)\in\mathcal{K}_{\varepsilon}$
then $\left|\xi-\xi'\right|\geq C_{\mathrm{mini-capt}}$.

Let $x\in I{}_{w_{0}}\cap K_{a}$ and for $m\leq n$ define $x_{m}:=\phi_{w_{0,m}}(x)$,
$\xi_{m}:=\zeta_{L^{m}(w)}\left(x_{m}\right)$, $x'_{m}:=\phi_{w'_{0,m}}(x)$
and $\xi'_{m}:=\zeta_{L^{m}(w')}\left(x'_{m}\right)$. From Proposition
\ref{prop:S_tilde_bijective} one has $\tilde{\phi}_{w_{m,m+1}}(x_{m},\xi_{m})=(x_{m+1},\xi_{m+1})$.
As $x_{m}\in K_{a}$ with $a$ chosen large enough according to (\ref{eq:Ka_w}),
we have $\left(x_{m},\xi_{m}\right),\left(x'_{m},\xi'_{m}\right)\in\mathcal{K}_{\varepsilon}$.
We have $x_{n_{1}-1}=x'_{n_{1}-1}$ from definition of $n_{1}$ and
we have $\left(x_{n_{1}},\xi_{n_{1}}\right)=\tilde{\phi}_{w_{n_{1}-1,n_{1}}}\left(x_{n_{1}-1},\xi_{n_{1}-1}\right)$,
$\left(x'_{n_{1}},\xi'_{n_{1}}\right)=\tilde{\phi}_{w'_{n_{1}-1,n_{1}}}\left(x'_{n_{1}-1},\xi'{}_{n_{1}-1}\right)$
and $w_{n_{1}}\neq w'_{n_{1}}$ so from above we deduce that $\left|\xi_{n_{1}-1}-\xi'_{n_{1}-1}\right|\geq C_{\mathrm{mini-capt}}$.
Furthermore by Lemma \ref{lem:escape_F} we know that $\left|\xi_{n_{1}-1}-\xi'_{n_{1}-1}\right|<\tilde{C}$.
Using the definition of the canonical map $\tilde{\phi}$ in (\ref{eq:canonical_map_Fij})
we compute
\[
\left|\xi_{n_{1}-1}-\xi'_{n_{1}-1}\right|=\left|\left(\tilde{\phi}_{w_{0,n}}\left(x,\xi_{0}\right)\right)_{\xi}-\left(\tilde{\phi}_{w'_{0,n}}\left(x,\xi'_{0}\right)\right)_{\xi}\right|=e^{J_{w_{0,n}}\left(x\right)}\left|\xi_{0}-\xi'_{0}\right|.
\]
Now using the bounds for $\left|\xi_{n_{1}-1}-\xi'_{n_{1}-1}\right|$
from above as well as the bounded variation estimate from Lemma \ref{lem:bounded_variation}
we obtain (\ref{eq:estimate-1}).
\end{proof}
\begin{center}{\color{blue}\fbox{\color{black}\parbox{16cm}{
\begin{lem}
\label{lem:9}For any $m>0$, there exists $C>0$, such that for any
$0\leq k,k'<m-\frac{3}{2}$, $w,w'\in\mathcal{W}$, $n\geq1$, $\nu>0$,
we have
\begin{equation}
\nu^{-\left(k+k'+1\right)}\left|\mathrm{Tr}\left(\tilde{\Pi}_{k',w',n}^{*}\tilde{\Pi}_{k,w,n}\right)\right|\leq C.\label{eq:general_bound}
\end{equation}
Moreover for any $M_{1},M_{2}\geq0$, there exists $C_{M_{1},M_{2}}$,
such that for any $w,w'\in\mathcal{W}$ with $x_{L^{n}\left(w\right)}\neq x_{L^{n}\left(w'\right)}$
and $\min_{x\in K_{a}}\left(\left|\zeta_{w}\left(x\right)-\zeta_{w'}\left(x\right)\right|\right)\neq0$,
with $K_{a}$ given in (\ref{eq:Ka_w}), we have 
\begin{eqnarray}
\nu^{-\left(k+k'+1\right)}\left|\mathrm{Tr}\left(\tilde{\Pi}_{k',w',n}^{*}\tilde{\Pi}_{k,w,n}\right)\right| & \leq & C_{M_{1},M_{2}}\left(\frac{\nu^{-1}}{\left|x_{L^{n}(w)}-x_{L^{n}(w')}\right|}\right)^{M_{1}}\label{eq:Tr_proj_stat_phase_bound}\\
 &  & \qquad\left(\frac{\nu^{-1}}{\min_{x\in K_{a}}\left(\left|\zeta_{w}\left(x\right)-\zeta_{w'}\left(x\right)\right|\right)}\right)^{M_{2}}.\nonumber 
\end{eqnarray}
\end{lem}

}}}\end{center}
\begin{proof}
Using Dirac notation (\ref{eq:Pi_tilde_U_S}) we have
\begin{eqnarray*}
\tilde{\Pi}_{k',w',n}^{*}\tilde{\Pi}_{k,w,n} & = & |\tilde{\mathcal{S}}_{w'}^{\left(k'\right)}\rangle\langle\tilde{\mathcal{U}}_{L^{n}\left(w'\right)}^{\left(k'\right)},\tilde{\mathcal{U}}_{L^{n}\left(w\right)}^{\left(k\right)}\rangle_{L^{2}}\langle\tilde{\mathcal{S}}_{w}^{\left(k\right)}|.\rangle_{L^{2}}
\end{eqnarray*}
so
\begin{equation}
\mathrm{Tr}\left(\tilde{\Pi}_{k',w',n}^{*}\tilde{\Pi}_{k,w,n}\right)=\langle\tilde{\mathcal{U}}_{L^{n}\left(w'\right)}^{\left(k'\right)},\tilde{\mathcal{U}}_{L^{n}\left(w\right)}^{\left(k\right)}\rangle_{L^{2}}\cdot\langle\tilde{\mathcal{S}}_{w}^{\left(k\right)},\tilde{\mathcal{S}}_{w'}^{\left(k'\right)}\rangle_{L^{2}}.\label{eq:Tr_Pi_Pi}
\end{equation}
We first consider the term $\langle\tilde{\mathcal{U}}_{L^{n}\left(w'\right)}^{\left(k'\right)},\tilde{\mathcal{U}}_{L^{n}\left(w\right)}^{\left(k\right)}\rangle_{L^{2}}$.
Since the following estimates are uniform with respect to the words
$w,w'$, for simplicity we will estimate $\langle\tilde{\mathcal{U}}_{w'}^{\left(k'\right)},\tilde{\mathcal{U}}_{w}^{\left(k\right)}\rangle_{L^{2}}$
without the action of $L^{n}$. In the expression (\ref{eq:def_U_k_w})
of $|\tilde{\mathcal{U}}_{w}^{\left(k\right)}\rangle$, we have the
distribution $\hat{\chi}^{-1}\widehat{\mathcal{T}}_{w}\frac{1}{k!}\delta^{\left(k\right)}$
and the product formula for derivatives gives that
\[
\hat{\chi}^{-1}\widehat{\mathcal{T}}_{w}\frac{1}{k!}\delta^{\left(k\right)}=\sum_{l=0}^{k}c_{l}\nu^{\left(k-l\right)}\delta_{x_{w}}^{\left(l\right)}
\]
with constants $c_{l}$ independent on $w$ and $\nu$. Then 
\[
\langle\tilde{\mathcal{U}}_{w'}^{\left(k'\right)},\tilde{\mathcal{U}}_{w}^{\left(k\right)}\rangle_{L^{2}}\underset{(\ref{eq:def_U_k_w})}{=}\sum_{l'=0}^{k'}\sum_{l=0}^{k}c_{l'}c_{l}\nu^{\left(k+k'-l-l'\right)}\langle\delta_{x_{w'}}^{\left(l'\right)},\hat{A}_{m}^{2}\delta_{x_{w}}^{\left(l\right)}\rangle_{L^{2}}.
\]
We have
\[
\left|\langle\delta_{x_{w'}}^{\left(l'\right)},\hat{A}_{m}^{2}\delta_{x_{w}}^{\left(l\right)}\rangle_{L^{2}}\right|\leq\left\Vert \delta_{x_{w}}^{\left(l'\right)}\right\Vert _{H_{\nu}^{-m}}\left\Vert \delta_{x_{w}}^{\left(l\right)}\right\Vert _{H_{\nu}^{-m}}\underset{(\ref{eq:bound_Dirac_k})}{\leq}C\nu^{l+l'+1},
\]
so we have the general bound
\begin{equation}
\left|\langle\tilde{\mathcal{U}}_{w'}^{\left(k'\right)},\tilde{\mathcal{U}}_{w}^{\left(k\right)}\rangle_{L^{2}}\right|\leq C.\nu^{\left(k+k'+1\right)}\label{eq:gen_b1}
\end{equation}

Let us now suppose that $x_{w}\neq x_{w'}$. We use the non stationary
phase approximation and get that for any $M_{1}\geq0$,
\begin{eqnarray*}
\left|\langle\delta_{x_{w'}}^{\left(l'\right)},\hat{A}_{m}^{2}\delta_{x_{w}}^{\left(l\right)}\rangle_{L^{2}}\right| & \underset{(\ref{eq:def_norm_Hm})}{=} & \left|\int\langle\xi\rangle^{-2m}\mathcal{F}_{\nu}(\delta_{x_{w}}^{(l)})(\xi)\overline{\mathcal{F}_{\nu}(\delta_{x_{w}}^{(l')})(\xi)}d\xi\right|\\
 & \underset{}{=} & \frac{1}{2\pi}\nu^{1+l+l'}\left|\int e^{-i\nu\xi\left(x_{w}-x_{w'}\right)}\left\langle \xi\right\rangle ^{-2m}\xi^{l+l'}d\xi\right|\\
 & \leq & C_{M_{1}}\nu^{1+l+l'}\left(\frac{1/\nu}{\left|x_{w}-x_{w'}\right|}\right)^{M_{1}}
\end{eqnarray*}
hence
\[
\left|\langle\tilde{\mathcal{U}}_{w'}^{\left(k'\right)},\tilde{\mathcal{U}}_{w}^{\left(k\right)}\rangle_{L^{2}}\right|\leq C_{M_{1}}\nu^{\left(k+k'+1\right)}\left(\frac{1/\nu}{\left|x_{w}-x_{w'}\right|}\right)^{M_{1}}
\]
with $C_{M_{1}}$ independent on $\nu,w$. This also gives that for
any $n$:
\begin{equation}
\left|\langle\tilde{\mathcal{U}}_{L^{n}(w')}^{\left(k'\right)},\tilde{\mathcal{U}}_{L^{n}(w)}^{\left(k\right)}\rangle_{L^{2}}\right|\leq C_{M_{1}}\nu^{\left(k+k'+1\right)}\left(\frac{1/\nu}{\left|x_{L^{n}(w)}-x_{L^{n}(w')}\right|}\right)^{M_{1}}.\label{eq:gen_B1}
\end{equation}
 Let us consider now the second term $\langle\tilde{\mathcal{S}}_{w}^{\left(k\right)},\tilde{\mathcal{S}}_{w'}^{\left(k'\right)}\rangle_{L^{2}}$
in (\ref{eq:Tr_Pi_Pi}). We have $\tilde{\mathcal{S}}_{w}^{\left(k\right)}:=\hat{A}_{m}^{-1}\hat{\chi}\left(\widehat{\mathcal{T}}_{w}^{-1}\right)^{*}x^{k}$
and using the form of $\widehat{\mathcal{T}}_{w}$ given in Theorem
\ref{thm:Global-normal-form.} we can write
\[
\hat{\chi}\left(\widehat{\mathcal{T}}_{w}^{-1}\right)^{*}x^{k}=e^{i\nu\varUpsilon_{w}^{(0)}\left(x\right)}a_{w}\left(x\right)
\]
with $a_{w}\in C_{0}^{\infty}\left(\mathbb{R}\right)$ given by 
\[
a_{w}\left(x\right)=\chi\left(x\right)e^{\varUpsilon_{w}^{(1)}\left(x\right)}\frac{\left(H_{w}^{-1}\left(x-x_{w}\right)\right)^{k}}{\left|H'_{w}\left(H_{w}^{-1}\left(x-x_{w}\right)\right)\right|}.
\]
Recall that from the choice of $\chi$ in Section \ref{subsec:Extension}
we have $\mathrm{supp}\left(a_{w}\right)\subset K_{a}$. Furthermore,
from (\ref{eq:estimate_H_Upsilon}) we conclude that $a_{w}$ and
its derivatives are bounded on $K_{a}$and that these bounds are uniform
with respect to $w\in\mathcal{W}$ and $\nu$. Then
\[
\langle\tilde{\mathcal{S}}_{w}^{\left(k\right)},\tilde{\mathcal{S}}_{w'}^{\left(k'\right)}\rangle_{L^{2}}=\langle a_{w},e^{-i\nu\varUpsilon_{w}^{(0)}\left(x\right)}\hat{A}_{m}^{-2}e^{i\nu\varUpsilon_{w'}^{(0)}\left(x\right)}a_{w'}\left(x\right)\rangle
\]
 From Egorov's theorem, since $\hat{A}_{m}^{-2}=\mathrm{Op}_{\nu}\left(\left\langle \xi\right\rangle ^{2m}\right)$,
we have
\[
\hat{A}_{m}^{-2}e^{i\nu\varUpsilon_{w'}^{(0)}\left(x\right)}=e^{i\nu\varUpsilon_{w'}^{(0)}\left(x\right)}\hat{B}
\]
with $\hat{B}\in S\left(\left\langle \xi-\frac{d}{dx}\varUpsilon_{w'}^{(0)}\left(x\right)\right\rangle ^{2m}\right)$.
Thus $\hat{B}$ is a continuous operator on $\mathcal{S}(\mathbb{R})$,
thus $\tilde{a}_{w'}:=\hat{B}a_{w'}\in\mathcal{S}\left(\mathbb{R}\right)$
with all derivatives uniformly bounded with respect to $\nu,w$. This
gives
\begin{equation}
\langle\tilde{\mathcal{S}}_{w}^{\left(k\right)},\tilde{\mathcal{S}}_{w'}^{\left(k'\right)}\rangle_{L^{2}}=\langle a_{w},e^{i\nu\left(\varUpsilon_{w'}^{(0)}-\varUpsilon_{w}^{(0)}\right)\left(x\right)}\tilde{a}_{w'}\left(x\right)\rangle\label{eq:S_w_S_w}
\end{equation}
We deduce the general bound
\begin{equation}
\left|\langle\tilde{\mathcal{S}}_{w}^{\left(k\right)},\tilde{\mathcal{S}}_{w'}^{\left(k'\right)}\rangle_{L^{2}}\right|\leq C\label{eq:gen_b2}
\end{equation}
with $C$ independent of $\nu$ and $w$.

Let us now assumed $\min_{x\in K_{a}}\left(\left|\zeta_{w}\left(x\right)-\zeta_{w'}\left(x\right)\right|\right)\neq0$
then we want to bound the oscillating integral 
\[
\left|\langle\tilde{\mathcal{S}}_{w}^{\left(k\right)},\tilde{\mathcal{S}}_{w'}^{\left(k'\right)}\rangle_{L^{2}}\right|=\left|\int_{K_{a}}\;\overline{a_{w}(x)}\tilde{a}_{w'}\left(x\right)e^{i\nu\left(\varUpsilon_{w'}^{(0)}-\varUpsilon_{w}^{(0)}\right)\left(x\right)}dx\right|
\]
by partial integration. Recall that $\frac{d}{dx}\varUpsilon_{w}^{(0)}\left(x\right)=\mathcal{\zeta}_{w}(x)$,
thus the differential operator $L:=\frac{-i/\nu}{\zeta_{w}\left(x\right)-\zeta_{w'}\left(x\right)}\frac{d}{dx}$
fulfills $Le^{i\nu\left(\varUpsilon_{w'}^{(0)}-\varUpsilon_{w}^{(0)}\right)\left(x\right)}=e^{i\nu\left(\varUpsilon_{w'}^{(0)}-\varUpsilon_{w}^{(0)}\right)\left(x\right)}$
and we can insert an arbitrary power of this differential operator
in front of the oscillating phase. Note, that its $L^{2}$-dual is
given by
\[
L^{*}=\frac{\frac{i}{\nu}\left(\frac{d}{dx}\zeta_{w}\left(x\right)-\frac{d}{dx}\zeta_{w'}\left(x\right)\right)}{\left(\zeta_{w}\left(x\right)-\zeta_{w'}\left(x\right)\right)^{2}}+\frac{i/\nu}{\zeta_{w}\left(x\right)-\zeta_{w'}\left(x\right)}\frac{d}{dx}.
\]
Note that without any additional knowledge about the $\zeta_{w}$
partial integration would only allow us to obtain remainder terms
of the form 
\[
\left(\frac{\nu^{-1/2}}{\min_{x\in K_{a}}\left(\left|\zeta_{w}\left(x\right)-\zeta_{w'}\left(x\right)\right|\right)}\right)^{M_{2}},
\]
where the term $\nu^{-1/2}$ comes from a non stationary phase estimate.
We can however improve this estimate crucially if we take into account
the regularity of the invariant foliation of $\tilde{\phi}$. Let
us explain this in more detail: 

Recall that $\mathcal{\ensuremath{K}}$ was the hyperbolic set of
the map $\tilde{\phi}$ (cf. Remark \ref{rem:min_captive}(4)). Further
more this hyperbolic set has a precise description via the symbolic
dynamics, i.e. for any $(x,\xi)\in\mathcal{K}$ there is $w\in\mathcal{W}$
such that $(x,\xi)=(x_{w_{-}}$,$\zeta_{w_{+}}(x_{w_{-}})).$ Recall
furthermore (cf. Remark \ref{rem:invariant_manifolds}), that the
stable manifold through such a point $(x_{w_{-}},\zeta_{w_{+}}(x_{w_{-}}))$
is locally given by $\{(x,\zeta_{w_{+}}(x)),x\in I_{w_{-a,0}}\}$.
Now for general hyperbolic $C^{\infty}$ diffeomorphisms, the regularity
theory of the invariant manifolds implies, that the stable manifolds
are $C^{\infty}$ and that they depend Hölder continuously on the
base point w.r.t. the $C^{\infty}$-topology \cite{hirsch1970stable}.
Let us make precise what this means using our notation: If we fix
the value of $x$, i.e. if we restrict ourselves to an unstable manifold,
then the map
\[
g:\begin{cases}
\mathcal{K\cap}\{x=x_{w_{-}}\} & \to C^{\infty}(I_{w_{-a,0}},\mathbb{R})\\
(x_{w_{-},}\zeta_{w_{+}}(x_{w_{-}})) & \mapsto\zeta_{w_{+}}
\end{cases}
\]
that associates to a point in $\mathcal{K}$ the function describing
the unstable manifold, is Hölder continuous, where we put the metrizable
$C^{\infty}-$topology on the right side. As the hyperbolic map $\tilde{\phi}$
acts on a two dimensional space one even knows that the Hölder regularity
is $2-\varepsilon$ for any $\varepsilon>0$ (This is a direct consequence
of the more general regularity estimate in terms of bunching coefficients,
see \cite[Proposition 2.3.3]{hasselblatt_02b}). In particular the
map $g$ is Lipschitz\footnote{Note that we in fact only need Lipschitz continuity of the stable
foliation which might be an important observation for generalizations
to higher dimensional settings.}, thus for any $k$, there is $C_{k}>0$ such that for any $x\in I_{w_{-a,0}}$
\begin{align*}
\left|\left(\frac{d}{dx}\right)^{k}\zeta_{w_{+}}\left(x\right)-\left(\frac{d}{dx}\right)^{k}\zeta_{w'_{+}}\left(x\right)\right| & \leq C_{k}\left|\zeta_{w_{+}}\left(x_{w_{-}}\right)-\zeta_{w'_{+}}\left(x_{w_{-}}\right)\right|\\
 & \underset{(\ref{eq:estimate-1})}{\leq}C_{k}^{\prime}\min_{y\in I_{w_{-a,0}}}\left|\zeta_{w_{+}}\left(y\right)-\zeta_{w'_{+}}\left(y\right)\right|.
\end{align*}
As $K_{a}$ is a finite union of $I_{w_{-a,0}}$ we obtain 
\[
\max_{x\in K_{a}}\left|\left(\frac{d}{dx}\right)^{k}\zeta_{w_{+}}\left(x\right)-\left(\frac{d}{dx}\right)^{k}\zeta_{w'_{+}}\left(x\right)\right|\leq C_{k}\min_{x\in K{}_{a}}\left|\zeta_{w_{+}}\left(x\right)-\zeta_{w_{+}'}\left(x\right)\right|.
\]
Using this estimate, partial integration of (\ref{eq:S_w_S_w}) with
respect to $L$ yields that for any $M_{2}\geq0$,
\begin{equation}
\left|\langle\tilde{\mathcal{S}}_{w}^{\left(k\right)},\tilde{\mathcal{S}}_{w'}^{\left(k'\right)}\rangle_{L^{2}}\right|\leq C_{M_{2}}.\left(\frac{\nu^{-1}}{\min_{x\in K_{a}}\left(\left|\zeta_{w}\left(x\right)-\zeta_{w'}\left(x\right)\right|\right)}\right)^{M_{2}}\label{eq:gen_B2}
\end{equation}
with $C_{M_{2}}$ independent on $\nu,w$. Finally (\ref{eq:gen_b1}),(\ref{eq:gen_b2})
and (\ref{eq:Tr_Pi_Pi}) give (\ref{eq:general_bound}). Eq.(\ref{eq:gen_B1}),
(\ref{eq:gen_B2}) and (\ref{eq:Tr_Pi_Pi}) give (\ref{eq:Tr_proj_stat_phase_bound}).
\end{proof}
~
\begin{proof}[Proof of Proposition \ref{Prop:Separation-of-orbits}]
Let us summarize what we have obtained so far. Lemma \ref{lem:Separation-of-xw_zetaw}
gives us lower bounds
\[
|x_{L^{n}(w)}-x_{L^{n}(w')}|\geq Ce^{-J_{w_{n-n_{2}+1,n}}},\qquad\min_{x\in K_{a}}|\zeta_{w}(x)-\zeta_{w'}(x)|\geq Ce^{-J_{w_{0,n_{1}-1}}},
\]
and in (\ref{eq:Tr_proj_stat_phase_bound}) we have the terms
\[
\frac{\nu^{-1}}{\left|x_{L^{n}(w)}-x_{L^{n}(w')}\right|}\leq\frac{1}{C\nu}e^{J_{w_{n-n_{2}+1,n}}}
\]
\[
\frac{1/\nu}{\min_{x\in K_{a}}\left(\left|\zeta_{w}\left(x\right)-\zeta_{w'}\left(x\right)\right|\right)}\leq\frac{1}{C\nu}e^{J_{w_{0,n_{1}-1}}}
\]
that we would like to be ``small''. Therefore, for any $J_{c}>0$,
any $\varepsilon>0$, $\nu>0$ we take $n=\left[\frac{2}{J_{c}+\varepsilon}\log\nu\right]$
(equivalently $\nu\asymp e^{n\frac{J_{c+\varepsilon}}{2}}$). Then
the condition $n_{2}\leq N_{2}(w,n,J_{c})$ implies 
\[
\frac{1}{\nu}e^{J_{w_{n-n_{2}+1,n}}}\leq Ce^{-n\frac{J_{c+\varepsilon}}{2}}e^{n\frac{J_{c}}{2}}=e^{-\frac{\varepsilon}{2}n}
\]
and $n_{1}\leq N_{1}(w,n,J_{c})$ implies
\[
\frac{1}{\nu}e^{-J_{w_{0,n_{1}-1}}}\leq Ce^{-n\frac{J_{c+\varepsilon}}{2}}e^{n\frac{J_{c}}{2}}=e^{-\frac{\varepsilon}{2}n}.
\]
Consider a pair of words $w,w'\in\mathcal{W}$ with $w_{0,n}\neq w_{0,n}'$
with $n_{1}\leq N_{1}$ or $n_{2}\leq N_{2}$. From (\ref{eq:Tr_proj_stat_phase_bound}),
(\ref{eq:separation_in_x}) and definitions of $N_{1},N_{2}$ we get
\begin{eqnarray*}
\left|\mathrm{Tr}\left(\tilde{\Pi}_{k',w',n}^{*}\tilde{\Pi}_{k,w,n}\right)\right| & \leq & C_{M_{1},M_{2}}\nu^{\left(k+k'+1\right)}\left(\frac{1}{\nu}e^{J_{w_{n-n_{2}+1,n}}}\right)^{M_{1}}\left(\frac{1}{\nu}e^{J_{w_{0,n_{1}-1}}}\right)^{M_{2}}\\
 & \leq & C_{M}\nu^{(k+k'+1)}e^{-\frac{\varepsilon}{2}nM}
\end{eqnarray*}
where, if $n_{1}\leq N_{1}$ we have set $M_{2}=M$ and $M_{1}=0$
otherwise we have set $M_{1}=0$ and $M_{2}=M$. This finishes the
proof of Proposition \ref{Prop:Separation-of-orbits}.
\end{proof}

\section{\label{subsec:Proof-of-main_Theorem}Proof of the main Theorems \ref{thm:main_result}
and \ref{thm:main_result-resolvent}}

\subsection{Proof of Theorem \ref{thm:main_result}}

Let
\[
\gamma_{\mathrm{asympt.}}:=\limsup_{\nu\rightarrow+\infty}\left(\log\left(r_{s}\left(\mathcal{L}_{\nu,\chi}{}_{\upharpoonright H_{\nu}^{-m}}\right)\right)\right).
\]
We will proceed in few steps in order to bound from above $\gamma_{\mathrm{asympt.}}$
. The following Proposition gives an upper bound $\gamma_{\mathrm{up}}$
with a complicated expression that will be simplified later.

\begin{center}{\color{blue}\fbox{\color{black}\parbox{16cm}{
\begin{prop}
\label{prop:8.1}Under the assumption of minimal captivity (\ref{eq:hyp_minimal_captivity_epsilon})
we have 

\[
\gamma_{\mathrm{asympt.}}\leq\gamma_{\mathrm{up}}
\]
with
\begin{equation}
\gamma_{up}:=\inf_{0\leq J_{c}<2J_{min}}\left(\gamma\left(J_{c}\right)\right),\label{eq:gamma_up}
\end{equation}
\begin{equation}
\gamma\left(J_{c}\right):=\frac{J_{c}}{4}+\lim_{n\rightarrow\infty}\frac{1}{2n}\log\left(\sum_{w_{0,n}}e^{2\left(V-J\right)_{w_{0,N^{*}}}}e^{\left(V-J\right)_{w_{N^{*},n}}}\sum_{w'_{N^{*},n}}e^{\left(V-J\right)_{w'_{N^{*},n}}}\right).\label{eq:gamma_C}
\end{equation}
\[
N^{*}:=N^{*}\left(w_{0,n},J_{c}\right):=\mathrm{max}\left\{ k\leq n,\quad\mbox{s.t. }J_{w_{0,k}}<nJ_{c}\right\} .
\]
\end{prop}

}}}\end{center}
\begin{proof}
In order to estimate the spectral radius of the transfer operator
we use that for any $n\geq1$:
\begin{eqnarray}
r_{s}\left(\mathcal{L}_{\nu,\chi}{}_{\upharpoonright H_{\nu}^{-m}}\right) & = & r_{s}\left(\hat{Q}{}_{\upharpoonright L^{2}}\right)\leq\left\Vert \left(\hat{Q}^{n}\right)^{*}\hat{Q}^{n}\right\Vert _{L^{2}}^{1/\left(2n\right)}=\left\Vert P_{n}\right\Vert _{L^{2}}^{1/\left(2n\right)}.\label{eq:8-1}
\end{eqnarray}

Now we use Proposition \ref{thm:7.1} for some arbitrary $\varepsilon>0,$
$0<\beta<1$, $0\leq J_{c}\leq2J_{min}-\varepsilon$, $\nu>0$ and
$n=\left[\frac{2}{J_{c}+\varepsilon}\log\nu\right]$ and calculate

\begin{eqnarray*}
\log\left(r_{s}\left(\mathcal{L}_{\nu,\chi}{}_{\upharpoonright H_{\nu}^{-m}}\right)\right) & \leq & \frac{1}{2n}\log\left\Vert P_{n}\right\Vert _{L^{2}}\\
 & \underset{(\ref{eq:expansion_P_n})}{\leq} & \frac{1}{2n}\log\left(C\nu\sum_{w_{0,n}}e^{2\left(V-J\right)_{w_{0,N^{*}}}}e^{\left(V-J\right)_{w_{N^{*},n}}}\sum_{w'_{N^{*},n}}e^{\left(V-J\right)_{w'_{N^{*},n}}}+C\beta^{n}\right)\\
 & \underset{(\ref{eq:n_erhenfest})}{=} & \frac{J_{c}+\varepsilon}{4}+\frac{1}{2n}\log\left(C\right)\\
 &  & +\frac{1}{2n}\log\left(\sum_{w_{0,n}}e^{2\left(V-J\right)_{w_{0,N^{*}}}}e^{\left(V-J\right)_{w_{N^{*},n}}}\sum_{w'_{N^{*},n}}e^{\left(V-J\right)_{w'_{N^{*},n}}}+\frac{1}{\nu}\beta^{n}\right).
\end{eqnarray*}
As we can choose $\varepsilon>0$ and $0<\beta<1$ arbitrarily small
we deduce Proposition \ref{prop:8.1}. The fact that the limit exists
is explained in Remark \ref{rem:Fekete}.
\end{proof}

\subsubsection{Expression of $\gamma\left(J_{c}\right)$ in terms of topological
pressure}

We will now express $\gamma\left(J_{c}\right)$ in Eq.(\ref{eq:gamma_C})
in a concise way that will finally give the formulation in Theorem
\ref{thm:main_result}. For this we will use the topological pressure
$\mathrm{Pr}\left(.\right)$ defined in (\ref{eq:def_pressure-1}).
\begin{rem}
\label{rem:J_not_constant} Let us first remark that if the Jacobian
$J$ is equal to (or even cohomologous to) a constant $J_{0}$ , then
the expression of $\gamma\left(J_{c}\right)$ and of $\gamma_{\mathrm{up}}$
in Theorem \ref{thm:main_result} are obtained very easily: in this
case $\lim_{n\to\infty}\frac{1}{n}J_{w_{0,n}}=J_{0}$ and by choosing
$J_{c}=J_{0}$ one obtains $\gamma(J_{0})=\frac{J_{0}}{4}+\frac{1}{2}\mathrm{Pr}(2(V-J))$
which is precisely the upper bound (\ref{eq:def_gamma_up}) for $\gamma_{\mathrm{asympt.}}$
in Theorem \ref{thm:main_result}. The rest of this section will be
devoted to derive (\ref{eq:def_gamma_up}) in the case where $J$
is not cohomologous to a constant, which we will suppose from now
on.
\end{rem}

Let $0\leq J_{c}<2J_{min}$ and 
\[
\quad N_{min}:=n\frac{J_{c}}{J_{max}}=\frac{2}{J_{max}}\log\left(\nu\right)\frac{J_{c}}{J_{c}+\varepsilon},\quad N_{max}:=n\frac{J_{c}}{J_{min}}=2\frac{1}{J_{min}}\log\left(\nu\right)\frac{J_{c}}{J_{c}+\varepsilon}.
\]
Note that $N_{min}\leq N_{max}$. As $\frac{J_{c}}{J_{c}+\varepsilon}\approx1$
we can interpret them as twice the Ehrenfest time for the most expanding
and less expanding trajectory respectively. For every word $w_{0,n}$
we have $N^{*}\left(w_{0,n},J_{c}\right)\in\left[N_{min};\mathrm{min}\left(N_{max};n\right)\right]$.
Let us sort the words $w_{0,n}$ in the sum (\ref{eq:gamma_C}) according
to their values $N^{*}\left(w_{0,n},J_{c}\right)$: 
\begin{eqnarray*}
\gamma\left(J_{c}\right) & = & \frac{J_{c}}{4}+\lim_{n\rightarrow\infty}\frac{1}{2n}\log\left(\sum_{\overline{N}=N_{min}}^{\mathrm{min}\left(N_{max};n\right)}\sum_{w_{0,\overline{N}}\mbox{ s.t. }J_{w_{0,\overline{N}}}\leq nJ_{c}<J_{w_{0,\overline{N}+1}}}e^{2\left(V-J\right)_{w_{0,\overline{N}}}}e^{2\left(n-\overline{N}\right)\mathrm{Pr}\left(V-J\right)}\right.\\
 &  & \qquad\qquad\qquad\qquad\left.+\sum_{w_{0,n}\mbox{ s.t. }N^{*}=n}e^{2\left(V-J\right)_{w_{0,n}}}\right).
\end{eqnarray*}

In the above formula, $\mathrm{Pr}\left(V-J\right)$ appears by using
(\ref{eq:def_pressure-1}). This gives
\begin{equation}
\gamma\left(J_{c}\right)=\frac{J_{c}}{4}+\mathrm{max}\left(\mathcal{A};\mathcal{B}\right)\label{eq:gamma_Jc_max_AB}
\end{equation}
with
\[
\mathcal{A}:=\lim_{n\rightarrow\infty}\frac{1}{2n}\log\left(\sum_{\overline{N}=N_{min}}^{\mathrm{min}\left(N_{max};n\right)}\left(\sum_{w_{0,\overline{N}}\mbox{ s.t. }J_{w_{0,\overline{N}}}\leq nJ_{c}<J_{w_{0,\overline{N}+1}}}e^{2\left(V-J\right)_{w_{0,\overline{N}}}}\right)e^{2\left(n-\overline{N}\right)\mathrm{Pr}\left(V-J\right)}\right)
\]
\[
\mathcal{B}:=\lim_{n\rightarrow\infty}\frac{1}{2n}\log\left(\sum_{w_{0,n}\mbox{ s.t. }J_{w_{0,n}}<nJ_{c}}e^{2\left(V-J\right)_{w_{0,n}}}\right).
\]
Recall that we are interested in determining $\inf_{0\leq J_{c}<2J_{min}}\gamma(J_{c})$.
Let us first discuss what happens if $J_{c}>J_{max}.$ In this case
$\mathcal{A}=-\infty$ and $\mathcal{B}=\frac{1}{2}Pr(2(V-J))$ are
both independent of $J_{c}$. Consequently, through the additional
$\frac{J_{c}}{4}$ term in (\ref{eq:gamma_Jc_max_AB}), the quantity
$\gamma(J_{c})$ becomes monotonously increasing in this regime. Thus
in order to find $\inf_{0\leq J_{c}<2J_{min}}\gamma(J_{c})$ we can
from no on suppose, that we only consider $J_{c}\leq J_{max}$.

Next we want to give precise expressions for the terms $\mathcal{A}$
and $\mathcal{B}$ in term of the topological pressure functions.
For this purpose we use some large deviations results presented in
Appendix \ref{sec:A-formula-by}:

For the term $\mathcal{B}$ we use the formula (\ref{eq:formula_LD})
giving
\[
\mathcal{B}=\frac{1}{2}\max_{\overline{J}\in\left[J_{min};J_{c}\right]}\left(-v_{1}\left(\overline{J}\right)\right)
\]
with
\[
v_{1}\left(\overline{J}\right):=\beta\left(\overline{J}\right)\overline{J}-u\left(\beta\left(\overline{J}\right)\right)
\]
\[
u\left(\beta\right):=\mathrm{Pr}\left(2\left(V-J\right)+\beta J\right).
\]
and with $\beta\left(\overline{J}\right)$ such that
\begin{equation}
\overline{J}=u'\left(\beta\left(\overline{J}\right)\right).\label{eq:J_bar}
\end{equation}
Observe that $v_{1}\left(\overline{J}\right)$ is the Legendre transform
of the convex and increasing function $u\left(\beta\right)$. For
the term $\mathcal{A}$ we change the variable $\overline{N}$ by
$\overline{J}=\frac{nJ_{c}}{\overline{N}}$ and we also use the formula
(\ref{eq:formula_LD}) giving
\begin{eqnarray*}
\mathcal{A} & = & \lim_{n\rightarrow\infty}\frac{1}{2n}\log\left(\int_{J_{c}}^{J_{max}}d\overline{J}e^{\frac{nJ_{c}}{\overline{J}}\left(\mathrm{Pr}\left(2\left(V-J\right)+\beta\left(\overline{J}\right)J\right)-\beta\left(\overline{J}\right)\overline{J}\right)}e^{2\left(n-\frac{nJ_{c}}{\overline{J}}\right)\mathrm{Pr}\left(V-J\right)}\right)\\
 & = & \frac{1}{2}\max_{\overline{J}\in\left[J_{c};J_{max}\right]}\left(-v_{2}\left(\overline{J}\right)\right)
\end{eqnarray*}
with
\[
v_{2}\left(\overline{J}\right):=\frac{J_{c}}{\overline{J}}\left(v_{1}\left(\overline{J}\right)+2\mathrm{Pr}\left(V-J\right)\right)-2\mathrm{Pr}\left(V-J\right).
\]
In summary we have that 
\begin{equation}
\gamma\left(J_{c}\right)=\frac{J_{c}}{4}-\frac{1}{2}\min_{\overline{J}\in\left[J_{min};J_{max}\right]}v\left(\overline{J}\right)\label{eq:gamma_C_2}
\end{equation}
 with the function $v\left(\overline{J}\right)$ defined in two parts:
\begin{eqnarray}
v\left(\overline{J}\right): & = & v_{1}\left(\overline{J}\right)\mbox{ if }\overline{J}\in[J_{min};J_{c}]\label{eq:def_v}\\
 & = & v_{2}\left(\overline{J}\right)\mbox{ if }\overline{J}\in[J_{c};J_{max}].\nonumber 
\end{eqnarray}

\subsubsection{Minimization of $\gamma\left(J_{c}\right)$ to deduce $\gamma_{\mathrm{up}}$}

Considering (\ref{eq:gamma_up}) we finally want to minimize $\gamma(J_{c})$
in order to obtain the final expression (\ref{eq:def_gamma_up}) for
$\gamma_{\mathrm{up}}.$ Note that this minimization demands two steps:
first for a given $J_{c}$, $\gamma(J_{c})$ is given by (\ref{eq:gamma_C_2})
as a minimum over some parameter $\overline{J}.$ In a second step
we then have to minimize $\gamma(J_{c})$ for $J_{c}\leq2J_{min}.$ 

Let us start with the first step and fix $J_{c}<2J_{min}$ for the
moment. The function $v\left(\overline{J}\right)$ depends on the
parameter $J_{c}$ and $v_{2}\left(J_{c}\right)=v_{1}\left(J_{c}\right)$
hence $v\left(\overline{J}\right)$ is continuous on $]J_{min},J_{max}[$.
The function $v\left(\overline{J}\right)$ is depicted on Figure \ref{fig:function_v}.
Note that the function $v_{1}$ itself does not depend on $J_{c}$.
Our goal is to minimize the composite function $v(\bar{J})$ that
is piece-wise defined via the functions $v_{1},v_{2}$. However the
functions $v_{1}$ and $v_{2}$ are themselves both well defined on
the whole interval $[J_{min},J_{max}]$ and as a first step we look
for the minima of $v_{1}$ and $v_{2}$ on the whole interval $[J_{min},J_{max}]$
\begin{center}
\begin{figure}[h]
\begin{centering}
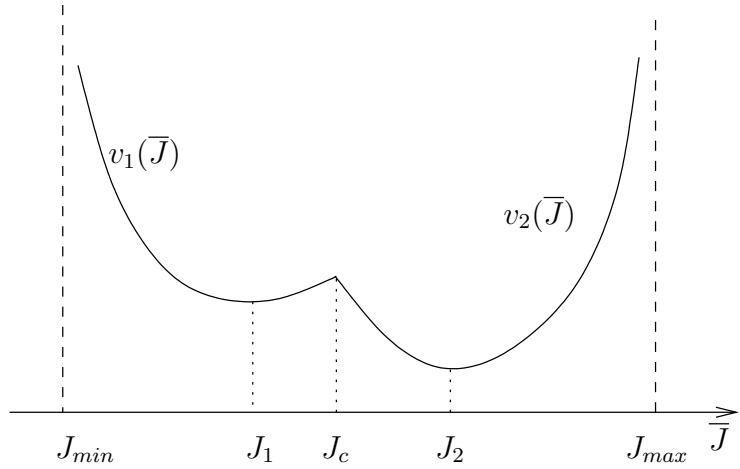
\par\end{centering}
\caption{\label{fig:function_v}The function $v\left(\overline{J}\right)$
defined in (\ref{eq:def_v}) depends on the parameter $J_{c}$ and
is defined piece-wise. We look for its global minimum $\min_{\overline{J}\in\left[J_{min};J_{max}\right]}v\left(\overline{J}\right)$.
It will finally turn out that the optimal value of $J_{c}$ is between
the local minima $J_{1}$ and $J_{2}$ of the functions $v_{1}$ and
$v_{2}$.}
\end{figure}
\par\end{center}

Recall that by Remark \ref{rem:J_not_constant} and Proposition \ref{prop:pressure_derivatives-1},
the function $v_{1}\left(\overline{J}\right)$ is strictly convex
and we compute that $v'_{1}\left(\overline{J}\right)=\beta\left(\overline{J}\right)$
hence its minimum is for $\overline{J}=J_{1}$ such that $\beta\left(J_{1}\right)=0$
giving
\[
\min_{\overline{J}}v_{1}\left(\overline{J}\right)=v_{1}\left(J_{1}\right)=-u\left(0\right)=-\mathrm{Pr}\left(2\left(V-J\right)\right)
\]
 Notice also that
\begin{equation}
\frac{d\beta\left(\overline{J}\right)}{d\overline{J}}=\frac{d^{2}v_{1}}{d\overline{J}^{2}}>0\label{eq:dbeta_dJ}
\end{equation}
so $\beta\left(\overline{J}\right)$ is increasing.

For the function $v_{2}\left(\overline{J}\right)$ we compute its
derivative
\begin{eqnarray*}
v_{2}'\left(\overline{J}\right) & = & -\frac{J_{c}}{\overline{J}^{2}}\left(v_{1}\left(\overline{J}\right)+2\mathrm{Pr}\left(V-J\right)\right)+\frac{J_{c}}{\overline{J}}v'_{1}\left(\overline{J}\right)\\
 & = & \frac{J_{c}}{\overline{J}^{2}}\left(-v_{1}\left(\overline{J}\right)-2\mathrm{Pr}\left(V-J\right)+\overline{J}\beta\left(\overline{J}\right)\right)\\
 & = & \frac{J_{c}}{\overline{J}^{2}}\left(u\left(\beta\left(\overline{J}\right)\right)-2\mathrm{Pr}\left(V-J\right)\right).
\end{eqnarray*}
Let $J_{2}\in\left[J_{min};J_{max}\right]$ be such that $v_{2}'\left(J_{2}\right)=0$,
i.e.
\[
u\left(\beta\left(J_{2}\right)\right)=2\mathrm{Pr}\left(V-J\right).
\]
We have 
\[
u(\beta(J_{2}))=2\mathrm{Pr}\left(V-J\right)>\mathrm{Pr}\left(2\left(V-J\right)\right)=u(\beta(J_{1}))
\]
 hence $J_{2}>J_{1}$ and from (\ref{eq:dbeta_dJ}) $\beta\left(J_{2}\right)>\beta\left(J_{1}\right)=0$.
Plugging $J_{2}$ into $v_{2}$ we get the minimum
\begin{eqnarray*}
\min_{\overline{J}}v_{2}\left(\overline{J}\right) & = & v_{2}\left(J_{2}\right)=\frac{J_{c}}{J_{2}}\left(v_{1}\left(J_{2}\right)+2\mathrm{Pr}\left(V-J\right)\right)-2\mathrm{Pr}\left(V-J\right)\\
 & = & J_{c}\beta\left(J_{2}\right)-2\mathrm{Pr}\left(V-J\right).
\end{eqnarray*}
\begin{rem}
The function $v_{1}$ hence its minimum $J_{1}$ also do not depend
on $J_{c}$. The value $J_{2}$ does not depend on $J_{c}$ neither
but $v_{2}\left(J_{2}\right)$ depends on $J_{c}$.
\end{rem}

The two local minima coincide $v_{1}\left(J_{1}\right)=v_{2}\left(J_{2}\right)$
for the parameter $J_{c}=\left\langle J\right\rangle $ given by
\begin{equation}
\left\langle J\right\rangle :=\frac{2\mathrm{Pr}\left(V-J\right)-\mathrm{Pr}\left(2\left(V-J\right)\right)}{\beta\left(J_{2}\right)}.\label{eq:J_average_def}
\end{equation}
Since $u\left(\beta\right)$ is strictly convex and $\beta\left(\overline{J}\right)$
is increasing we have that 
\begin{eqnarray*}
u'\left(\beta\left(J_{1}\right)\right) & < & \frac{u\left(\beta\left(J_{2}\right)\right)-u\left(\beta\left(J_{1}\right)\right)}{\beta\left(J_{2}\right)-\beta\left(J_{1}\right)}<u'\left(\beta\left(J_{2}\right)\right)
\end{eqnarray*}
and using (\ref{eq:J_bar}) and (\ref{eq:J_average_def}) this gives
\[
J_{1}<\left\langle J\right\rangle <J_{2}.
\]
Note that up to now we considered the minima of $v_{1}$ and $v_{2}$
independently. Now let us consider the minimum of the composite function
$v$ which will give us a concrete value for $\gamma(J_{c})=\frac{J_{c}}{4}-\frac{1}{2}\inf_{\overline{J}}v\left(\overline{J}\right)$:

We finally deduce for (\ref{eq:gamma_C_2}) that
\begin{itemize}
\item If $J_{c}\leq\left\langle J\right\rangle $ then
\begin{eqnarray}
\gamma\left(J_{c}\right) & = & \frac{J_{c}}{4}-\frac{1}{2}\inf_{\overline{J}}v_{2}\left(\overline{J}\right)\nonumber \\
 & = & \frac{J_{c}}{4}-\frac{1}{2}J_{c}\beta\left(J_{2}\right)+\mathrm{Pr}\left(V-J\right).\label{eq:gamma_Jc}
\end{eqnarray}
\item If $J_{c}\geq\left\langle J\right\rangle $ then
\begin{eqnarray*}
\gamma\left(J_{c}\right) & = & \frac{J_{c}}{4}-\frac{1}{2}\inf_{\overline{J}}v_{1}\left(\overline{J}\right)\\
 & = & \frac{J_{c}}{4}+\frac{1}{2}\mathrm{Pr}\left(2\left(V-J\right)\right)
\end{eqnarray*}
that is minimal for $J_{c}=\left\langle J\right\rangle $.
\end{itemize}
From (\ref{eq:gamma_Jc}), we have $\gamma'\left(J_{c}\right)_{/J_{c}\leq\left\langle J\right\rangle }=\frac{1}{2}\left(\frac{1}{2}-\beta\left(J_{2}\right)\right)$
hence if $\beta\left(J_{2}\right)>\frac{1}{2}$ then
\begin{eqnarray*}
\gamma_{\mathrm{up}} & \underset{(\ref{eq:gamma_up})}{=} & \inf_{J_{c}\leq\left\langle J\right\rangle }\left(\gamma\left(J_{c}\right)\right)=\gamma\left(\left\langle J\right\rangle \right)\\
 & = & \frac{\left\langle J\right\rangle }{4}+\frac{1}{2}\mathrm{Pr}\left(2\left(V-J\right)\right),
\end{eqnarray*}
otherwise if $\beta\left(J_{2}\right)<\frac{1}{2}$ then
\[
\gamma_{\mathrm{up}}=\gamma\left(0\right)=\mathrm{Pr}\left(V-J\right).
\]
We have finished the proof of the main Theorem \ref{thm:main_result}.

\subsection{Proof of Theorem \ref{thm:main_result-resolvent}}

From (\ref{eq:def_Pn}) and (\ref{eq:Q_F}) we have
\[
\frac{1}{n}\log\left\Vert \mathcal{L}_{\nu,\chi}^{n}\right\Vert _{H_{\nu}^{-m}}=\frac{1}{2n}\log\left\Vert P_{n}\right\Vert _{L^{2}}.
\]
In Section \ref{subsec:Proof-of-main_Theorem} we have shown that
for $n$ related to $\nu$ by $n=\left[\frac{2}{\left\langle J\right\rangle +\varepsilon}\log\nu\right]$
we have for $\nu\rightarrow\infty$,
\[
\frac{1}{2n}\log\left\Vert P_{n}\right\Vert _{L^{2}}\leq\gamma_{\mathrm{up}}+o\left(1\right).
\]
We deduce that for any $\epsilon>0$, $\exists\nu_{0}>0$, $\forall\nu>\nu_{0}$,
\[
\frac{1}{n}\log\left\Vert \mathcal{L}_{\nu,\chi}^{n}\right\Vert _{H_{\nu}^{-m}}\leq\gamma_{\mathrm{up}}+\epsilon
\]
\[
\left\Vert \mathcal{L}_{\nu,\chi}^{n}\right\Vert _{H_{\nu}^{-m}}\leq e^{n\left(\gamma_{\mathrm{up}}+\epsilon\right)}.
\]
In \cite[thm 2.9, or proof of thm 2.11]{faure_arnoldi_tobias_13}
we have shown that for any $r$,
\[
\left\Vert \mathcal{L}_{\nu,\chi}^{r}\right\Vert _{H_{\nu}^{-m}}\leq C_{0}e^{r\left(\gamma_{\mathrm{sc}}+\epsilon\right)}.
\]
Let us suppose that $\gamma_{\mathrm{up}}<\gamma_{\mathrm{sc}}$ in
order to improve this bound (otherwise Theorem \ref{thm:main_result-resolvent}
is covered by \cite[thm 2.9]{faure_arnoldi_tobias_13}). For any $t\in\mathbb{N}$,
we write $t=Nn+r$ with $r\leq n$, $N\in\mathbb{N}$, and we have
\begin{eqnarray*}
\left\Vert \mathcal{L}_{\nu,\chi}^{t}\right\Vert _{H_{\nu}^{-m}} & \leq & \left\Vert \mathcal{L}_{\nu,\chi}^{n}\right\Vert _{H_{\nu}^{-m}}^{N}\left\Vert \mathcal{L}_{\nu,\chi}^{r}\right\Vert _{H_{\nu}^{-m}}\\
 & \leq & e^{Nn\left(\gamma_{\mathrm{up}}+\epsilon\right)}C_{0}e^{r\left(\gamma_{\mathrm{sc}}+\epsilon\right)}\\
 & \leq & C_{0}e^{t\left(\gamma_{\mathrm{up}}+\epsilon\right)}e^{r\left(\gamma_{\mathrm{sc}}-\gamma_{\mathrm{up}}\right)}\\
 & \leq & C_{0}e^{t\left(\gamma_{\mathrm{up}}+\epsilon\right)}e^{n\left(\gamma_{\mathrm{sc}}-\gamma_{\mathrm{up}}\right)}\\
 & \leq & C_{0}e^{t\left(\gamma_{\mathrm{up}}+\epsilon\right)}\nu^{\frac{2}{\left\langle J\right\rangle +\epsilon}\left(\gamma_{\mathrm{sc}}-\gamma_{\mathrm{up}}\right)}.
\end{eqnarray*}
The relation $\left(z-\mathcal{L}_{\nu,\chi}\right)^{-1}=z^{-1}\sum_{t\geq0}\left(\frac{\mathcal{L}_{h,\chi}}{z}\right)^{t}$
gives that 
\begin{eqnarray*}
\left\Vert \left(z-\mathcal{L}_{\nu,\chi}\right)^{-1}\right\Vert _{H_{\nu}^{-m}} & \leq & \left|z\right|^{-1}\sum_{t\geq0}\frac{\left\Vert \mathcal{L}_{\nu,\chi}^{t}\right\Vert _{H_{\nu}^{-m}}}{\left|z\right|^{t}}\leq\left|z\right|^{-1}C_{0}\nu^{\frac{2}{\left\langle J\right\rangle +\epsilon}\left(\gamma_{\mathrm{sc}}-\gamma_{\mathrm{up}}\right)}\sum_{t\geq0}\frac{e^{t\left(\gamma_{\mathrm{up}}+\epsilon\right)}}{\left|z\right|^{t}}\\
 & = & \frac{C_{0}\nu^{\frac{2}{\left\langle J\right\rangle +\epsilon}\left(\gamma_{\mathrm{sc}}-\gamma_{\mathrm{up}}\right)}}{\left|z\right|-e^{\left(\gamma_{\mathrm{up}}+\epsilon\right)}}\leq C_{1}\nu^{\frac{2}{\left\langle J\right\rangle +\epsilon}\left(\gamma_{\mathrm{sc}}-\gamma_{\mathrm{up}}\right)}
\end{eqnarray*}
We have finished the proof of Theorem \ref{thm:main_result-resolvent}.

\appendix
\newpage

\section{\label{sec:Examples}Examples}

In this section we compare the upper bound on the spectral gap $\gamma_{\textup{up}}$
from Theorem \ref{thm:main_result} with the so far known bounds $\gamma_{\textup{Gibbs}}=\mathrm{Pr}\left(V-J\right)$
and $\gamma_{\textup{sc}}=\mathrm{tsup}\left(V-\frac{1}{2}J\right)$
for two explicit examples of IFS according to Definition \ref{def:IFSAn-iterated-function}:
a ``two branched linear IFS'' and the ``truncated Gauss map''
which plays an important role in the study of continued fraction expansion. 

\subsection{Linear IFS}

Within the class of dynamical systems, treated in this article, the
linear IFS is perhaps the most simple, nevertheless non trivial example.

\subsubsection{Definition of the model}

See Figure \ref{fig:Linear_IFS}.

\begin{center}{\color{red}\fbox{\color{black}\parbox{16cm}{
\begin{defn}
\label{def:lin_ifs}''\textbf{Linear IFS''.} Let $0<a_{1}<b_{1}<a_{2}<b_{2}<1$.
Consider the intervals $I_{1}:=[a_{1},b_{1}]$ and $I_{2}:=[a_{2},b_{2}]$
and the adjacency matrix $A=\left(\begin{array}{cc}
1 & 1\\
1 & 1
\end{array}\right)$. The contracting maps are the linear functions:

\[
\phi_{i,j}:\begin{cases}
I_{i} & \rightarrow I_{j}\\
x & \mapsto a_{j}+\left(b_{j}-a_{j}\right)x
\end{cases}.
\]
\end{defn}

}}}\end{center}

The Jacobian function $J_{i,j}\left(x\right):=-\log\frac{d\phi_{i,j}}{dx}$
is constant on intervals $x\in I_{i}$: 
\[
J_{j}:=J_{i,j}\left(x\right)=-\log(b_{j}-a_{j})>0.
\]

\begin{figure}[h]
\centering 
\centering{}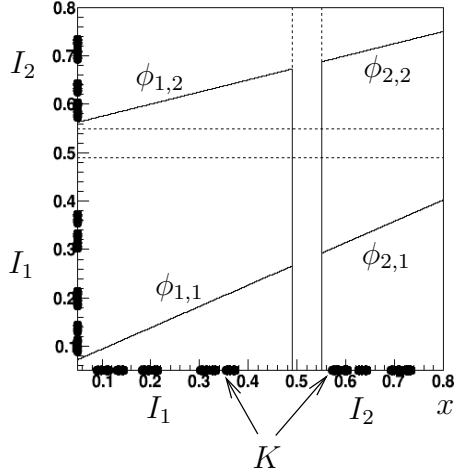\caption{\label{fig:Linear_IFS}Graphs of $\phi_{i,j}$ a the linear IFS of
Definition \ref{def:lin_ifs} with interval $I_{1}=\left[0.05,0.49\right]$,
$I_{2}=\left[0.55,0.8\right]$ giving Jacobians $J_{1}=0.821\ldots$
and $J_{2}=1.38\ldots$. The trapped set $K$ defined in (\ref{eq:trapped_set_K_def1})
is a dyadic Cantor set of Hausdorff dimension $\delta=\dim_{H}\left(K\right)=0.643\ldots$
given by (\ref{eq:def_delta-1}).}
\end{figure}

The topological pressure takes a particularly simple form:

\begin{center}{\color{blue}\fbox{\color{black}\parbox{16cm}{
\begin{lem}
Let $\phi_{i,j}$ be a linear IFS with Jacobians $J_{1},J_{2}>0$.
Then the topological pressure function (\ref{eq:def_P_beta-1}) is
given by 
\begin{equation}
P\left(\beta\right)=\log\left(e^{-\beta J_{1}}+e^{-\beta J_{2}}\right).\label{eq:P_beta_linear}
\end{equation}
\end{lem}

}}}\end{center}
\begin{proof}
We have 
\[
P(\beta)\underset{(\ref{eq:def_P_beta-1})}{=}\mathrm{Pr}\left(-\beta J\right)=\lim_{n\to\infty}\frac{1}{n}\log\left(\sum_{w_{0,n}}e^{-\beta J_{w_{0,n}}(x_{w})}\right).
\]
The fact that $J(x)$ is constant in $I_{1}$ and $I_{2}$ gives $J_{w_{0,n}}(x)=\sum_{k=1}^{n}J_{w_{k}}$
and since the adjacency matrix $A$ is full, 
\[
\sum_{w_{0,n}}e^{-\beta J_{w_{0,n}}(x_{w})}=\sum_{w_{0,n}}e^{-\beta\sum_{k=1}^{n}J_{w_{k}}}=\bigg(e^{-\beta J_{1}}+e^{-\beta J_{2}}\bigg)^{n}.
\]
\end{proof}
In order to apply Theorem \ref{thm:main_result} we choose a roof
function $\tau:I\to\mathbb{R}$ such that the minimal captivity assumption
is fulfilled. This can be achieved by a piece-wise linear function.
\begin{center}{\color{blue}\fbox{\color{black}\parbox{16cm}{
\begin{lem}
\label{lem:lin_IFS_min_captive} Let $\phi_{i,j}$ be a linear IFS
as defined in Definition \ref{def:lin_ifs} and suppose that $0<J_{1}\leq J_{2}$.
Let 
\begin{equation}
\tau(x)=\tau_{i}\cdot x,\quad x\in I_{i},\label{eq:tau_linear_IFS}
\end{equation}
with $\tau_{1}:=0$ and $\tau_{2}:=1$. Then the minimal captivity
assumption (Assumption \ref{hyp:minimal_capt}) is fulfilled. 
\end{lem}

}}}\end{center}

\begin{figure}[h]
\centering 
\begin{centering}
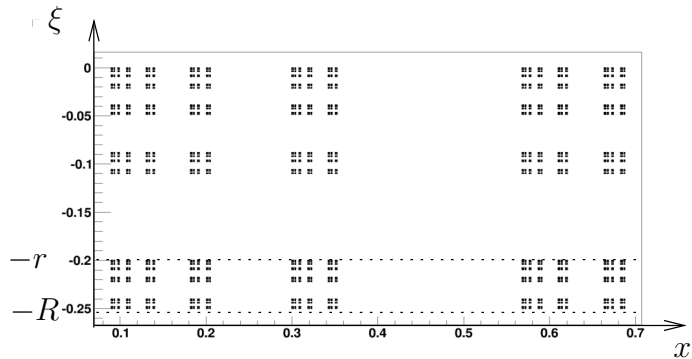
\par\end{centering}
\caption{\label{fig:lin_IFS_trappes_set}Trapped set $\mathcal{K}$ in phase
space $T^{*}\mathbb{R}$ for the linear IFS of Figure \ref{fig:Linear_IFS}
constructed with a linear function $\tau$ given in (\ref{eq:tau_linear_IFS}).
One clearly sees its Cantor set nature. The dashed lines indicate
the values $-R,-r$ that appear in the proof of Lemma~\ref{lem:lin_IFS_min_captive}. }
\end{figure}
\begin{proof}
With the above definitions, the canonical map takes a particularly
simple form 
\[
(x',\xi')=\tilde{\phi}_{i,j}(x,\xi)\underset{(\ref{eq:def_symplectic_map_Fij-1})}{=}(\phi_{i,j}(x),e^{J_{j}}\xi+\tau_{j}).
\]
Let $x\in I$, $\xi>0$. We have $\xi'\geq e^{J_{1}}\xi$ hence any
trajectory starting from positive $\xi$ escapes to infinity. This
implies that the trapped set is $\mathcal{K}\subset I\times[-\infty,0]$.

Let $R:=\frac{1}{e^{J_{2}}-1}>0$. For $j=1$ we get $\xi'=e^{J_{1}}\xi$
and for $j=2$ we get
\[
\left(\xi'+R\right)=\left(e^{J_{2}}\xi+1+\frac{1}{e^{J_{2}}-1}\right)=e^{J_{2}}\left(\xi+R\right).
\]
This implies that all trajectories starting with $\xi<-R$ escape
towards minus infinity and $\mathcal{K}\subset I\times[-R,0]$.

Let $(x,\xi)\in\mathcal{K}$ and $i$ such that $x\in I_{i}$. In
order to prove minimal captivity we have to show that either $(x_{1}',\xi_{1}')=\tilde{\phi}_{i,1}(x,\xi)\notin\mathcal{K}$
or $(x_{2}',\xi_{2}')=\tilde{\phi}_{i,2}(x,\xi)\notin\mathcal{K}$.
Let $r:=e^{-J_{2}}<R$. If $-r<\xi\leq0$ then $\xi'_{2}=e^{J_{2}}\xi+1>0$
which implies $(x_{2}',\xi_{2}')\notin\mathcal{K}$. If $-R\leq\xi\leq-r$
then $\xi'_{1}=e^{J_{1}}\xi\leq-e^{J_{1}-J_{2}}$. We use the constraint\footnote{Obviously this constraint is not necessary for the proof but simplifies
the choice of the constants.} $I_{1}\cup I_{2}\subset[0,1]$ that gives 
\begin{equation}
e^{-J_{1}}+e^{-J_{2}}<1.\label{eq:ineq}
\end{equation}
We deduce $\xi'_{1}<-R$ and $(x_{1}',\xi_{1}')\notin\mathcal{K}$.
\end{proof}

\subsubsection{Estimates for the asymptotic spectral gap $\gamma_{\mathrm{asympt.}}$ }

Let us now consider the asymptotic spectral radius of the family of
transfer operators $\mathcal{L}_{\nu}$ for a linear IFS with unstable
Jacobians $0<J_{1}\leq J_{2}$ and $\tau$ as in Lemma \ref{lem:lin_IFS_min_captive}
and with a potential function of the form 
\begin{equation}
V(x)=\left(1-a\right)J\left(x\right),\quad a\in\mathbb{R}.\label{eq:V_a}
\end{equation}
We recall that the value $a=0$, giving $V=J$ is interesting for
counting orbits (\ref{eq:h_top}) and that $a=1/2$ is the ``quantum
case''\cite{faure-tsujii_prequantum_maps_12}. The different upper
bound estimates for the asymptotic spectral gap $\gamma_{\mathrm{asympt.}}$
can be expressed very explicitly as follows. 
\begin{align}
\gamma_{\textup{Gibbs}} & \underset{(\ref{eq:def_gamma_Gibbs})}{:=}\mathrm{Pr}\left(V-J\right)=P\left(a\right)=\log\left(e^{-aJ_{1}}+e^{-aJ_{2}}\right)\label{eq:gamma_linear_case}\\
\gamma_{\textup{sc}} & \underset{(\ref{eq:def_gamma_sc})}{:=}\mathrm{tsup}\left(V-\frac{1}{2}J\right)=\left(\frac{1}{2}-a\right)J_{1}\mbox{ if }a\geq\frac{1}{2},\quad\left(\frac{1}{2}-a\right)J_{2}\mbox{ if }a\leq\frac{1}{2}.\nonumber \\
\gamma_{\textup{conj}} & \underset{(\ref{eq:def_gamma_conj})}{:=}\frac{1}{2}\mathrm{Pr}\left(2\left(V-J\right)\right)=\frac{1}{2}P\left(2a\right)\nonumber \\
\gamma_{\textup{up}} & \underset{(\ref{eq:def_gamma_up})}{:=}\gamma_{\textup{conj}}+\frac{1}{4}\left\langle J\right\rangle ,\quad\left\langle J\right\rangle :=\frac{2}{\beta_{0}}\left(\gamma_{\mathrm{Gibbs}}-\gamma_{\textup{conj}}\right)\nonumber 
\end{align}
where $\beta_{0}$ solves the equation 
\begin{equation}
P\left(2a-\beta_{0}\right)=2P\left(a\right).\label{eq:def_beta0}
\end{equation}

Let us introduce 
\[
\omega:=\exp\left(J_{1}-J_{2}\right)\in]0,1]
\]
 that measures the non homogeneity of the Jacobian. Note that up to
dynamical equivalence the linear IFS is uniquely determined by $\delta$
and $\omega$ where $\delta$ is the Hausdorff dimension of the trapped
set $K$ defined in (\ref{eq:def_delta-1}). Thus given a fixed potential
and a set of parameters $\left(\delta,\omega\right)$ we can ask the
question which of the known estimates $\gamma_{\mathrm{sc}},\gamma_{\mathrm{Gibbs}}$
and $\gamma_{\mathrm{up}}$ is the best i.e. lowest one. This leads
to a partition of the $\delta,\omega$ parameter space which is shown
in Figure \ref{fig:phases_Linear_IFS} for three different choices
of $V$. One observes that $\gamma_{\mathrm{up}}$ obtained in this
paper gives the best (lowest) result (in grey domain) for intermediate
values of $\delta$ and for $a\neq1/2$. $\gamma_{\mathrm{Gibbs}}$
is better for small values of $\delta$ (i.e. very open system) or
very small values of $\omega$ (i.e. very in-homogeneous Jacobian)
whereas $\gamma_{\mathrm{sc}}$ is better for large values of $\delta$
and $\omega$ (i.e. closed system with homogeneous Jacobian).

Figure \ref{fig:gamma_Linear_IFS} is a plot of $\gamma_{\mathrm{sc}},\gamma_{\mathrm{Gibbs}},\gamma_{\mathrm{up}},\gamma_{conj}$
as functions of $\delta=\dim_{H}\left(K\right)\in]0,1[$ and for $\omega=0.5$.

\begin{figure}[h]
\begin{centering}
\scalebox{0.9}[0.9]{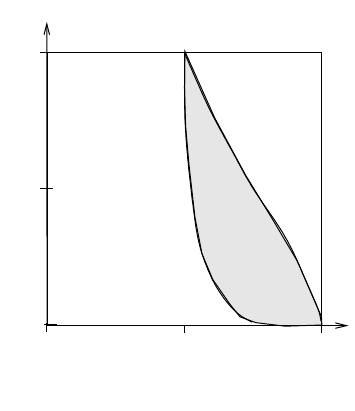}\scalebox{0.9}[0.9]{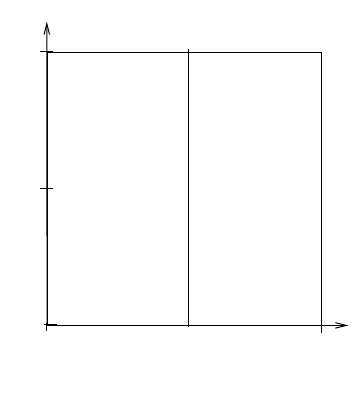}\scalebox{0.9}[0.9]{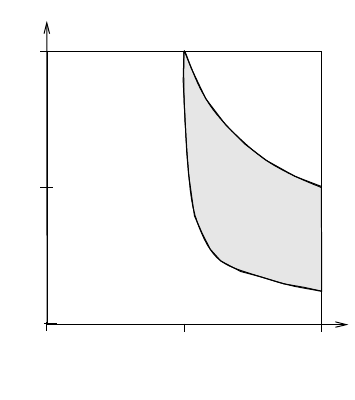}
\par\end{centering}
\caption{\label{fig:phases_Linear_IFS}''\textbf{Phase diagram}'' in the
domain $\left(\delta,\omega\right)\in]0,1[^{2}$ with $\delta=\dim_{H}\left(K\right)\in]0,1[$
being the Hausdorff dimension of the trapped set $K$ given by (\ref{eq:def_delta-1})
and $\omega:=\exp\left(J_{1}-J_{2}\right)\in]0,1]$ that measures
the homogeneity of the Jacobian (we have $\omega=1$ if $J_{1}=J_{2}$).
The three different plots correspond to three different potential
functions $V.$ For each of the three potentials and each value $\left(\delta,\omega\right)$
and $a$, we indicate by a numerical calculation, which value $\gamma_{\mathrm{Gibbs}},\gamma_{\mathrm{sc}},\gamma_{\mathrm{up}}$
is the lowest. }
\end{figure}

\begin{figure}[h]
\begin{centering}
\scalebox{0.9}[0.9]{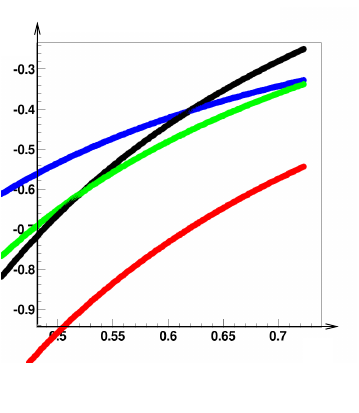}\scalebox{0.9}[0.9]{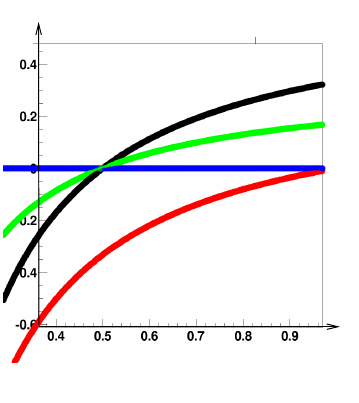}\scalebox{0.9}[0.9]{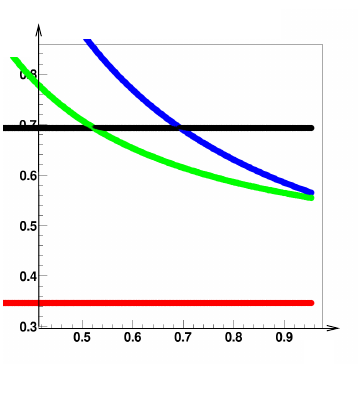}
\par\end{centering}
\caption{\label{fig:gamma_Linear_IFS}Plot of various estimates $\gamma_{\mathrm{sc}},\gamma_{\mathrm{Gibbs}},\gamma_{\mathrm{up}},\gamma_{conj}$
defined in Section (\ref{sec:The-main-result}) for the linear IFS
model with $\omega=0.5$, as a function of $\delta=\dim_{H}\left(K\right)\in]0,1[$
.}
\end{figure}

\begin{center}{\color{blue}\fbox{\color{black}\parbox{16cm}{
\begin{prop}
For a linear IFS with roof function $\tau$ given in (\ref{eq:tau_linear_IFS})
we have the following three properties that appear on Figure \ref{fig:phases_Linear_IFS}:

\begin{enumerate}
\item For any $a\in\mathbb{R}$, potential $V=\left(1-a\right)J$, and $\omega=1$
(i.e. homogeneous case $J=J_{1}=J_{2}$) we have $\gamma_{\mathrm{Gibbs}}<\min(\gamma_{\mathrm{sc}},\gamma_{\mathrm{up}})$
if $\delta<0.5$ and $\gamma_{\mathrm{sc}}<\min\left(\gamma_{\mathrm{Gibbs}},\gamma_{\mathrm{up}}\right)$
if $\delta>0.5$. For $\delta=0.5$ we have $\gamma_{\mathrm{sc}}=\gamma_{\mathrm{Gibbs}}=\gamma_{\mathrm{up}}$.
\item For the potential $V\left(x\right)=0$, for any $\omega\in]0,1[$
there exists $\delta=\delta\left(\omega\right)$ such that $\gamma_{\mathrm{up}}<\gamma_{Gibss}=\gamma_{\mathrm{sc}}$.
\item For the potential $V\left(x\right)=\frac{1}{2}J\left(x\right)$, for
any $\omega\in]0,1[$, we have $\gamma_{\mathrm{Gibbs}}<\min(\gamma_{\mathrm{sc}},\gamma_{\mathrm{up}})$
if $\delta<0.5$ and $\gamma_{\mathrm{sc}}<\min\left(\gamma_{\mathrm{Gibbs}},\gamma_{\mathrm{up}}\right)$
if $\delta>0.5$.
\end{enumerate}
\end{prop}

}}}\end{center}
\begin{proof}
Proof of case (1). Suppose $J=J_{1}=J_{2}$. This is the case $\omega:=\exp\left(J_{1}-J_{2}\right)=1$
on Figure \ref{fig:phases_Linear_IFS}. We have 
\[
P\left(\beta\right)\underset{(\ref{eq:P_beta_linear})}{=}\log\left(2e^{-\beta J}\right)=h_{top}-\beta J
\]
 with $h_{top}\underset{(\ref{eq:def_htop-1})}{=}P\left(0\right)=\log2$
being the topological entropy. Then the Hausdorff dimension $\delta=\dim_{H}\left(K\right)$
of the trapped set $K$, given by $P\left(\delta\right)\underset{(\ref{eq:def_delta-1})}{=}0$,
is $\delta=\frac{h_{top}}{J}$. We get 
\begin{alignat*}{1}
\gamma_{\textup{Gibbs}} & =h_{top}\left(1-\frac{a}{\delta}\right)\\
\gamma_{\textup{sc}} & =h_{top}\left(\frac{1}{2}-a\right)\cdot\frac{1}{\delta}\\
\gamma_{\textup{up}} & =h_{top}\left(\frac{1}{2}+\left(\frac{1}{4}-a\right)\frac{1}{\delta}\right),\qquad\beta_{0}\underset{(\ref{eq:def_beta0})}{=}\delta,\quad\left\langle J\right\rangle =J\\
\gamma_{\textup{conj}} & =h_{top}\left(\frac{1}{2}-\frac{a}{\delta}\right)
\end{alignat*}
We deduce that for $\delta<\frac{1}{2}$ then $\gamma_{\mathrm{Gibbs}}<\min(\gamma_{\mathrm{sc}},\gamma_{\mathrm{up}})$,
for $\delta>\frac{1}{2}$ then $\gamma_{\mathrm{sc}}<\min(\gamma_{\mathrm{Gibbs}},\gamma_{\mathrm{up}})$
and for $\delta=\frac{1}{2}$ then $\gamma_{\mathrm{sc}}=\gamma_{\mathrm{Gibbs}}=\gamma_{\mathrm{up}}$.

Proof of case (2). Suppose $V=0$. For a given $0<\omega<1$ we choose
$0<J_{1}<2h_{top}=2\log2$ such that 
\begin{equation}
e^{J_{1}/2}=1+\omega.\label{eq:ref_alpha}
\end{equation}
From $\omega=\exp\left(J_{1}-J_{2}\right)$ this gives a value of
$J_{2}$ and $\delta$. Eq.(\ref{eq:gamma_linear_case}) gives $\gamma_{\textup{Gibbs}}=\gamma_{\textup{sc}}$.
According to the statement of Theorem\emph{\ref{thm:main_result},
}if we show that $\beta_{0}>1/2$, this implies that $\gamma_{\textup{up}}<\gamma_{\textup{Gibbs}}$.
From (\ref{eq:def_beta0}) the condition for $\beta_{0}$ is 
\[
P(2-\beta_{0})-2P(1)=0.
\]
As $P(2-\beta)-2P(1)$ is strictly increasing in $\beta$ it is sufficient
to show that $P\left(2-\frac{1}{2}\right)-2P(1)<0.$ Using that $P\left(1\right)=-\frac{1}{2}J_{1}$
from our choice $\gamma_{\textup{Gibbs}}=\gamma_{\textup{sc}}$, we
compute 
\begin{eqnarray*}
P(3/2)-2P(1) & \underset{(\ref{eq:P_beta_linear})}{=} & \log\left(e^{-3/2J_{1}}(1+e^{-3/2\left(J_{2}-J_{1}\right)})\right)+J_{1}\\
 & = & \log\left(e^{-1/2J_{1}}(1+e^{-3/2\left(J_{2}-J_{1}\right)})\right)\\
 & < & \log\left(e^{-1/2J_{1}}(1+\omega)\right)\\
 & \underset{(\ref{eq:ref_alpha})}{=} & 0.
\end{eqnarray*}
\end{proof}

\subsubsection{\label{subsec:Numerical-observations-for_Linear}Numerical observations
for the Ruelle-Pollicott resonances and $\gamma_{\mathrm{asympt.}}$
and discussion}

As the linear branches of the IFS can be extended from the intervals
$I_{i}$ to disks in the complex plane, the linear IFS can also be
considered as a holomorphic IFS and its Ruelle-Pollicott spectrum
can be calculated using a dynamical zeta function approach, introduced
by Jenkinson and Pollicott \cite{Jenkinson_Pollicott_02} (see also
\cite{zworski_lin_guillope_02,Borthwick_14,Borthwick_Weich_14,wei14,faure_weich_barkhofen_2014}
for applications and further details). 

Figure \ref{fig:spec_105}(a) shows the Ruelle-Pollicott spectrum
of $\mathcal{L}_{\nu}$ for a given value of $\nu$. Figure \ref{fig:spec_105}(b)
shows the value $\log\left(r_{s}\left(\mathcal{L}_{\nu}\right)\right)=\max_{j}\left(\mathrm{Re}\left(\log\left(\lambda_{j}\right)\right)\right)$
as a function of $\nu$, that we want to bound for $\nu\rightarrow\infty$.
It can be observed, that $\log\left(r_{s}\left(\mathcal{L}_{\nu}\right)\right)$
decays rather quickly starting from $\gamma_{\textup{Gibbs}}$ and
then oscillates in a wide range. Each ``bump'' is due to an individual
eigenvalue. The numerical results indicate that \textbf{the new bound
$\gamma_{\mathrm{up}}$ is not an optimal bound} of $\log\left(r_{s}\left(\mathcal{L}_{\nu}\right)\right)$.
Furthermore the conjecture $\gamma_{conj}=\frac{1}{2}\mathrm{Pr}\left(2\left(V-J\right)\right)$
proposed (\ref{eq:def_gamma_conj}) is not observed to be an upper
bound in this range of $\nu$. However the value of $\log\left(r_{s}\left(\mathcal{L}_{\nu}\right)\right)$
performs ``large fluctuations'' touching the value of $\gamma_{conj}$
several times. A similar phenomenon has been observed for the related
question of the asymptotic spectral gap for the Laplacian on Schottky
surfaces (see \cite[Figure 13]{Borthwick_Weich_14}). The conjecture
$\gamma_{\mathrm{asympt.}}=\gamma_{conj}$ could thus hold if one
suspects, that the ``large fluctuations'' of $\log\left(r_{s}\left(\mathcal{L}_{\nu}\right)\right)$
die out in the semiclassical limit.

\-
\begin{figure}[H]
\centering{}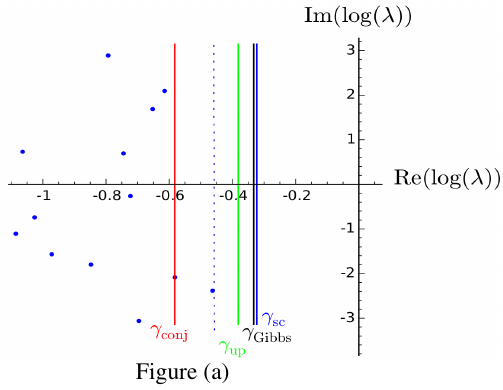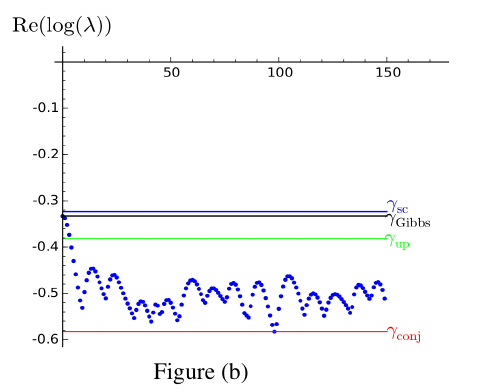\caption{\label{fig:spec_105}Model of linear IFS. \textbf{Figure (a) }shows
the Ruelle-Pollicott eigenvalues $\lambda_{j}\in\mathbb{C}$ (blue
points) of the operator $\mathcal{L}_{\nu}$ for parameters $\nu=105$,
$V=0$, $J_{2}=J_{1}+1$ and $\delta=0.65$. Vertical lines show $\gamma_{conj},\gamma_{\mathrm{up}},\gamma_{\mathrm{Gibbs}},\gamma_{\mathrm{sc}}$
and $\log\left(r_{s}\left(\mathcal{L}_{\nu}\right)\right)=\max_{j}\left(\mathrm{Re}\left(\log\left(\lambda_{j}\right)\right)\right)$
in dotted line. \textbf{Figure (b) }shows $\log\left(r_{s}\left(\mathcal{L}_{\nu}\right)\right)$
with blue points, as a function of $\nu$.}
\end{figure}

\subsection{\label{subsec:Truncated-Gauss-map}Truncated Gauss map}

The model of transfer operators considered here is constructed from
the Gauss map and has simple expressions. The Gauss map is important
in number theory in relation with continued fractions. The Gauss map
is defined by
\begin{equation}
G:\begin{cases}
\left]0,1\right] & \rightarrow\left]0,1\right[\\
y & \rightarrow\frac{1}{y}\mbox{ mod }1
\end{cases}.\label{eq:Gauss_map}
\end{equation}
As this map has an infinite number of branches it does not fit into
the Definition \ref{def:IFSAn-iterated-function} of an IFS. However
if we restrict ourselves to a finite number of branches we get a well
defined IFS. For more details on this construction we refer to \cite[Section 2.1 and 7.1]{faure_arnoldi_tobias_13}.

\begin{center}{\color{red}\fbox{\color{black}\parbox{16cm}{
\begin{defn}
Let $N\geq1$. We consider the finite number of inverse branches of
the Gauss map given for $1\leq j\leq N$ by $(G^{-1})_{j}(x):=1/(x+j)$.
Now for $1\leq i\leq N$ let $a_{i}=1+i$ and $b_{i}$ such that $(G^{-1})_{i}(\frac{1}{N+1})<b_{i}<\frac{1}{i}$.
Then we set the intervals of the truncated Gauss IFS to be $I_{i}=[a_{i},b_{i}]$.
We take the full $N\times N$ matrix as adjacency matrix and define
the maps 
\[
\phi_{i,j}(x):=(G^{-1})_{j}(x):=\frac{1}{x+j},~1\leq i,j\leq N.
\]
See Figure \ref{fig:Gauss_Cutter}.
\end{defn}

}}}\end{center}

The dynamical properties of such truncated Gauss IFS play an important
role in the study of continued fraction expansions (see e.g. \cite{Hen92,MU99}).
In \cite[Prop.7.1]{faure_arnoldi_tobias_13} it has been shown, that
the minimal captivity assumption is fulfilled for roof function $\tau(x)=-J(x)$.
So Theorem \ref{thm:main_result} can be applied.

\begin{figure}[h]
\begin{centering}
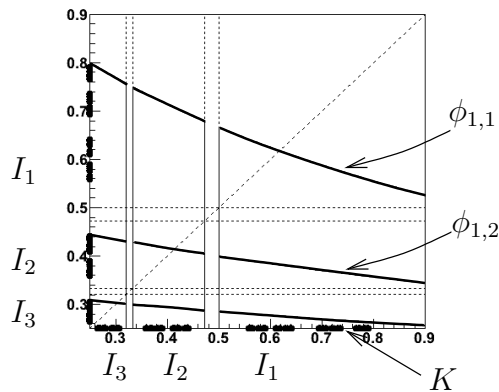
\par\end{centering}
\caption{\label{fig:Gauss_Cutter}The iterated functions system (IFS) defined
from the truncated Gauss map (\ref{eq:Gauss_map}). Here we have $N=3$
branches. The maps $\phi$: $\phi_{i,j}:I_{i}\rightarrow I_{j}$,
$i,j=1\ldots N$ are contracting and given by $\phi_{i,j}\left(x\right)=\frac{1}{x+j}$.
The trapped set $K$ defined in (\ref{eq:trapped_set_K_def1}) is
a $N$-adic Cantor set.}
\end{figure}

\begin{figure}[h]
\begin{centering}
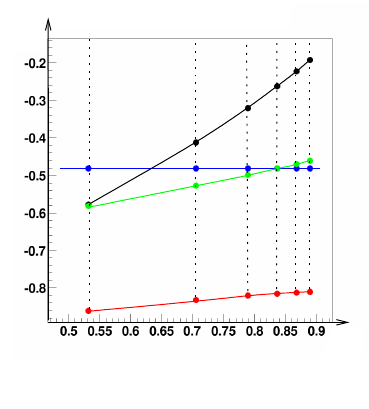$\qquad$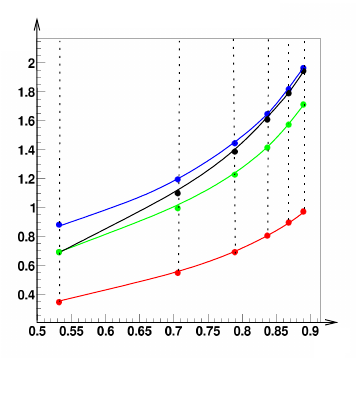
\par\end{centering}
\caption{\label{fig:gamma_Gauss}Plot of various estimates $\gamma_{\mathrm{sc}},\gamma_{\mathrm{Gibbs}},\gamma_{\mathrm{up}},\gamma_{conj}$
defined in Section \ref{subsec:Theorems} for the truncated Gauss
map as a function of $\delta=\dim_{H}\left(K\right)\in]0,1[$ and
for $N=2,3\ldots7$ branches. We put some tiny lines in color between
the dots to help the reading.}
\end{figure}

Figure \ref{fig:gamma_Gauss} shows $\gamma_{\mathrm{sc}},\gamma_{\mathrm{Gibbs}},\gamma_{\mathrm{up}},\gamma_{conj}$
as a function of $\delta$ for $V=0$ and $V=J$. Figure \ref{fig:Gauss_Numerical}
shows numerical results for $\log\left(r_{s}\left(\mathcal{L}_{\nu}\right)\right)$.
We can make the same observations and comments as in Section \ref{subsec:Numerical-observations-for_Linear}.

\begin{figure}[H]
\centering{}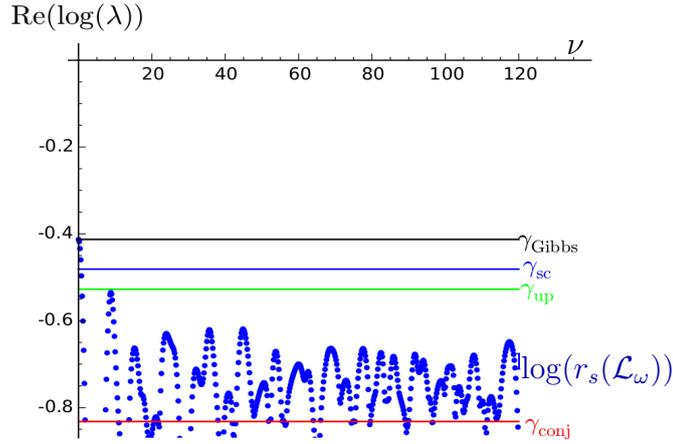\caption{\label{fig:Gauss_Numerical}Numerical values of $\log\left(r_{s}\left(\mathcal{L}_{\nu}\right)\right)$
(blue points) as a function of $\nu$ for the truncated Gauss map
model with $N=3$ branches and $V=0$. }
\end{figure}

\section{\label{sec:A-formula-by}Topological pressure}

\subsection{Definition and basic properties}

We use the notations introduced in Section \ref{subsec:Symbolic-dynamics}.
For a given admissible word $w_{0,n}$ of length $n+1$, let $w\in\mathcal{W}$
be an arbitrary extension of $w_{0,n}$. Let $x_{w}:=S\left(w_{-}\right)\in K$
according to Definition \ref{def:The-symbolic-coding_S}. For a function
$g\in C\left(I;\mathbb{R}\right)$, we define $g_{w_{0,n}}\left(x_{w}\right):=\sum_{k=1}^{n}g\left(\phi_{w_{0,k}}\left(x_{w}\right)\right)$
its Birkhoff sum. Note that $g_{w_{0,n}}\left(x_{w}\right)$ is not
completely determined by $w_{0,n}$ but depends also on its extension,
to a bi-infinite word $w\in\mathcal{W}$. However this dependence
is well controllable for Lipschitz functions:

\noindent\fcolorbox{blue}{white}{\begin{minipage}[t]{1\columnwidth - 2\fboxsep - 2\fboxrule}%
\begin{lem}
\label{lem:bounded_variation}If $g\in C\left(I;\mathbb{R}\right)$
is Lipschitz, then there is a constant $C$ such that for any $n\in\mathbb{N}$
any $w_{0,n}$ and two arbitrary points $x,y\in I_{w_{0}}$ we have
\[
\left|g_{w_{0,n}}\left(x\right)-g_{w_{0,n}}\left(y\right)\right|\leq C.
\]
In particular for two arbitrary extensions $w,w'\in\mathcal{W}$ of
$w_{0,n}$ we have 
\[
\left|g_{w_{0,n}}\left(x_{w}\right)-g_{w_{0,n}}\left(x_{w'}\right)\right|\leq C.
\]
\end{lem}

\end{minipage}}
\begin{proof}
The statement follows directly from a geometric series argument using
the fact that $\phi_{i,j}$ are uniformly contracting and that $g$
is Lipschitz (see e.g. \cite[Proposition 4.1]{Falconer_97} )
\end{proof}
\begin{center}{\color{red}\fbox{\color{black}\parbox{16cm}{
\begin{defn}
The \textbf{topological pressure} of a function $g\in C\left(I;\mathbb{R}\right)$
which is Lipschitz continuous is defined as
\begin{equation}
\mathrm{Pr}\left(g\right):=\lim_{n\rightarrow\infty}\frac{1}{n}\log\left(\sum_{w_{0,n}}e^{g_{w_{0,n}}\left(x_{w}\right)}\right).\label{eq:def_pressure-1}
\end{equation}
Equivalently
\begin{equation}
\sum_{w_{0,n}}e^{g_{w_{0,n}\left(x_{w}\right)}}=e^{n\left(\mathrm{Pr}\left(g\right)+R\left(n\right)\right)},\quad R\left(n\right)\underset{n\rightarrow\infty}{\longrightarrow0}.\label{eq:sum_with_pressure-1}
\end{equation}
\end{defn}

}}}\end{center}
\begin{rem}
\label{rem:Fekete}Lemma \ref{lem:bounded_variation} assures that
$\mathrm{Pr}\left(g\right)$ is independent on the extensions of the
words. The fact, that the limit $n\to\infty$ exists can be seen as
follows: If we set $a_{n}:=\log\left(\sum_{w_{0,n}}e^{g_{w_{0,n}}\left(x_{w}\right)}\right)$,
then using Lemma \ref{lem:bounded_variation} we deduce that there
is a constant $c>0$ such that $a_{k+m}\geq a_{k}+a_{m}-c$. Consequently
$\tilde{a}_{k}=a_{k}-c$ is a superadditive sequence (i.e. $\tilde{a}_{k+m}\geq\tilde{a}_{k}+\tilde{a}_{m}$)
thus the limit $\lim_{n\to\infty}\frac{\tilde{a}_{n}}{n}=\lim_{n\to\infty}\frac{a_{n}}{n}$
exists in $\mathbb{R}\cup\{\infty\}$ from Fekete's Lemma. The fact
that the limit is finite is deduced from the crude bound $\sum_{w_{0,n}}e^{g_{w_{0,n}}\left(x_{w}\right)}\leq N^{n}e^{n\sup_{I}g}$.
\end{rem}

~
\begin{rem}
The expression of $\mathrm{Pr}\left(g\right)$, Eq.(\ref{eq:def_pressure-1})
is similar to the Helmholtz free energy in \href{https://en.wikipedia.org/wiki/Statistical_mechanics}{statistical physics}.
\end{rem}

A particular useful example of a topological pressure is with the
choice of function $g=-\beta J$ where $\beta\in\mathbb{R}$ and $J$
is the unstable Jacobian (\ref{eq:def_J}):

\begin{equation}
P(\beta):=\mathrm{Pr}\left(-\beta J\right)=\lim_{n\to\infty}\frac{1}{n}\log\left(\sum_{w_{0,n}}e^{-\beta J_{w_{0,n}}\left(x_{w}\right)}\right).\label{eq:def_P_beta-1}
\end{equation}
The Bowen formula \cite[p.77]{Falconer_97} gives the \textbf{Hausdorff
dimension} $\delta=\dim_{H}K\in\left[0,1\right]$ of the trapped set
$K$ , (\ref{eq:trapped_set_K_def1}), as the unique solution of
\begin{equation}
P\left(\delta\right)=0.\label{eq:def_delta-1}
\end{equation}
The \textbf{topological entropy }counts the exponential rate of number
of trajectories with respect to time $n$:
\begin{equation}
h_{top}:=P\left(0\right)=\lim_{n\to\infty}\frac{1}{n}\log\left(\sharp\left\{ w_{0,n}\mbox{ admissible}\right\} \right)\label{eq:def_htop-1}
\end{equation}

\subsection{Distribution of time averages of $f$ weighted by $g$}

The theory of large deviations has originally been developed in the
context of stochastic processes and has later been adapted for hyperbolic
dynamical systems (see e.g. \cite{young1990large,kifer1992averaging,kifer1994large}).\emph{
}In this section we will shortly collect a few of these results in
the context of our systems and give self contained proofs for the
sake of completeness.

Let $f,g\in C\left(I,\mathbb{R}\right)$ be two functions. For a given
$n\geq1$, we use the function $g$ to define a probability measure
$p_{g}$ on the set of admissible words (or trajectories) $w_{0,n}$
with a given length $n+1$: 
\begin{equation}
p_{g}\left(w_{0,n}\right):=\frac{1}{Z_{n}\left(g\right)}e^{g_{w_{0,n}}\left(x_{w}\right)}\label{eq:def_p_g_n}
\end{equation}
 where $Z_{n}\left(g\right):=\sum_{w_{0,n}}e^{g_{w_{0,n}}\left(x_{w}\right)}$
is the normalization factor (called ``partition function'' in physics).
We are interested in the distribution of time Birkhoff averages of
the function $f$ for large time $n$, namely the values $\left(\frac{1}{n}f_{w_{0,n}}\left(x_{w}\right)\right)_{w_{0,n}}$
where each value $\frac{1}{n}f_{w_{0,n}}\left(x_{w}\right)$ is weighted
by the probability $p_{g}\left(w_{0,n}\right)$. Let
\[
f_{min}:=\lim_{n\to\infty}\inf_{w_{0,n}}\left(\frac{1}{n}f_{w_{0,n}}\left(x_{w}\right)\right),\quad\quad f_{max}:=\lim_{n\to\infty}\sup_{w_{0,n}}\left(\frac{1}{n}f_{w_{0,n}}\left(x_{w}\right)\right)
\]
be the limit values of the distribution. The average of this distribution
is
\[
\left\langle f\right\rangle _{n,g}:=\sum_{w_{0,n}}p_{g}\left(w_{0,n}\right)\left(\frac{1}{n}f_{w_{0,n}}\left(x_{w}\right)\right),
\]
and its variance is
\begin{eqnarray*}
\mathrm{Var}_{n,g}\left(f\right) & := & \sum_{w_{0,n}}p_{g}\left(w_{0,n}\right)\left(\left(\frac{1}{n}f_{w_{0,n}}\left(x_{w}\right)\right)-\left\langle f\right\rangle _{n,g}\right)^{2}.
\end{eqnarray*}
To express some results concerning this distribution, let us introduce
the function
\begin{equation}
u:\beta\in\mathbb{R}\rightarrow u\left(\beta\right):=\mathrm{Pr}\left(g+\beta f\right)\in\mathbb{R}\label{eq:def_u}
\end{equation}

\begin{center}{\color{blue}\fbox{\color{black}\parbox{16cm}{
\begin{prop}
\label{prop:pressure_derivatives-1}The function $u$ is convex. We
have
\[
\lim_{n\rightarrow\infty}\left\langle f\right\rangle _{n,g}=\left(\frac{du}{d\beta}\right)\left(0\right),\qquad\lim_{n\rightarrow\infty}n\mathrm{Var}_{n,g}\left(f\right)=\left(\frac{d^{2}u}{d\beta^{2}}\right)\left(0\right).
\]
We also have
\[
f_{min}=\lim_{\beta\to-\infty}\left(\frac{du}{d\beta}\right)\left(\beta\right),\qquad f_{max}=\lim_{\beta\to+\infty}\left(\frac{du}{d\beta}\right)\left(\beta\right).
\]
\end{prop}

}}}\end{center}
\begin{proof}
Write $S_{n}\left(\beta\right):=\sum_{w_{0,n}}e^{g_{w_{0,n}}\left(x_{w}\right)+\beta f_{w_{0,n}}\left(x_{w}\right)}$
and $u_{n}\left(\beta\right):=\frac{1}{n}\log S_{n}\left(\beta\right)$.
We have $\left(\frac{du_{n}}{d\beta}\right)\left(0\right)=\frac{1}{n}\frac{S'_{n}\left(0\right)}{S_{n}\left(0\right)}=\left\langle f\right\rangle _{n,g}$
and $\left(\frac{d^{2}u_{n}}{d\beta^{2}}\right)\left(0\right)=\frac{1}{n}\left(\frac{S_{n}^{\prime\prime}\left(0\right)}{S_{n}\left(0\right)}-\left(\frac{S'_{n}\left(0\right)}{S_{n}\left(0\right)}\right)^{2}\right)=n\mathrm{Var}_{n,g}\left(f\right)$.
We deduce that $\left(\frac{d^{2}u}{d\beta^{2}}\right)\left(0\right)\geq0$.
We can replace $g$ by $g+\beta f$ and deduce that $\left(\frac{d^{2}u}{d\beta^{2}}\right)\left(\beta\right)\geq0$.
So $u$ is convex.
\end{proof}
We will now consider ``large deviations'' of the distribution. Note
that the variance of the distribution is of order $1/n$ so being
at a distance $\asymp1$ from the expectation value is already a ``large
deviation''. Thus we consider for an interval $\mathcal{I}\subset]f_{min},f_{max}[$
the quantity

\[
P\left(n,\mathcal{I}\right):=\sum_{w_{0,n}\mbox{ s.t. }\frac{1}{n}f_{w_{0,n}}\in\mathcal{I}}p_{g}\left(w_{0,n}\right)
\]
which represents the probability that $\left(\frac{1}{n}f_{w_{0,n}}\right)\in\mathcal{I}$.
In particular for $f_{min}<t<f_{max}$ let 
\begin{eqnarray}
\Omega\left(t\right): & = & \lim_{\epsilon\rightarrow0}\lim_{n\to\infty}\frac{1}{n}\log\left(P\left(n,[t-\varepsilon,t+\varepsilon[\right)\right)\quad\in\mathbb{R}\cup\left\{ -\infty\right\} .\label{eq:def_Omega}\\
 & \underset{(\ref{eq:def_p_g_n})}{=} & -\mathrm{Pr}\left(g\right)+\lim_{\epsilon\rightarrow0}\lim_{n\to\infty}\frac{1}{n}\log\left(\sum_{w_{0,n}\mbox{ s.t. }\frac{1}{n}f_{w_{0,n}}\in[t-\varepsilon,t+\varepsilon[}e^{g_{w_{0,n}}\left(x_{w}\right)}\right)\nonumber 
\end{eqnarray}
be the exponential rate of the probability as $n\rightarrow\infty$
for a small interval around $t$. In the last expression, the limit
$n\to\infty$ exists from a superadditivity argument analogous to
the argument given in Remark \ref{rem:Fekete} above. The limit $\varepsilon\to0$
exists because one obtains a monotonously decreasing sequence. 

Note that if $f$ is cohomologous to a constant $c$ (i.e. $f=c+\eta-\eta\circ\phi^{-1}$
with some function $\eta$), then there is another constant $C$ such
that for any word $w_{0,n}$ we have $\left|f_{w_{0,n}}-nc\right|\leq C$.
In particular the complete distribution of $f_{w_{0,n}}$ is contained
in the interval $[c-C/n,c+C/n]$, so the question of studying large
deviations becomes trivial in this case. We therefore assume from
now on, that $f$ is not cohomologous to a constant, which implies,
that the pressure function $u(\beta)$ is strictly convex.

\begin{center}{\color{blue}\fbox{\color{black}\parbox{16cm}{
\begin{prop}
\label{prop:LD}\textbf{''Large deviations''}. Let $t\in\mathbb{R}$
and $\beta\left(t\right)$ be defined by
\[
t=\frac{d}{d\beta}\mathrm{Pr}\left(g+\beta f\right)_{/\beta=\beta\left(t\right)}=\frac{du}{d\beta}{}_{/\beta=\beta\left(t\right)},
\]
and 
\[
v\left(t\right):=\beta\left(t\right)\cdot t-\mathrm{Pr}\left(g+\beta\left(t\right)f\right)
\]
be the Legendre transform \cite[p.61]{arnold-mmmc} of the function
$u$, Eq.(\ref{eq:def_u}). Then for $f_{min}<t<f_{max}$ we have
\begin{equation}
\Omega\left(t\right)=-\mathrm{Pr}\left(g\right)-v\left(t\right).\label{eq:Omega}
\end{equation}
\end{prop}

}}}\end{center}
\begin{rem}
We have 
\begin{equation}
\left(\frac{dv}{dt}\right)\left(t\right)=\beta\left(t\right).\label{eq:rem_dvdt}
\end{equation}

The functions $u$ and $v$ are convex. We deduce:
\end{rem}

\begin{center}{\color{blue}\fbox{\color{black}\parbox{16cm}{
\begin{cor}
\label{cor:LD_2}Let $t_{0}\in]f_{min},f_{max}[$ such that $\left(\frac{dv}{dt}\right)\left(t_{0}\right)=\beta\left(t_{0}\right)=0.$
For any interval $\mathcal{I}=\left[t_{a},t_{b}\right]$ with $f_{min}<t_{a}<t_{b}<f_{max}$
we have
\begin{eqnarray*}
\lim_{n\rightarrow\infty}\frac{1}{n}\log\left(P\left(n,\mathcal{I}\right)\right) & = & -\mathrm{Pr}\left(g\right)+\sup_{t\in\mathcal{I}}\left(-v\left(t\right)\right)
\end{eqnarray*}
equivalently
\begin{eqnarray}
\lim_{n\rightarrow\infty}\frac{1}{n}\log\left(\sum_{w_{0,n}\mbox{ s.t. }\frac{1}{n}f_{w_{0,n}}\in\mathcal{I}}e^{g_{w_{0,n}}\left(x_{w}\right)}\right) & = & \sup_{t\in\mathcal{I}}\left(-v\left(t\right)\right)\label{eq:formula_LD}
\end{eqnarray}
with
\begin{eqnarray}
\sup_{t\in\mathcal{I}}\left(-v\left(t\right)\right) & =\begin{cases}
-v\left(t_{0}\right) & \mbox{ if }t_{0}\in\left[t_{a},t_{b}\right]\\
-v\left(t_{b}\right) & \mbox{ if }t_{0}\geq t_{b}\\
-v\left(t_{a}\right) & \mbox{ if }t_{0}\leq t_{a}
\end{cases}\label{eq:SP-1-1}
\end{eqnarray}
\end{cor}

}}}\end{center}
\begin{proof}[Proof of Proposition \ref{prop:LD}]
We are grateful to Mark Pollicott and Richard Sharp for explaining
Proposition \ref{prop:LD} and Corollary \ref{cor:LD_2} to us. Based
on ideas from Kifer \cite{kifer1992averaging,kifer1994large} these
kind of formulas can be derived from the work of Pollicott and Sharp
\cite{PS96,Pol95,Sha92} using the variational approach to the pressure
function. In the sequel we provide a self-contained proof, which fits
into the periodic orbit definition of the topological pressure which
we use in this article.

For any two functions $f,g\in C\left(I,\mathbb{R}\right)$, for any
$t\in\mathbb{R}$ and any $\varepsilon>0$ let us define the following
quantity
\[
K_{g,\varepsilon}(t):=\lim_{n\to\infty}\frac{1}{n}\log\left(\sum_{w_{0,n}\mbox{ s.t. }\frac{1}{n}f_{w_{0,n}}\in[t-\varepsilon,t+\varepsilon[}e^{g_{w_{0,n}}\left(x_{w}\right)}\right)
\]
Recall that we denoted by $N$ the number of letters, so we get the
very crude estimate $-\infty\leq K_{g,\varepsilon}(t)\leq\log N+\cdot\max_{x\in I}g$.
We also deduce from that fact, that $e^{g_{w_{0,n}}}>0$ and the monotonicity
of the logarithm, that for any fixed $t\in\mathbb{R}$ and for $\varepsilon\to0$
the expression $K_{g,\varepsilon}(t)$ is monotonously decreasing.
Thus we can define
\[
K_{g}(t):=\lim_{\varepsilon\to0}K_{g,\varepsilon}(t)\in\mathbb{R}\cup{-\infty}.
\]
Notice that $\Omega\left(t\right)\underset{(\ref{eq:def_Omega})}{=}-\mathrm{Pr}\left(g\right)+K_{g}\left(t\right)$.
In a first step let us show the following Lemma

\noindent\fcolorbox{blue}{white}{\begin{minipage}[t]{1\columnwidth - 2\fboxsep - 2\fboxrule}%
\begin{lem}
The function $t\rightarrow K_{g}(t)$ is an upper semi-continuous
concave function.
\end{lem}

\end{minipage}}
\begin{proof}
The upper semi-continuity follows easily from the definition of $K_{g}$:
for a given $t_{0}$ and $\varepsilon>0$ take $\delta>0$ such that
$0\leq K_{g,\delta}(t_{0})-K_{g}(t_{0})\leq\varepsilon$. Then for
any $t$ such that $\left|t-t_{0}\right|<\delta$ we get, that $K_{g}(t)\leq K_{g,\delta}(t_{0})\leq K_{g}(t_{0})+\varepsilon$.
For every $\varepsilon>0$, $K_{g,\varepsilon}$ is midpoint concave
because for any $t_{1},t_{2}\in[f_{min},f_{max}]$ we have
\begin{eqnarray}
K_{g,\varepsilon}\left(\frac{t_{1}+t_{2}}{2}\right): & = & \lim_{n\to\infty}\frac{1}{2n}\log\left(\sum_{w_{0,2n}\mbox{ s.t. }\frac{1}{2n}f_{w_{0,2n}}\in[\frac{t_{1}+t_{2}}{2}-\varepsilon,\frac{t_{1}+t_{2}}{2}+\varepsilon[}e^{g_{w_{0,2n}}}\right)\nonumber \\
 & \geq & \lim_{n\to\infty}\frac{1}{2n}\log\left(\sum_{w_{0,2n}\mbox{ s.t. }\frac{1}{n}f_{w_{0,n}}\in[t_{1}-\varepsilon,t_{1}+\varepsilon[\textup{ and }\frac{1}{n}f_{w_{n,2n}}\in[t_{2}-\varepsilon,t_{2}+\varepsilon[}e^{g_{w_{0,n}}+g_{w_{n,2n}}}\right)\nonumber \\
 & = & \lim_{n\to\infty}\frac{1}{2n}\log\left(\left(\sum_{w_{0,n}\mbox{ s.t. }\frac{1}{n}f_{w_{0,n}}\in[t_{1}-\varepsilon,t_{1}+\varepsilon[}e^{g_{w_{0,n}}}\right)\right.\label{eq:LD_product}\\
 &  & \qquad\qquad\left.\quad\quad\left(\sum_{w_{0,n}\mbox{ s.t. }\frac{1}{n}f_{w_{0,n}}\in[t_{2}-\varepsilon,t_{2}+\varepsilon[}e^{g_{w_{0,n}}}\right)\right)\nonumber \\
 & = & \frac{1}{2}\left(K_{g,\varepsilon}(t_{1})+K_{g,\varepsilon}(t_{2})\right)\nonumber 
\end{eqnarray}
Taking the limit $\varepsilon\to0$ we deduce that $K_{g}$ is midpoint
concave. As upper semi-continuity implies Lebesgue measurable we deduce
that $K_{g}$ is concave.
\end{proof}
\begin{rem}
Note that in (\ref{eq:LD_product}) we crucially use the transitivity
of the adjacency matrix as assumed in Definition \ref{def:IFSAn-iterated-function}.
Without this assumption the statement that $K_{g}(t)$ is concave
becomes obviously false: Assume for example the case with $N=2$ intervals
and the non-transitive adjacency matrix $A=\left(\begin{array}{ll}
1 & 0\\
0 & 1
\end{array}\right)$. If now $f$ is piece-wise constant, with $f_{/I_{1}}\neq f_{/I_{2}}$then
for any word $w_{0,n}$, either $f_{w_{0,n}}=nf_{/I_{1}}$ or $f_{w_{0,n}}=nf_{/I_{2}}$
and consequently $K_{g}(t)\geq0$ if $t=f_{/I_{1}}$or $t=f_{/I_{2}}$
and $K_{g}(t)=-\infty$ else. 
\end{rem}

We continue the proof of Proposition \ref{prop:LD}. Let us now show
that $K_{g}(t)=-v\left(t\right)$. Recall from (\ref{eq:rem_dvdt})
that $\frac{d}{dt}v(t)=\beta(t)$ hence for any $t_{0}$ we have $\frac{d}{dt}\left(\beta\left(t_{0}\right)t-v\left(t\right)\right)_{|t=t_{0}}=0$
or in other words $\beta\left(t_{0}\right)t-v\left(t\right)$ has
a maximum at $t=t_{0}$ given by
\begin{equation}
\max_{t}\left(\beta\left(t_{0}\right)t-v\left(t\right)\right)=\beta\left(t_{0}\right)t_{0}-v\left(t_{0}\right)=\mathrm{Pr}\left(g+\beta\left(t_{0}\right)f\right)=u\left(\beta\left(t_{0}\right)\right).\label{eq:max_K_tilde}
\end{equation}
Recall the definition (\ref{eq:def_pressure-1}) of the topological
pressure. For $K\in\mathbb{N}$ and $\Delta_{K}:=\frac{f_{max}-f_{min}}{K}$
we write 
\begin{eqnarray*}
\mathrm{Pr}\left(g+\beta f\right) & = & \lim_{n\rightarrow\infty}\frac{1}{n}\log\left(\sum_{k=0}^{K-1}\left(\sum_{w_{0,n\mbox{ s.t. }\frac{1}{n}f_{w_{0,n}}\in[f_{min}+k\Delta_{K},f_{min}+(k+1)\Delta_{K}[}}e^{g_{w_{0,n}}+\beta f_{w_{0,n}}}\right)\right)\\
 & = & \max_{k=0,\ldots K-1}\lim_{n\rightarrow\infty}\frac{1}{n}\log\left(\sum_{w_{0,n}\mbox{ s.t. }\frac{1}{n}f_{w_{0,n}}\in[f_{min}+k\Delta_{K},f_{min}+(k+1)\Delta_{K}[}e^{g_{w_{0,n}}+\beta f_{w_{0,n}}}\right)
\end{eqnarray*}
Recalling the definition of $K_{g,\varepsilon}(t)$ above we get for
any $k=0,\ldots,K-1$
\begin{align*}
\left|\lim_{n\rightarrow\infty}\frac{1}{n}\log\left(\sum_{w_{0,n}\mbox{s.t. }\frac{1}{n}f_{w_{0,n}}\in[f_{min}+k\Delta_{K},f_{min}+(k+1)\Delta_{K}[}e^{g_{w_{0,n}}+\beta f_{w_{0,n}}}\right)-\left(K_{g,\Delta_{K}/2}(f_{k})+\beta f_{k}\right)\right|\\
\leq\beta\Delta_{K}
\end{align*}
where $f_{k}=f_{min}+(k+\frac{1}{2})\Delta_{K}$. Taking the limit
$K\to\infty$, we get 
\[
u\left(\beta\right)=\mathrm{Pr}\left(g+\beta f\right)=\max_{t}\left(\beta t+K_{g}(t)\right)
\]
which is the same expression that we have obtained for $-v\left(t\right)$
in (\ref{eq:max_K_tilde}). As we have shown that $K_{g}$ is upper
semi continuous and concave, the Fenchel-Moreau theorem implies that
$K_{g}=-v$. We have finished the proof of Proposition \ref{prop:LD}.
\end{proof}

\section{\label{sec:Discussion-about-}Discussion about $\gamma_{\mathrm{asympt.}}$
in hyperbolic dynamics}

\subsection{Motivation to study $\gamma_{\mathrm{asympt.}}$ }

Let us consider the case of an Anosov flow $\phi^{t}=e^{tX}$, $t\in\mathbb{R}$
(also called uniformly hyperbolic flow) generated by a \textbf{Anosov
vector field} $X$ on a closed manifold $M$. A typical example is
the geodesic vector field $X$ associated to a Riemannian manifold
$\left(\mathcal{M},g\right)$ with (variable) negative curvature:
$X$ is the Hamiltonian vector field on $M=T_{1}^{*}\mathcal{M}$
(the unit cotangent bundle). This example is special because the flow
preserves the canonical Liouville contact one form $\alpha$ on $M$.
More generally Anosov flows that preserve a contact one form are called
``contact Anosov flows''. We introduce an arbitrary smooth function
$V\in C^{\infty}\left(M;\mathbb{R}\right)$ called the \textbf{potential
function} and consider the operator
\[
A:=-X+V.
\]
$A$ has intrinsic discrete spectrum (of finite multiplicity) in certain
anisotropic Sobolev spaces $\mathcal{H}\left(M\right)$ \cite{liverani_butterley_07,fred_flow_09,faure_tsujii_Ruelle_resonances_density_2016}
and the set of eigenvalues $\left(z_{j}\right)_{j}\subset\mathbb{C}$
of $A$ are called the\textbf{ Ruelle-Pollicott resonances of $A$
for positive time $t\geq0$. }The operator $A$ is the generator of
$\mathcal{L}^{t}:=e^{tA},t\geq0,$ called the transfer operator giving
transport of functions $u\in C^{\infty}\left(M\right)$ by 
\begin{equation}
\mathcal{L}^{t}u=e^{tA}u=e^{V_{\left[-t,0\right]}}\cdot\left(u\circ\phi_{-t}\right)\label{eq:transfer_op}
\end{equation}
with $V_{\left[-t,0\right]}:=\int_{-t}^{0}V\circ\phi^{s}ds$.

We define

\begin{equation}
\gamma_{\mathrm{asympt.}}:=\limsup_{\nu\rightarrow\infty}\sup_{j}\left\{ \mathrm{Re}\left(z_{j}\right),\mbox{ s.t.}\left|\mathrm{Im}\left(z_{j}\right)\right|\geq\nu\right\} \label{eq:def_gamma_asympt}
\end{equation}
i.e. $\gamma_{\mathrm{asympt.}}$ is such that for any $\epsilon>0$
there are only finitely many Ruelle-Pollicott resonances on the right
of the line $\mathrm{Re}\left(z\right)=\gamma_{\mathrm{asympt.}}+\epsilon$.
To express the importance of the quantity $\gamma_{\mathrm{asympt.}}$,
we will assume the following two properties about the spectrum of
$A$. We will see many examples in Section \ref{subsec:Known-results-about}
where these assumptions are satisfied.

\noindent\fcolorbox{red}{white}{\begin{minipage}[t]{1\columnwidth - 2\fboxsep - 2\fboxrule}%
\begin{assumption}
\label{assu:We-will-assume}We will assume
\begin{enumerate}
\item The Ruelle-Pollicott spectrum of the operator $A$ has a single and
simple dominant real eigenvalue $\gamma_{\mathrm{Gibbs}}$\textbf{\footnotemark}.
\item ``Uniform control of the norm of the resolvent'': there exists $\epsilon>0$,
$\nu>0$, $C>0$ such that $\gamma_{\mathrm{asympt.}}+\epsilon<\gamma_{\mathrm{Gibbs}}$
and
\begin{equation}
\forall z\in\mathbb{C}\text{ s.t. }\mathrm{Re}\left(z\right)\geq\gamma_{\mathrm{asympt.}}+\epsilon,\left|\mathrm{Im}\left(z\right)\right|\geq\nu,\quad\left\Vert \left(z-A\right)^{-1}\right\Vert _{\mathcal{H}\left(M\right)}\leq C.\label{eq:hyp_resolvent}
\end{equation}
\end{enumerate}
\end{assumption}

\end{minipage}}

\footnotetext{i.e. other eigenvalues $z_{j}\in \mathbb{C}$ satisfy $\mathrm{Re} \left(z_{j} \right)< \gamma_{\mathrm{Gibbs}}$.}

Equivalently to (\ref{eq:hyp_resolvent}) one has \cite{engel_1999}

\begin{equation}
\exists C,\forall t\geq0,\quad\left\Vert \mathcal{L}^{t}-\sum_{j\mbox{ s.t }\mathrm{Re}\left(z_{j}\right)\geq\gamma_{\mathrm{asympt.}}+\epsilon}\mathcal{L}^{t}\Pi_{j}\right\Vert _{\mathcal{H}\left(M\right)}\leq Ce^{\left(\gamma_{\mathrm{asympt.}}+\epsilon\right)t}\label{eq:decay}
\end{equation}
where $\Pi_{j}$ is the spectral projector of finite rank associated
to $z_{j}$. If the eigenvalue $z_{j}$ is simple then $\mathcal{L}^{t}\Pi_{j}=e^{z_{j}t}\Pi_{j}$.
The sum in (\ref{eq:decay}) is finite. In particular

\begin{equation}
\exists\epsilon>0,\exists C,\forall t\geq0,\quad\left\Vert \mathcal{L}^{t}-e^{\gamma_{\mathrm{Gibbs}}t}\Pi_{\mathrm{Gibbs}}\right\Vert _{\mathcal{H}\left(M\right)}\leq Ce^{\left(\gamma_{\mathrm{Gibbs}}-\epsilon\right)t}\label{eq:decay-1}
\end{equation}

From the construction of $\mathcal{H}\left(M\right)$, one has $C^{\infty}\left(M\right)\subset\mathcal{H}\left(M\right)\subset\mathcal{D}'\left(M\right)$.
We define the dual space $\mathcal{H}'\left(M\right)$ by 
\[
\mathcal{H}'\left(M\right):=\left\{ u\in\mathcal{D}'\left(M\right),\text{ s.t. }v\in\mathcal{H}\left(M\right)\rightarrow\left\langle u,v\right\rangle _{L^{2}\left(M;dx\right)}\in\mathbb{C}\text{ is bounded}\right\} ,
\]
where $dx$ is an arbitrary smooth volume on $M$ (in case of contact
Anosov flow $dx$ is inherited from the contact structure). Eq.(\ref{eq:decay})
implies some expansions of time correlation of functions (as in \cite[Corollary 1.2]{tsujii_08}\cite[Corollary 5]{nonenmacher_zworski_2013}):

\begin{equation}
\forall u\in\mathcal{H}'\left(M\right),v\in\mathcal{H}\left(M\right),\quad\langle u,\mathcal{L}^{t}v\rangle_{L^{2}\left(M,dx\right)}=\sum_{j\,\mathrm{s.t.}\mathrm{Re}\left(z_{j}\right)\geq\gamma_{\mathrm{asympt.}}+\epsilon}\langle u,\mathcal{L}^{t}\Pi_{j}v\rangle_{L^{2}}+O\left(e^{\left(\gamma_{\mathrm{asympt.}}+\epsilon\right)t}\right).\label{eq:expansion_L^t}
\end{equation}

\subsubsection{Gibbs measure}
\begin{rem}
Eq.(\ref{eq:expansion_L^t}) shows that Ruelle-Pollicott resonances
describe the correlation functions w.r.t. Lebesgue measure $dx$ for
the dynamics weighted with the potential $V$. We will see later in
Corollary \ref{cor:Decay-of-correlations} that the same spectrum
describes the correlation functions w.r.t. a Gibbs measure defined
from $V$ but for the pure flow dynamics, i.e. without potential $V$.
\end{rem}

The Atiyah-Bott flat trace formula \cite{atiyah_67}\cite{guillemin_1977_lectures}
gives that

\[
\gamma_{\mathrm{Gibbs}}=\mathrm{Pr}\left(V-J\right)\in\mathbb{R}
\]
where\footnote{The choice $A=X+V$ would have give instead $\gamma_{\mathrm{Gibbs}}=\mathrm{Pr}\left(V-\mathrm{div}X/_{E_{u}}\right)$.}
\[
J:=-\mathrm{div}X/_{E_{s}}>0.
\]
$\mathrm{div}X/_{E_{s}}<0$ is the expansion rate along the stable
direction $E_{u}$. We have $\mathrm{div}X/_{E_{s}}+\mathrm{div}X/_{E_{u}}=\mathrm{div}X$
hence for a volume preserving flow, $\mathrm{div}X=0$, one has $J:=-\mathrm{div}X/_{E_{s}}=\mathrm{div}X/_{E_{u}}$.
$\mathrm{Pr}\left(\varphi\right)$ is the topological pressure of
a function $\varphi\in C\left(M\right)$ and is defined for flows
using a sum over periodic orbits $\gamma$ as follows:
\begin{equation}
\mathrm{Pr}\left(\varphi\right):=\lim_{t\rightarrow\infty}\frac{1}{t}\log\left(\sum_{p.o.\gamma\mbox{ s.t. }\left|\gamma\right|\in\left[t,t+1\right]}e^{\int_{\gamma}\varphi}\right).\label{eq:def_Pressure}
\end{equation}
We denote $\Pi_{\mathrm{Gibbs}}$ the rank one spectral projector
associated to the eigenvalue $\gamma_{\mathrm{Gibbs}}$. It defines
the so called ``\textbf{Gibbs equilibrium measure}''\footnote{$\mu_{\mathrm{Gibbs}}$ is a \textbf{positive measure} (i.e. distribution
of order $0$) because of the following argument. Let $\varphi\in C^{\infty}\left(M;\mathbb{R}^{+}\right)$.
Let us denote $\delta_{x}$ the Dirac measure at $x\in M$. The Atiyah-Bott
flat trace of an operator $A$ is $\mathrm{Tr}^{\flat}\left(A\right):=\int_{M}\langle\delta_{x},A\delta_{x}\rangle dx$,
see \cite{guillemin_1977_lectures}. The Schwartz kernel of the operator
$\mathcal{L}^{t}$ is positive hence for any $t\geq0$,
\[
\mathrm{Tr}^{\flat}\left(\mathcal{M}_{\varphi}e^{-\gamma_{\mathrm{Gibbs}}t}\mathcal{L}^{t}\right):=\int_{M}\varphi\left(x\right)\langle\delta_{x},e^{-\gamma_{\mathrm{Gibbs}}t}\mathcal{L}^{t}\delta_{x}\rangle dx\leq\left|\varphi\right|_{C^{0}}\int_{M}\langle\delta_{x},e^{-\gamma_{\mathrm{Gibbs}}t}\mathcal{L}^{t}\delta_{x}\rangle dx=\left|\varphi\right|_{C^{0}}\mathrm{Tr}^{\flat}\left(e^{-\gamma_{\mathrm{Gibbs}}t}\mathcal{L}^{t}\right).
\]
We make $t\rightarrow+\infty$. Using (\ref{eq:decay-1}) and additional
arguments that can be found in \cite[Appendix B]{faure-tsujii_anosov_flows_13}
one obtains
\[
\mu_{\mathrm{Gibbs}}\left(\varphi\right)\leq\left|\varphi\right|_{C^{0}}\mathrm{Tr}\left(\Pi_{\mathrm{Gibbs}}\right)=\left|\varphi\right|_{C^{0}}.
\]
} associated to the potential $V$ by 
\[
\mu_{\mathrm{Gibbs}}:\varphi\in C^{\infty}\left(M\right)\rightarrow\mathrm{Tr}\left(\mathcal{M}_{\varphi}\Pi_{\mathrm{Gibbs}}\right)\in\mathbb{R}
\]
where $\mathcal{M}_{\varphi}u=\varphi u$ denotes the multiplication
operator by a function $\varphi\in C^{\infty}\left(M\right)$. $\mathcal{M}_{\varphi}:\mathcal{H}\left(M\right)\rightarrow\mathcal{H}\left(M\right)$
is a bounded operator. The operator $\mathcal{M}_{\varphi}\Pi_{\mathrm{Gibbs}}$
is finite rank hence trace class in $\mathcal{H}\left(M\right)$.
The Gibbs measure $\mu_{\mathrm{Gibbs}}$ has the following properties:

\noindent\fcolorbox{blue}{white}{\begin{minipage}[t]{1\columnwidth - 2\fboxsep - 2\fboxrule}%
\begin{lem}
\label{lem:Invariance-of-Gibbs}``\textbf{Invariance of Gibbs measure
under the flow}''. We have 
\[
\forall\varphi\in C^{\infty}\left(M\right),\forall t\geq0,\quad\mu_{\mathrm{Gibbs}}\left(\varphi\circ\phi^{-t}\right)=\mu_{\mathrm{Gibbs}}\left(\varphi\right).
\]
\end{lem}

\end{minipage}}
\begin{proof}
Let us write $\mathcal{L}_{0}^{t}=e^{-tX}$ and $\mathcal{M}_{\varphi}u=\varphi u$
the multiplication operator by $\varphi$. We have the relations $\mathcal{L}^{t}\Pi_{\mathrm{Gibbs}}=\Pi_{\mathrm{Gibbs}}\mathcal{L}^{t}=e^{t\gamma_{\mathrm{Gibbs}}}\Pi_{\mathrm{Gibbs}}$,
$\mathcal{L}^{t}=\mathcal{M}_{e^{V_{\left[-t,0\right]}}}\mathcal{L}_{0}^{t}$,
$\mathcal{M}_{u}\mathcal{M}_{v}=\mathcal{M}_{v}\mathcal{M}_{u}$,
$\mathcal{L}^{t}\mathcal{M}_{u}=\mathcal{M}_{\mathcal{L}_{0}^{t}u}\mathcal{L}^{t}$
and circularity of $\mathrm{Tr}\left(.\right)$ and deduce
\begin{align*}
\mu_{\mathrm{Gibbs}}\left(\varphi\circ\phi^{-t}\right) & =\mathrm{Tr}\left(\mathcal{M}_{\mathcal{L}_{0}^{t}\varphi}\Pi_{\mathrm{Gibbs}}\right)=e^{-t\gamma_{\mathrm{Gibbs}}}\mathrm{Tr}\left(\mathcal{M}_{\mathcal{L}_{0}^{t}\varphi}\mathcal{L}^{t}\Pi_{\mathrm{Gibbs}}\right)\\
 & =e^{-t\gamma_{\mathrm{Gibbs}}}\mathrm{Tr}\left(\mathcal{L}^{t}\mathcal{M}_{\varphi}\Pi_{\mathrm{Gibbs}}\right)=e^{-t\gamma_{\mathrm{Gibbs}}}\mathrm{Tr}\left(\mathcal{M}_{\varphi}\Pi_{\mathrm{Gibbs}}\mathcal{L}^{t}\right)=\mathrm{Tr}\left(\mathcal{M}_{\varphi}\Pi_{\mathrm{Gibbs}}\right)\\
 & =\mu_{\mathrm{Gibbs}}\left(\varphi\right).
\end{align*}
\end{proof}
The expansion (\ref{eq:expansion_L^t}) implies some expansion for
correlation functions expressed with the Gibbs measure (that is more
usual in dynamical systems theory) as follows.

\noindent\fcolorbox{blue}{white}{\begin{minipage}[t]{1\columnwidth - 2\fboxsep - 2\fboxrule}%
\begin{cor}
\label{cor:Decay-of-correlations}``\textbf{Decay of correlations
for the Gibbs measure}''. Under assumption \ref{assu:We-will-assume},
we have 
\begin{align*}
\forall\varphi_{1},\varphi_{2}\in C^{\infty}\left(M\right),\forall t\geq0,\quad\mu_{\mathrm{Gibbs}}\left(\left(\varphi_{1}\circ\phi^{-t}\right).\varphi_{2}\right)= & \mu_{\mathrm{Gibbs}}\left(\varphi_{1}\right)\mu_{\mathrm{Gibbs}}\left(\varphi_{2}\right)\\
+\sum_{j\mbox{ s.t }\mathrm{Re}\left(z_{j}\right)>\gamma_{\mathrm{asympt.}}+\epsilon,\quad z_{j}\neq\gamma_{\mathrm{Gibbs}}} & e^{-t\gamma_{\mathrm{Gibbs}}}\mathrm{Tr}\left(\mathcal{M}_{\varphi_{2}}\left(\mathcal{L}^{t}\Pi_{j}\right)\mathcal{M}_{\varphi_{1}}\Pi_{\mathrm{Gibbs}}\right)\\
+O\left(e^{-\left(\gamma_{\mathrm{Gibbs}}-\left(\gamma_{\mathrm{asympt.}}+\epsilon\right)\right)t}\right)
\end{align*}
\end{cor}

\end{minipage}}
\begin{rem}
~
\begin{enumerate}
\item The quantity $\left(\gamma_{\mathrm{Gibbs}}-\left(\gamma_{\mathrm{asympt.}}+\epsilon\right)\right)$
that governs the exponential decay of the remainder is called the
``\textbf{asymptotic spectral gap}''.
\item We have assumed that $\gamma_{\mathrm{Gibbs}}$ is dominant eigenvalue
i.e. $\gamma_{\mathrm{Gibbs}}>\mathrm{Re}\left(z_{j}\right)$ for
$z_{j}\neq\gamma_{\mathrm{Gibbs}}$. Then the second line decays and
one gets \textbf{exponential mixing} property for the Gibbs measure:
\[
\mu_{\mathrm{Gibbs}}\left(\left(\varphi_{1}\circ\phi^{-t}\right).\varphi_{2}\right)=\mu_{\mathrm{Gibbs}}\left(\varphi_{1}\right)\mu_{\mathrm{Gibbs}}\left(\varphi_{2}\right)+O\left(e^{-\left(\gamma_{\mathrm{Gibbs}}-\max_{j}\mathrm{Re}\left(z_{j}\right)\right)t}\right).
\]
\item If eigenvalues $z_{j}$ are simple then $\mathcal{L}^{t}\Pi_{j}=e^{z_{j}t}\Pi_{j}$
and each term on the second line writes as
\[
e^{-t\gamma_{\mathrm{Gibbs}}}\mathrm{Tr}\left(\mathcal{M}_{\varphi_{2}}\left(\mathcal{L}^{t}\Pi_{j}\right)\mathcal{M}_{\varphi_{1}}\Pi_{\mathrm{Gibbs}}\right)=e^{-t\left(\gamma_{\mathrm{Gibbs}}-z_{j}\right)}\mathrm{Tr}\left(\mathcal{M}_{\varphi_{2}}\Pi_{j}\mathcal{M}_{\varphi_{1}}\Pi_{\mathrm{Gibbs}}\right).
\]
\end{enumerate}
\end{rem}

\begin{proof}
Recall the relations at the beginning of proof of Lemma \ref{lem:Invariance-of-Gibbs}.
We have
\begin{align*}
\mu_{\mathrm{Gibbs}}\left(\left(\varphi_{1}\circ\phi^{-t}\right).\varphi_{2}\right) & =\mathrm{Tr}\left(\mathcal{M}_{\mathcal{L}_{0}^{t}\varphi_{1}}\mathcal{M}_{\varphi_{2}}\Pi_{\mathrm{Gibbs}}\right)=e^{-t\gamma_{\mathrm{Gibbs}}}\mathrm{Tr}\left(\mathcal{M}_{\mathcal{L}_{0}^{t}\varphi_{1}}\mathcal{M}_{\varphi_{2}}\mathcal{L}^{t}\Pi_{\mathrm{Gibbs}}\right)\\
 & =e^{-t\gamma_{\mathrm{Gibbs}}}\mathrm{Tr}\left(\mathcal{M}_{\varphi_{2}}\mathcal{L}^{t}\mathcal{M}_{\varphi_{1}}\Pi_{\mathrm{Gibbs}}\right)\\
 & =e^{-t\gamma_{\mathrm{Gibbs}}}\sum_{j\mbox{ s.t }\mathrm{Re}\left(z_{j}\right)>\gamma_{\mathrm{asympt.}}+\epsilon}\mathrm{Tr}\left(\mathcal{M}_{\varphi_{2}}\left(\mathcal{L}^{t}\Pi_{j}\right)\mathcal{M}_{\varphi_{1}}\Pi_{\mathrm{Gibbs}}\right)\\
 & \qquad+O\left(e^{-\left(\gamma_{\mathrm{Gibbs}}-\left(\gamma_{\mathrm{asympt.}}+\epsilon\right)\right)t}\right).
\end{align*}
From the fact that $\Pi_{\mathrm{Gibbs}}$ is a rank one projector
we deduce that the first term of the sum is 
\[
e^{-t\gamma_{\mathrm{Gibbs}}}\mathrm{Tr}\left(\mathcal{M}_{\varphi_{2}}\left(\mathcal{L}^{t}\Pi_{\mathrm{Gibbs}}\right)\mathcal{M}_{\varphi_{1}}\Pi_{\mathrm{Gibbs}}\right)=\mathrm{Tr}\left(\mathcal{M}_{\varphi_{2}}\Pi_{\mathrm{Gibbs}}\mathcal{M}_{\varphi_{1}}\Pi_{\mathrm{Gibbs}}\right)=\mu_{\mathrm{Gibbs}}\left(\varphi_{1}\right)\mu_{\mathrm{Gibbs}}\left(\varphi_{2}\right).
\]
\end{proof}

\subsubsection{Special choice $V=J=-\mathrm{div}X/_{E_{s}}$}

In particular, by choosing\footnote{In general $J$ is only Hölder continuous so it requires some special
arguments, namely considering the extension of the transfer operator
on a Grassmanian bundle \cite{faure-tsujii_prequantum_maps_12,faure-tsujii_anosov_flows_13}.} the potential $V=J$ we get the \textbf{topological entropy} $h_{top}$:
\begin{equation}
\gamma_{\mathrm{Gibbs}}=\mathrm{Pr}\left(V-J\right)=\mathrm{Pr}\left(0\right)=\lim_{t\rightarrow\infty}\frac{1}{t}\log\mathcal{N}\left(t\right)=:h_{top}\label{eq:h_top}
\end{equation}
where $\mathcal{N}\left(t\right):=\sharp\left\{ \gamma\mbox{ periodic orbit s.t. }\left|\gamma\right|\leq t\right\} $
counts the periodic orbits of the flow of period less than $t$. In
this case the Gibbs measure is called ``\textbf{the Bowen Margulis
measure of maximal entropy}'' $\mu_{\mathrm{Gibbs}}=:\mu_{B.M.}$.

\subsubsection{Special choice $V=0$}

Let $\boldsymbol{1}\left(x\right)=1$ be the constant function $1$
on $M$. One has $X\left(\boldsymbol{1}\right)=0$. If we choose potential
$V=0$ then $A=-X$, $A\left(\boldsymbol{1}\right)=0$, $\gamma_{\mathrm{Gibbs}}=0$
and
\[
\Pi_{\mathrm{Gibbs}}=\boldsymbol{1}\langle\mu|.\rangle_{L^{2}},
\]
with $\mu\in\mathcal{H}'\left(M\right)$. The Gibbs measure is $\mu_{\mathrm{Gibbs}}=\mu dx=:\mu_{S.R.B.}$
and is called ``\textbf{the Sinai-Ruelle-Bowen measure}''. If the
flow is volume preserving, i.e. $\mathrm{div}_{dx}X=0$ then $\mu_{S.R.B.}=dx$.

\subsection{\label{subsec:Known-results-about}Known results about $\gamma_{\mathrm{asympt.}}$}

\subsubsection{Contact Anosov flows}

For contact Anosov flows, it has been shown by C. Liverani \cite{liverani_contact_04}
that in the case $V=0$,
\begin{equation}
\exists\epsilon>0,\quad\gamma_{\mathrm{asympt.}}<\gamma_{\mathrm{Gibbs}}-\epsilon.\label{eq:result_Liverani}
\end{equation}
M. Tsujii \cite{tsujii_FBI_10} (and \cite{nonenmacher_zworski_2013}
for a generalization to other semiclassical operators), has shown
an explicit upper bound for $\gamma_{\mathrm{asympt.}}$ (his method
works for any smooth potential $V$):
\begin{equation}
\gamma_{\mathrm{asympt.}}\leq\gamma_{\mathrm{sc}}:=\mathrm{tsup}\left(D\right):=\lim_{t\rightarrow\infty}\sup_{x\in M}\left(\frac{1}{t}\int_{0}^{t}D\circ\phi^{-s}\left(x\right)ds\right)\label{eq:gamm_asympt}
\end{equation}
with the so called damping function $D:=V-\frac{1}{2}J$ and where
the linear functional $\mathrm{tsup}()$ called ``time-averaged-sup''
is defined from the last expression. The proof uses semiclassical
analysis with $\nu:=\left|\mathrm{Im}\left(z\right)\right|\rightarrow\infty$
being the frequency in the neutral direction. For this we consider
the flow $\phi^{t}$ lifted on the cotangent bundle $\tilde{\phi}_{t}:T^{*}M\rightarrow T^{*}M$.
Since $\phi^{t}$ preserves the contact one form $\alpha$, the trapped
set $\mathcal{K}$ (i.e. non wandering set) for the lifted flow $\tilde{\phi}_{t}$
is the line bundle $\mathcal{K}=\mathbb{R}\alpha\subset T^{*}M$.
A crucial property is that $\mathcal{K}\backslash\left\{ 0\right\} $
is a smooth symplectic submanifold of $T^{*}M$ and transversally
the dynamics of $\tilde{\phi}_{t}$ is hyperbolic. From this and using
semiclassical techniques, one deduces (\ref{eq:gamm_asympt}) and
also a band structure of the spectrum \cite{faure_tsujii_band_CRAS_2013}.

In particular for the special choice $V=\frac{1}{2}J=\frac{1}{2}\mathrm{div}X/_{E_{u}}$
called ``\textbf{semi-classical potential}'' it is shown in \cite{faure-tsujii_anosov_flows_13}
that
\[
\gamma{}_{asympt}=\gamma_{\mathrm{sc}}=0.
\]
i.e. there is an accumulation of Ruelle resonances on the imaginary
axis. In that case $\gamma_{\mathrm{Gibbs}}>0$.

\subsubsection{Anosov flows in dimension 3}

M. Tsujii has shown in \cite{tsujii_2016_exponential_mixing} that
for generic volume preserving Anosov flow in dimension 3 and $V=0$,
there exists $\gamma_{Tsujii}<0$ such that
\[
\gamma_{\mathrm{asympt.}}\leq\gamma_{Tsujii}<\gamma_{\mathrm{Gibbs}}=0
\]
and one has a uniform control of the resolvent on $\mathrm{Re}\left(z\right)\geq\gamma_{Tsujii}$
that gives decay of correlations.

M. Tsujii considers in \cite{15_tsujii} the case $V=J$ for an expanding
semi-flow and gives a bound for $\gamma_{\mathrm{asympt.}}$ that
improves previous known results.

\subsubsection{Open hyperbolic dynamics}

To the authors best knowledge an analog of (\ref{eq:gamm_asympt})
for open hyperbolic flows (i.e. Axiom A flow) is not known. However
in \cite{faure_arnoldi_tobias_13} the authors proved an analog of
(\ref{eq:gamm_asympt}) together with resolvent estimates for the
$\mathbb{R}$-extensions of IFS which can be considered as a toy model
of an Axiom A flow. 

The present article concerns open dynamics. Its purpose is to improve
the established bounds $\gamma_{\mathrm{sc}}$ and $\gamma_{\mathrm{Gibbs}}$.
We directly work on a model for open dynamical systems such that there
is hope, that our methods and results can be useful for the study
of Ruelle-Pollicott resonances of open hyperbolic flows, such as Axiom
A flows.

\subsubsection{\label{subsec:Quantum-hyperbolic-dynamics}Quantum hyperbolic dynamics}

In quantum mechanics similar questions concerning the asymptotic spectral
gap of an operator arises as follows. On a negative curvature smoothed
closed manifold $\left(\mathcal{M},g\right)$, consider the operator\footnote{If we put $\Phi=\left(\psi,\varphi\right)\in L^{2}\left(\mathcal{M}\right)\oplus L^{2}\left(\mathcal{M}\right)$
then the Schrodinger equation $i\partial_{t}\Phi=P\Phi$ is equivalent
to the ``damped wave equation'' $\partial_{t}^{2}\psi=\Delta\psi-2D\partial_{t}\psi$
with $\varphi=i\partial_{t}\psi$.}
\[
P:=\left(\begin{array}{cc}
0 & \mathrm{Id}\\
-\Delta & 2iD
\end{array}\right)
\]
on $H^{1}\left(\mathcal{M}\right)\oplus L^{2}\left(\mathcal{M}\right)$
where $\Delta$ is the Laplace Beltrami operator and $D\in C^{\infty}\left(\mathcal{M};\mathbb{R}\right)$
is a smooth function \cite{sjostrand_2000}. $P$ has discrete spectrum
$\left(z_{j}\right)_{j}\subset\mathbb{C}$ that belongs to the band
$\mathrm{Im}\left(z_{j}\right)\in\left[\inf D,\sup D\right]$ (for
$\mathrm{Re}z_{j}>0$). One defines
\begin{equation}
\gamma_{\mathrm{asympt.}}:=\limsup_{\nu\rightarrow+\infty}\sup_{j}\left\{ \mathrm{Im}\left(z_{j}\right),\mbox{ s.t.}\mathrm{Re}\left(z_{j}\right)\geq\nu\right\} .\label{eq:def_gamma_asympt_quantum}
\end{equation}
G. Lebeau \cite{lebeau1996equation} has shown that 
\begin{equation}
\gamma_{\mathrm{asympt.}}\leq\gamma_{\mathrm{sc}}:=\mathrm{tsup}\left(D\right):=\lim_{t\rightarrow\infty}\sup_{\left(x,\xi\right)\in T_{1}^{*}\mathcal{M}}\left(\frac{1}{t}\int_{0}^{t}D\circ\phi^{-s}\left(x,\xi\right)ds\right)\label{eq:gamma_sc_quantum}
\end{equation}
where $\phi^{s}:T_{1}^{*}\mathcal{M}\rightarrow T_{1}^{*}\mathcal{M}$
is the geodesic flow and $D$ is trivially extended to $M=T_{1}^{*}\mathcal{M}$
by $D\left(x,\xi\right):=D\left(x\right)$. The bound (\ref{eq:gamma_sc_quantum})
is similar to the bound (\ref{eq:gamm_asympt}).

For ``open quantum dynamics'' Dyatlov and Zahl have recently established
\cite{dyatlov_zahl_15} a new bound for the asymptotic spectral gap
for resonances of the Laplacian on convex co-compact manifolds of
constant negative curvature. Although their model is different, it
would be interesting to compare their results methods and concepts
with ours.

In particular it has been shown recently that there is an exact relation
between Ruelle Pollicott resonances and quantum resonances on convex
co-compact hyperbolic surfaces \cite{guillarmou_weich_resonances_16}.
For these models it has also been shown that $\gamma_{\mathrm{asympt}}<\gamma_{\mathrm{sc}}$
in \cite{dyatlov_bourgain_2016}.

\subsection{\label{subsec:Conjecture-for}Conjecture for $\gamma_{\mathrm{asympt.}}$}

In this Section we discuss a conjecture for the asymptotic spectral
gap $\gamma_{\mathrm{asympt.}}$.

This conjecture is motivated from the expression (\ref{eq:bound})
that appears in the sketch of proof of Theorem \ref{thm:main_result}
and that leads us to our result $\gamma_{\mathrm{asympt.}}\leq\gamma_{\mathrm{up}}$.

In (\ref{eq:bound}) we have a sum of complex numbers over pairs of
orbits $w,w'$. This sum has the form $\sum_{w,w'}e^{\left(V-J\right)_{w}+\left(V-J\right)_{w'}+i\nu\left(\tau_{w}-\tau_{w'}\right)}T_{w',w}$.
In this double sum, we are not able to control the phases $e^{i\nu\left(\tau_{w}-\tau_{w'}\right)}$
of non diagonal terms so we have considered a time $n\sim2\frac{\log\nu}{\left\langle J\right\rangle }$
for which these non diagonal terms vanish (because $T_{w',w}\sim0$).
Then the last term in (\ref{eq:final_step}) gives the remainder $\left\langle J\right\rangle /4$
in our result (\ref{eq:def_gamma_up}).

If we bound all the phases by $1$ and consider the limit $n\rightarrow\infty$
one obtains the bound $\gamma_{\mathrm{asympt.}}\leq\gamma_{\mathrm{Gibbs}}$.
However if one were able to show that phases behave as ``random phases''
(this could hold generically), then the non diagonal terms in (\ref{eq:bound})
become negligible. Consequently we can make the diagonal approximation
for arbitrary long time and if we take time $n=A\frac{\log\left(\nu\right)}{\left\langle J\right\rangle }$
with $A\gg1$ arbitrary large then the last term in (\ref{eq:final_step})
becomes $\frac{1}{2A}\left\langle J\right\rangle \ll1$ and is negligible.
One obtains the conjecture that:

\noindent\fcolorbox{red}{white}{\begin{minipage}[t]{1\columnwidth - 2\fboxsep - 2\fboxrule}%
\begin{conjecture}
For a generic system,
\begin{equation}
\gamma_{\mathrm{asympt.}}=\gamma_{conj}:=\frac{1}{2}\mathrm{Pr}\left(2\left(V-J\right)\right).\label{eq:def_gamma_conj}
\end{equation}
\end{conjecture}

\end{minipage}}

This conjecture can been found in \cite[p.9]{dolgopyat_addendum_98}.
It makes sense for a general hyperbolic dynamics (Anosov flow, Axiom
A flow, ...), even for \emph{quantum} systems as those discussed in
Section \ref{subsec:Quantum-hyperbolic-dynamics} for which the conjecture
is\footnote{\label{fn:as-explained-in}as explained in \cite{faure-tsujii_prequantum_maps_12},\emph{
quantum} hyperbolic system with damping $D$ can be though as a classical
system with the potential written as $V=D+\frac{1}{2}J$ where $D$
is called the effective damping function.} 
\[
\gamma_{\mathrm{asympt.}}^{(quantum)}=\gamma_{conj}^{(quantum)}:=\frac{1}{2}\mathrm{Pr}\left(2\left(D-\frac{1}{2}J\right)\right)=\frac{1}{2}\mathrm{Pr}\left(2D-J\right)
\]
\begin{itemize}
\item In particular if we choose the potential $V=J$ (this choice is used
for counting periodic orbits) the conjecture is
\[
\gamma_{\mathrm{asympt.}}=\gamma_{conj}=\frac{1}{2}\mathrm{Pr}\left(0\right)=\frac{1}{2}h_{top},
\]
where $h_{top}=\mathrm{Pr}\left(0\right)$ is the topological entropy.
\item In particular for $D=0$ we have $\gamma_{conj}^{(quantum)}=\frac{1}{2}\mathrm{Pr}\left(-J\right)$
and for hyperbolic surfaces this gives $\gamma_{conj}^{(quantum)}=\frac{\delta-1}{2}$
where $\delta$ is the Hausdorff dimension of the limit set \cite{bortwick_book_07}.
This conjecture has been made in \cite{naud_jakobson2012} for convex
co-compact hyperbolic surfaces.
\end{itemize}
Some numerical observations are in favor of this conjecture, e.g.
Figure \ref{fig:spec_105} and Figure \ref{fig:Gauss_Numerical}.
With some other numerical observations the value $\frac{1}{2}\mathrm{Pr}\left(2\left(V-J\right)\right)$
describes rather the maximum of the distribution of concentration
of eigenvalues $\gamma_{max}$ and not $\gamma_{\mathrm{asympt.}}$
\cite[figure 2]{Lu_Sridhar_Zworski_03},\cite[figure 4]{Barkhofen_Weich_PSKZ13},\cite[figure 27]{Borthwick_14},\cite[Section 5.3]{Borthwick_Weich_14}.
One could conjecture that both coincide in the semiclassical limit
$\nu\rightarrow+\infty$ (and for generic hyperbolic systems), i.e.
that $\gamma_{max}=\gamma_{\mathrm{asympt.}}=\gamma_{conj}$.
\begin{verse}

\appendix
\bibliographystyle{alpha}
\bibliography{../../articles}

\newcommand{\etalchar}[1]{$^{#1}$}
\begin{thebibliography}{BWP{\etalchar{+}}13}

\bibitem[AB67]{atiyah_67}
M.~F. Atiyah and R.~Bott.
\newblock A {L}efschetz fixed point formula for elliptic complexes. {I}.
\newblock {\em Ann. of Math. (2)}, 86:374--407, 1967.

\bibitem[AFW13]{faure_arnoldi_tobias_13}
J.F. Arnoldi, F.Faure, and T.~Weich.
\newblock Asymptotic spectral gap and weyl law for {R}uelle resonances of open
  partially expanding maps.
\newblock {\em Ergodic Theory and Dynamical Systems,
  \href{https://fr.arxiv.org/abs/1302.3087}{link}}, pages 1--58, 2013.

\bibitem[Arn76]{arnold-mmmc}
V.I. Arnold.
\newblock {\em Les m\'ethodes math\'ematiques de la m\'ecanique classique}.
\newblock Ed. Mir. Moscou, 1976.

\bibitem[Arn88]{arnold-ed2}
V.I. Arnold.
\newblock {\em Geometrical methods in the theory of ordinary differential
  equations}.
\newblock Springer Verlag, 1988.

\bibitem[BD00]{debievre-00}
F.~Bonechi and S.~DeBi\`{e}vre.
\newblock Exponential mixing and ln(h) timescales in quantized hyperbolic maps
  on the torus.
\newblock {\em Comm. Math. Phys.}, 211:659--686, 2000.

\bibitem[BD16]{dyatlov_bourgain_2016}
Jean Bourgain and Semyon Dyatlov.
\newblock Spectral gaps without the pressure condition.
\newblock {\em arXiv preprint arXiv:1612.09040}, 2016.

\bibitem[BFW14]{faure_weich_barkhofen_2014}
S.~Barkhofen, F.~Faure, and T.~Weich.
\newblock Resonance chains in open systems, generalized zeta functions and
  clustering of the length spectrum.
\newblock {\em Nonlinearity. \href{http://arxiv.org/abs/1403.7771}{link}},
  27:1829, 2014.

\bibitem[BL07]{liverani_butterley_07}
O.~Butterley and C.~Liverani.
\newblock Smooth {A}nosov flows: correlation spectra and stability.
\newblock {\em J. Mod. Dyn.}, 1(2):301--322, 2007.

\bibitem[Bor07]{bortwick_book_07}
D.~Borthwick.
\newblock {\em Spectral theory of infinite-area hyperbolic surfaces}.
\newblock Birkhauser, 2007.

\bibitem[Bor14]{Borthwick_14}
D.~Borthwick.
\newblock Distribution of resonances for hyperbolic surfaces.
\newblock {\em Experimental Mathematics}, 23:25--45, 2014.

\bibitem[BS02]{brin-02}
M.~Brin and G.~Stuck.
\newblock {\em Introduction to Dynamical Systems}.
\newblock Cambridge University Press, 2002.

\bibitem[BW97]{weinstein_97}
S.~Bates and A.~Weinstein.
\newblock {\em Lectures on the Geometry of Quantization}, volume~8.
\newblock American Mathematical Soc., 1997.

\bibitem[BW14]{Borthwick_Weich_14}
D.~Borthwick and T.~Weich.
\newblock Symmetry reduction of holomorphic iterated function schemes and
  factorization of selberg zeta functions.
\newblock {\em arXiv preprint arXiv:1407.6134 (to appear J.Spectral Theory)},
  2014.

\bibitem[BWP{\etalchar{+}}13]{Barkhofen_Weich_PSKZ13}
S.~Barkhofen, T.~Weich, A.~Potzuweit, H-J. St{\"o}ckmann, U.~Kuhl, and
  M.~Zworski.
\newblock Experimental observation of the spectral gap in microwave n-disk
  systems.
\newblock {\em Physical review letters}, 110(16):164102, 2013.

\bibitem[DeL92]{delatte_92}
D.~DeLatte.
\newblock Nonstationnary normal forms and cocycle invariants.
\newblock {\em Random and Computational dynamics}, 1:229--259, 1992.

\bibitem[DeL95]{delatte_95}
D.~DeLatte.
\newblock On normal forms in hamiltonian dynamics, a new approach to some
  convergence questions.
\newblock {\em Ergod. Th. and Dynam. Sys.}, 15:49--66, 1995.

\bibitem[Dol98]{dolgopyat_98}
D.~Dolgopyat.
\newblock On decay of correlations in {A}nosov flows.
\newblock {\em Ann. of Math. (2)}, 147(2):357--390, 1998.

\bibitem[Dol02]{dolgopyat_02}
D.~Dolgopyat.
\newblock On mixing properties of compact group extensions of hyperbolic
  systems.
\newblock {\em Israel J. Math.}, 130:157--205, 2002.

\bibitem[DP98]{dolgopyat_addendum_98}
D.~Dolgopyat and M.~Pollicott.
\newblock Addendum to'periodic orbits and dynamical spectra'.
\newblock {\em Ergodic Theory and Dynamical Systems}, 18(2):293--301, 1998.

\bibitem[DZ15]{dyatlov_zahl_15}
S.~Dyatlov and J.~Zahl.
\newblock Spectral gaps, additive energy, and a fractal uncertainty principle.
\newblock {\em arXiv preprint arXiv:1504.06589}, 2015.

\bibitem[EN99]{engel_1999}
K.J. Engel and R.~Nagel.
\newblock {\em One-parameter semigroups for linear evolution equations}, volume
  194.
\newblock Springer, 1999.

\bibitem[Fal97]{Falconer_97}
K.~Falconer.
\newblock {\em Techniques in fractal geometry}.
\newblock John Wiley \& Sons Ltd., Chichester, 1997.

\bibitem[Fau]{faure_animations_expanding_maps}
F.~Faure.
\newblock Multimedia results for expanding maps.
\newblock {\em \href{http://www-fourier.ujf-grenoble.fr/\textasciitilde
  faure/articles}{link}}.

\bibitem[Fau07]{fred-trace-06}
F.~Faure.
\newblock Semiclassical formula beyond the {E}hrenfest time in quantum
  chaos.(i) trace formula.
\newblock {\em Annales de l'Institut Fourier, No.7.}, 57:2525--2599, 2007.

\bibitem[FND03]{fred-steph-02}
F.~Faure, S.~Nonnenmacher, and S.~DeBi\`evre.
\newblock Scarred eigenstates for quantum cat maps of minimal periods.
\newblock {\em Communications in Mathematical Physics}, 239:449--492, 2003.

\bibitem[FS11]{fred_flow_09}
F.~Faure and J.~Sj{\"o}strand.
\newblock Upper bound on the density of {R}uelle resonances for {A}nosov flows.
  a semiclassical approach.
\newblock {\em Comm. in Math. Physics, Issue 2.
  \href{https://fr.arxiv.org/abs/1003.0513}{link}}, 308:325--364, 2011.

\bibitem[FT13]{faure_tsujii_band_CRAS_2013}
F.~Faure and M.~Tsujii.
\newblock Band structure of the {R}uelle spectrum of contact {A}nosov flows.
\newblock {\em Comptes rendus - Math\'ematique 351 , 385-391, (2013)
  \href{https://arxiv.org/abs/1301.5525}{link}}, 2013.

\bibitem[FT15]{faure-tsujii_prequantum_maps_12}
F.~Faure and M.~Tsujii.
\newblock Prequantum transfer operator for symplectic {A}nosov diffeomorphism.
\newblock {\em Asterisque 375 (2015),
  \href{https://fr.arxiv.org/abs/1206.0282}{link}}, pages ix+222 pages, 2015.

\bibitem[FT16]{faure-tsujii_anosov_flows_13}
F.~Faure and M.~Tsujii.
\newblock The semiclassical zeta function for geodesic flows on negatively
  curved manifolds.
\newblock {\em Inventiones mathematicae.
  \href{https://arxiv.org/abs/1311.4932}{link}}, 2016.

\bibitem[FT17]{faure_tsujii_Ruelle_resonances_density_2016}
F.~Faure and M.~Tsujii.
\newblock Fractal {W}eyl law for the ruelle spectrum of {A}nosov flows.
\newblock {\em arXiv:1706.09307 \href{https://arxiv.org/abs/1706.09307}{link}},
  2017.

\bibitem[GHW16]{guillarmou_weich_resonances_16}
C.~Guillarmou, J.~Hilgert, and T.~Weich.
\newblock Classical and quantum resonances for hyperbolic surfaces.
\newblock {\em arXiv preprint arXiv:1605.08801}, 2016.

\bibitem[GLZ04]{zworski_lin_guillope_02}
L.~Guillope, K.~Lin, and M.~Zworski.
\newblock {The Selberg zeta function for convex co-compact. {S}chottky groups}.
\newblock {\em Comm. Math. Phys.}, 245(1):149--176, 2004.

\bibitem[Gui77]{guillemin_1977_lectures}
V.~Guillemin.
\newblock Lectures on spectral theory of elliptic operators.
\newblock {\em Duke Mathematical Journal}, 44(3):485--517, 1977.

\bibitem[Has02]{hasselblatt_02b}
B.~Hasselblatt.
\newblock Hyperbolic dynamical systems.
\newblock {\em Handbook of Dynamical Systems}, 1:239--319, 2002.

\bibitem[Hen92]{Hen92}
D.~Hensley.
\newblock Continued fraction {C}antor sets, {H}ausdorff dimension, and
  functional analysis.
\newblock {\em Journal of Number Theory}, 40(3):336--358, 1992.

\bibitem[HP70]{hirsch1970stable}
Morris~W Hirsch and Charles~C Pugh.
\newblock Stable manifolds and hyperbolic sets.
\newblock In {\em Global Analysis (Proc. Sympos. Pure Math., Vol. XIV,
  Berkeley, Calif., 1968)}, pages 133--163, 1970.

\bibitem[JN12]{naud_jakobson2012}
D.~Jakobson and F.~Naud.
\newblock On the critical line of convex co-compact hyperbolic surfaces.
\newblock {\em Geometric and Functional Analysis}, 22(2):352--368, 2012.

\bibitem[JP02]{Jenkinson_Pollicott_02}
O.~Jenkinson and M.~Pollicott.
\newblock Calculating {H}ausdorff dimension of {J}ulia sets and {K}leinian
  limit sets.
\newblock {\em American Journal of Mathematics}, 124(3):495--545, 2002.

\bibitem[KH95]{katok_hasselblatt}
A.~Katok and B.~Hasselblatt.
\newblock {\em Introduction to the Modern Theory of Dynamical Systems}.
\newblock Cambridge University Press, 1995.

\bibitem[Kif92]{kifer1992averaging}
Yuri Kifer.
\newblock Averaging in dynamical systems and large deviations.
\newblock {\em Inventiones mathematicae}, 110(1):337--370, 1992.

\bibitem[Kif94]{kifer1994large}
Yuri Kifer.
\newblock Large deviations, averaging and periodic orbits of dynamical systems.
\newblock {\em Communications in mathematical physics}, 162(1):33--46, 1994.

\bibitem[Leb96]{lebeau1996equation}
Gilles Lebeau.
\newblock Equation des ondes amorties.
\newblock In {\em Algebraic and geometric methods in mathematical physics},
  pages 73--109. Springer, 1996.

\bibitem[Liv04]{liverani_contact_04}
C.~Liverani.
\newblock On contact {A}nosov flows.
\newblock {\em Ann. of Math. (2)}, 159(3):1275--1312, 2004.

\bibitem[LSZ03]{Lu_Sridhar_Zworski_03}
W.~Lu, S.~Sridhar, and M.~Zworski.
\newblock {Fractal Weyl laws for chaotic open systems}.
\newblock {\em Physical review letters}, 91(15):154101, 2003.

\bibitem[MU99]{MU99}
R.~Mauldin and M.~Urba{\'n}ski.
\newblock Conformal iterated function systems with applications to the geometry
  of continued fractions.
\newblock {\em Transactions of the American Mathematical Society},
  351(12):4995--5025, 1999.

\bibitem[Nau05]{naud_expanding_2005}
F.~Naud.
\newblock Expanding maps on {C}antor sets and analytic continuation of zeta
  functions.
\newblock In {\em Annales Scientifiques de Ecole Normale Superieure},
  volume~38, pages 116--153. Elsevier, 2005.

\bibitem[Nel69]{nelson_69}
E.~Nelson.
\newblock {\em Topics in dynamics}, volume 969.
\newblock Princeton University Press, 1969.

\bibitem[NZ15]{nonenmacher_zworski_2013}
S.~Nonnenmacher and M.~Zworski.
\newblock Decay of correlations for normally hyperbolic trapping.
\newblock {\em Inventiones mathematicae}, 200(2):345--438, 2015.

\bibitem[Pes04]{pesin_04}
Y.~Pesin.
\newblock {\em Lectures on Partial Hyperbolicity and Stable Ergodicity}.
\newblock European Mathematical Society, 2004.

\bibitem[Pol95]{Pol95}
M.~Pollicott.
\newblock {Large deviations, Gibbs measures and closed orbits for hyperbolic
  flows.}
\newblock {\em {Math. Z.}}, 220(2):219--230, 1995.

\bibitem[PS96]{PS96}
M.~{Pollicott} and R.~{Sharp}.
\newblock {Large deviations and the distribution of pre-images of rational
  maps.}
\newblock {\em {Commun. Math. Phys.}}, 181(3):733--739, 1996.

\bibitem[Rue89]{ruelle_89}
David Ruelle.
\newblock The thermodynamic formalism for expanding maps.
\newblock {\em Communications in Mathematical Physics}, 125(2):239--262, 1989.

\bibitem[Sha92]{Sha92}
R.~Sharp.
\newblock {Prime orbit theorems with multi-dimensional constraints for Axiom A
  flows.}
\newblock {\em {Monatsh. Math.}}, 114(3-4):261--304, 1992.

\bibitem[Sj{\"o}00]{sjostrand_2000}
J.~Sj{\"o}strand.
\newblock Asymptotic distribution of eigenfrequencies for damped wave
  equations.
\newblock {\em Publ. Res. Inst. Math. Sci}, 36(5):573--611, 2000.

\bibitem[Tay74]{Taylor_book_PDO}
Michael Taylor.
\newblock {\em Pseudo differential operators}.
\newblock Lecture Notes in Mathematics, Vol. 416. Springer-Verlag, Berlin,
  1974.

\bibitem[Tay96]{taylor_tome1}
M.~Taylor.
\newblock {\em Partial differential equations, Vol I}.
\newblock Springer, 1996.

\bibitem[Tsu10]{tsujii_08}
M.~Tsujii.
\newblock Quasi-compactness of transfer operators for contact {A}nosov flows.
\newblock {\em Nonlinearity, arXiv:0806.0732v2 [math.DS]}, 23(7):1495--1545,
  2010.

\bibitem[Tsu12]{tsujii_FBI_10}
M.~Tsujii.
\newblock Contact {A}nosov flows and the fourier--bros--iagolnitzer transform.
\newblock {\em Ergodic theory and dynamical systems}, 32(06):2083--2118, 2012.

\bibitem[Tsu15]{15_tsujii}
M.~Tsujii.
\newblock The error term of the prime orbit theorem for expanding semiflows.
\newblock {\em arXiv preprint arXiv:1502.00422}, 2015.

\bibitem[Tsu16]{tsujii_2016_exponential_mixing}
M.~Tsujii.
\newblock Exponential mixing for generic volume-preserving {A}nosov flows in
  dimension three.
\newblock {\em arXiv preprint arXiv:1601.00063}, 2016.

\bibitem[Wei15]{wei14}
T.~Weich.
\newblock Resonance chains and geometric limits on {S}chottky surfaces.
\newblock {\em Communications in Mathematical Physics}, 337(2):727--765, 2015.

\bibitem[You90]{young1990large}
Lai-Sang Young.
\newblock Large deviations in dynamical systems.
\newblock {\em Transactions of the American Mathematical Society},
  318(2):525--543, 1990.

\bibitem[Zwo12]{zworski_book_2012}
M.~Zworski.
\newblock {\em Semiclassical Analysis}.
\newblock Graduate Studies in Mathematics Series. Amer Mathematical Society,
  2012.

\end{thebibliography}
\end{verse}

\end{document}